\documentclass[12pt,oneside,english]{amsart}
\usepackage{babel} 
\usepackage{graphicx}
\usepackage{amssymb}
\usepackage{amsthm}
\usepackage{amsmath}
\usepackage{color}
\usepackage{comment}
\usepackage{hyperref}
\usepackage{mathtools}
\usepackage{url}
\usepackage{tikz}
\usepackage{stix}
\usepackage{nicematrix}
\usepackage{esint}

\providecommand{\underrel}[2]{\mathrel{\mathop{#2}\limits_{#1}}}

\usepackage{geometry}
\newtheorem{theorem}{Theorem}
\newtheorem{definition}{Definition}
\newtheorem{lemma}{Lemma}
\newtheorem{proposition}{Proposition}

\newtheorem{corollary}{Corollary}

\newtheorem{remark}{Remark}

\renewcommand{\S}{\ensuremath{\text{Sig}_d}}
\newcommand{\F}{\ensuremath{\mathscr{F}}}
\newcommand{\J}{\ensuremath{\mathscr{J}}}

\def\Xint#1{\mathchoice
{\XXint\displaystyle\textstyle{#1}}%
{\XXint\textstyle\scriptstyle{#1}}%
{\XXint\scriptstyle\scriptscriptstyle{#1}}%
{\XXint\scriptscriptstyle\scriptscriptstyle{#1}}%
\!\int}
\def\XXint#1#2#3{{\setbox0=\hbox{$#1{#2#3}{\int}$}
\vcenter{\hbox{$#2#3$}}\kern-.5\wd0}}

\def\fint{\Xint-}

\providecommand{\underrel}[2]{\mathrel{\mathop{#2}\limits_{#1}}}
\providecommand{\esc}[2]{\left\langle #1,#2\right\rangle}
\providecommand{\rn}{\mathbb{R}^n}

\newcommand{\norm}[1]{\left\lVert#1\right\rVert}

\providecommand{\integ}[1]{\langle #1\rangle}
\providecommand{\iinteg}[1]{\langle \langle #1\rangle\rangle}

\DeclareMathOperator{\Cn}{{\mathbb{C}^n}}
\DeclarePairedDelimiter{\tnorm}{\lvert\lvert\lvert}{\rvert\rvert\rvert}

\allowdisplaybreaks 

\author{Joshua Isralowitz}
\address{(J.I.) Department of Mathematics and Statistics, University at Albany (New York, USA) }
\email{jisralowitz@albany.edu}
\author{
  Israel P. Rivera-Ríos
}
\address{(I. P. R.-R.) Departamento de Análisis Matemático, Estadística
e Investigación Operativa y Matemática Aplicada. Facultad de Ciencias.
Universidad de Málaga (Málaga, Spain).}
\email{israelpriverarios@uma.es}
\author{
  Francisco Sáez-Rivas 
}

\address{(F. S. R.) Departamento de Análisis Matemático, Estadística
e Investigación Operativa y Matemática Aplicada. Facultad de Ciencias.
Universidad de Málaga (Málaga, Spain).}
\email{fsr034@uma.es}

\title[Convex body dom. of the vector valued multi-symbol commutator]{Convex body domination for the commutator of 
 vector-valued operators with multiple matrix-valued symbols}

\thanks{Joshua Isralowitz was supported by a grant from the Simons Foundation International [SFI-MPS-TSM-00014292, JI].
Francisco Sáez-Rivas was supported by an FPU fellowship (FPU23/02107), from
 the Spanish Ministry of Science and Innovation.
 The second and third authors were partially supported by the Spanish Ministry of Science and Innovation
through the project PID2022-136619NB-I00 funded by MCIN/AEI/10.13039/501100011033/FEDER, UE and by Junta de Andalucía through the project PPRO-FQM354-G-2023 (FQM354-G-FEDER) }
\begin{document}


\maketitle


\bigskip

\begin{abstract}
We provide convex body domination results for the generalized vector-valued
 commutator of those operators
that admit specific forms of convex body domination themselves. We also prove some 
strong-type estimates and other consequences of these results, and we study the 
$BMO$ spaces that appear naturally in this context.
\end{abstract}
\section{Introduction}
We recall that a weight, namely, a nonnegative locally integrable
function $w$ belongs to $A_{p}$ for $1<p<\infty$ if 
\[
[w]_{A_{p}}=\sup_{Q}\fint_{Q}w(x)dx\left(\fint_{Q}w(y)^{-\frac{p'}{p}}dy\right)^{\frac{p}{p'}}<\infty.
\]
The $A_{p}$ class of weights characterizes the $L^{p}(w)$ boundedness
of the maximal function as B. Muckenhoupt established in the 70s.
Since Muckenhoupt's seminal work, a number of authors have devoted
a number of works to studying the connection of $A_{p}$ weights with
other operators such as singular integrals. 

In the last fifteen years the theory of weights has been a fruitful
field due to the study of quantitative estimates. Namely the study
of inequalities tracking the precise dependence on the so called $A_{p}$
constant, namely $[w]_{A_{p}}$. The paradigmatic question in that
line of research was the $A_{2}$ theorem finally established by T.
Hytönen \cite{H}, which in turn motivated the development of the
sparse domination theory which began with A. Lerner in \cite{L}. We will return to sparse domination, but in the vector valued setting, later, since it is a fundamental part of this work.

One of the possible extensions of the classical scalar theory of Calderón-Zygmund
operators, vector valued extensions, have been paid a lot of attention
in the last years. Let $W:\mathbb{R}^{d}\rightarrow\mathbb{C}^{n\times n}$
a matrix weight, namely, a matrix function such that $W(x)$ is positive
definite a.e. Given $\vec{f}:\mathbb{R}^{d}\rightarrow\mathbb{C}^{n}$ and
$1<p<\infty$, we define 
\[
\|\vec{f}\|_{L^{p}(W)}=\left(\int_{\mathbb{R}^{d}}\left|W^{\frac{1}{p}}(x)\vec{f}(x)\right|^{p}dx\right)^{\frac{1}{p}}.
\]
Let $1<p<\infty$. We say that a matrix weight $W$ is an $A_{p}$
weight if 
\[
[W]_{A_{p}}=\sup_{Q}\fint_{Q}\left(\fint_{Q}\left\Vert W^{\frac{1}{p}}(x)W^{-\frac{1}{p}}(y)\right\Vert ^{p'}dy\right)^{\frac{p}{p'}}dx<\infty.
\]

The first appearance of $A_{p}$ weights is due to S. Treil and A.
Volberg in \cite{TV}. In the late 90s M. Goldberg \cite{G}, F. Nazarov
and S. Treil \cite{NT}, and A. Volberg \cite{Volberg-S}, showed that if
$W$ is a matrix $A_{p}$ weight then certain classes of Calderón-Zygumund
operators acting componentwise on $\mathbb{C}^{n}$ valued functions
are bounded on $L^{p}(W)$. The presentation of the matrix $A_{p}$
class that we provided a few lines above is due to S. Roudenko \cite{roudenko}
and is equivalent to the definitions used in the aforementioned works.

Contrary to what happpens in the scalar setting, the $A_{2}$ conjecture
is not true. Let us devote a few words to this. A few years ago, F.
Nazarov, S. Petermichl, S. Treil and A. Volberg \cite{cbdwczo} established
the following quantitative estimate for $W\in A_{2}$, 
\begin{equation}
\|T\vec{f}\|_{L^{2}(W)}\leq c_{n,d,T}[W]_{A_{2}}^{\frac{3}{2}}\|f\|_{L^{2}(W)}\label{eq:A2NPTV}
\end{equation}
where $T$ is a Calderón-Zygmund operator. Very recently K. Domelevo,
S. Petermichl, S. Treil and A. Volberg \cite{DPTV} proved that \eqref{eq:A2NPTV}
is actually sharp. For $p\not=2$ the current record is due to D.
Cruz-Uribe, J. Isralowitz and K. Moen \cite{CUIM} who extended (\ref{eq:A2NPTV})
to every $1<p<\infty$, providing the following estimate 
\[
\|T\vec{f}\|_{L^{p}(W)}\leq c_{n,d,T}\left[W\right]_{A_{p}}^{1+\frac{1}{p-1}-\frac{1}{p}}\|\vec{f}\|_{L^{p}(W)}.
\]
It is not known whether this estimate is sharp for $p\not=2$. At this point it is worth recalling that assuming a stronger condition, namely, $W\in A_{q}$ with $1\leq q<p$, the dependencies on the $A_q$ constant from the scalar setting hold in this setting as well \cite{IPRR}, a fact that highlights even more the fact that \eqref{eq:A2NPTV} is sharp.

Besides the aforementioned results, a number of works have been devoted
to the study of matrix weighted estimates. In the following lines
we review some of them. We apologise in advance for any omission that the reader may consider inconvenient. Quantitative extrapolation and factorization
of $A_{p}$ weights were achieved by Bownik and Cruz-Uribe in \cite{BCU}. This
result was extended in several directions \cite{N,NP,CUS}. A number
of extensions have also been provided, such as  having
spaces of homogeneous type as departure spaces \cite{KNV,claro2025matrixweightedestimatesspaces}, \cite{DKPS} to the multiparametric
setting, or very recently \cite{KN} to the multilinear setting.
Endpoint estimates for the Christ-Goldberg maximal function,
Calderón-Zygmund operators and commutators were provided in  \cite{CUIMPRR} and shown
to be sharp in \cite{LeLiORR} in the case of the maximal function and Calderón-Zygmund operators.

\section{Main Results}
\subsection{Convex Body Domination for higher order commutators}
Before continuing our exposition and presenting our main results we need to fix some notation. For a matrix or vector function $F$, $r\in[1,\infty)$, and a finite measure set $Q$, we will denote by
$\integ{F}_{r,Q}$ the matrix or vector whose components are obtained by applying the norm in $L^r\left(\frac{dx}{|Q|}\right)$ over the ones of $F$. 
We will also employ the notation $\integ{F}_{Q}$ or $\fint_Q F$ to refer to the matrix or vector whose components are the integral averages of those of $F$.

For a scalar valued operator $T$ defined for scalar valued functions, 
$T\otimes I_n $ is the operator such that, for $\vec{g} = (g_1,...,g_n)$,
\begin{equation*}
\begin{split}
    (T\otimes I_n)(\vec{g}) = (T(g_1),...,T(g_n)).
\end{split}
\end{equation*}

Recall as well that the commutator of an operator $\vec{T}$, which maps $n-$ dimensional vector functions in an appropriate 
space into another space of $n$-dimensional vector functions,
 with respect to the $n\times n$ matrix function $B$ (usually referred to as the symbol) 
is defined as 
\begin{equation*}
\begin{split}
    [B,\vec{T}]\vec{f}(x) = B(x)\vec{T}\vec{f} (x)- \vec{T}(B\vec{f})(x).
\end{split}
\end{equation*}
For a vector $\vec{B}=(B_1,...,B_m)$ of $n\times n$ matrices, we are interested in the operator 
\begin{equation*}
    \begin{split}
        \vec{T}_{\vec{B}}\vec{f}(x) = [B_m,[B_{m-1},...,[B_1,\vec{T}]...]]\vec{f}(x).
    \end{split}
    \end{equation*}
     We will refer to this operator as the (generalized) commutator of $\vec{T}$ with respect to 
    the multi-symbol $\vec{B}$.

In \cite{cbdwczo}, the authors introduced an extension of sparse domination called 
convex body domination, which is very useful when working with operators in 
a vector setting. In this form of domination, operators are controlled pointwise by 
a convex body valued sparse operator, given by a sum,
generally over cubes in a sparse family, of symmetric, convex and compact sets indexed
in said cubes.
Precisely, for a measurable set $\Omega\subset \mathbb{R}^d$ with finite measure, $r\in[1,\infty)$,
    and $f\in L^r(\Omega,\mathbb{R}^n)$, the aforementioned sets are given by
    \begin{equation*}
    \begin{split}
        {\iinteg{\vec{f}}_{r,\Omega}} = \left\{ \integ{\phi \vec{f}}_\Omega:\;\phi:\Omega\rightarrow \mathbb{F}, \|\phi\|_{L^{r'}\left(\frac{dx}{|\Omega|}\right)}\leq 1 \right\},
    \end{split}
    \end{equation*}
    and, for a sparse family $\mathcal{S}$, the corresponding operator is given by
    \begin{equation*}
    \begin{split}
        L_{\mathcal{S},r}\vec{f} = \sum_{Q\in \mathcal{S}}\chi_Q\iinteg{\vec{f}}_{r,Q},
    \end{split}
    \end{equation*}
    so that an operator that admits convex body domination in its most basic form will satisfy that
    \begin{equation*}
    \begin{split}
        \vec{T}\vec{f}(x) \in C L_{\mathcal{S}}\vec{f}(x)
    \end{split}
    \end{equation*}
    a.e. on some set.
The symmetry, convexity and compactness of the sets $\iinteg{\vec{f}}_{r,\Omega}$,
for $r\in[1,\infty)$, is not hard to prove (it can be found in 
\cite{pliniohytoli}, for example).

Throughout this article, we will make some convenient modifications to this definition,
to adapt it to the operators we are interested in. The main one will be based on the 
following lemma.
\begin{lemma}\label{lemma:cbdexprbonita}
    Consider $r\in [1,\infty)$, $Q\subset \mathbb{R}^d$ with finite measure, $\vec{f}\in L^r(Q,\mathbb{F}^n)$ and $\vec{g}:\mathbb{R}^d\rightarrow \mathbb{F}^n$
    such that $g(x)\in \iinteg{\vec{f}}_{r,Q}$ a.e. on $\mathbb{R}^d$. Then there exists 
    $K:\mathbb{R}^d\times Q\rightarrow \mathbb{F}$ with 
    $\|K(x,\cdot)\|_{L^{r'}\left(\frac{dy}{|Q|}\right)}\leq 1$ 
     for a.e. $x\in \mathbb{R}^d$
     satisfying that, for a.e. $x\in \mathbb{R}^d$, 
    \begin{equation*}
    \begin{split}
        \vec{g}(x) = \integ{K(x,\cdot)\vec{f}}_Q.
    \end{split}
    \end{equation*}
\end{lemma}
The proof of this result can be found in \cite{cbdwczo,hytonotas,claro2025matrixweightedestimatesspaces}.
Using this lemma, if we have that
\begin{equation*}
\begin{split}
    \vec{T}\vec{f}(x)\in C L_{\mathcal{S},r}\vec{f}(x)
\end{split}
\end{equation*}
for a.e. $x\in \mathbb{R}^d$, we automatically deduce that 
there exist
functions $k_Q:\mathbb{R}^d\times Q\rightarrow \mathbb{F}$ with 
    $\|k_Q(x,\cdot)\|_{L^{r'}\left(\frac{dy}{|Q|}\right)}\leq 1$ 
    for a.e. $x\in \mathbb{R}^d$
such that
\begin{equation*}
\begin{split}
    \vec{T}\vec{f}(x) = C \sum_{Q\in \mathcal{S}}\chi_Q \integ{k_Q(x,\cdot)\vec{f}}_Q
\end{split}
\end{equation*}
for a.e. $x\in \mathbb{R}^d$.


In \cite{SA1EfVVO}, a convex body domination result for
the commutator of certain vector-valued operators with one scalar symbol was established, by means of an argument that involves
multiplying some well-chosen matrices appropriately.
This reasoning was later adapted in \cite{Cit2SaMWS} to cover the case of matrix symbols.
In this paper, we will further extend these results to commutators of a broader 
class of operators with an arbitrary amount of matrix symbols, employing a refinement of the 
argument that was used in the aforementioned papers. Precisely, we obtain a 
convex body domination result for the commutators of operators which works as long as the 
operators themselves admit a slightly stronger version of convex body domination.

\begin{theorem}\label{th:tocho}
    Let $\Omega$ be a measurable subset of $\mathbb{R}^d$, 
     $\mathbb{F}=\mathbb{R}$ or $\mathbb{C}$,
    $r\in [1,\infty]$, 
    and
    assume $\vec{T}=(T_1,...,T_n)$ is a {linear} operator 
    with the following property:
    \begin{itemize}
        \item[$(P_1)$] For every $\vec{f}\in L_c^r(\Omega,\mathbb{F}^{2^mn})$,
       there exist $C>0$,
        a collection of subsets of $\Omega$, $\mathcal{G}$ and
        functions $K_Q:Q\times Q\rightarrow \mathcal{M}_n(\mathbb{F})$, for $Q\in \mathcal{G}$, 
        with 
        $\|K_Q(x,\cdot)\|_{L^{r'}\left(\frac{dy}{|Q|}\right)}\leq 1$ for a.e $x\in Q$
        such that, for the $2^mn\times 2^m n$ matrix function, given by the $n\times n$ blocks
        \begin{equation*}
        \begin{split}
            \overline{K}_Q = 
\begin{pNiceArray}{cc|cc|cc|cc}
  \quad K_Q &&  {{0}} && \dots && {0}\\
  \hline
  {{0}} && K_Q&& \dots && {0}\\
  \hline
  \vdots && \vdots && \ddots     && \vdots\\
  \hline
  0 && 0 && \dots &&K_Q
\end{pNiceArray}
        \end{split}
        \end{equation*} 
        and the operator $\overline{T}$ given by 
            \begin{equation}\label{Tbarra}
        \begin{split}
            \overline{T}(\vec{g}) = \begin{pmatrix}
            \vec{T}(\vec{g}^1)\\\vec{T}(\vec{g}^2)\\\vdots\\\vec{T}(\vec{g}^{2^m})
            \end{pmatrix},
        \end{split}
        \end{equation}
        for an appropriate $\vec{g} = (g_1,g_2,...,g_{2^mn})$
        ($\vec{g}^k=(g_{(k-1)n+1},g_{(k-1)n+2},...,g_{(k-1)n+n})$),
        the equality
        \begin{equation}\label{eq:pconvT}
        \begin{split}
            \overline{T}\vec{f}(x) = C \sum_{Q\in \mathcal{G}}\chi_Q(x)\integ{\overline{K}_Q(x,\cdot)\vec{f}}_Q
        \end{split}
        \end{equation}
        is satisfied a.e. on $\Omega$.
    \end{itemize}
    
    Then, given any function $\vec{f}\in L_c^r(\Omega,\mathbb{F}^n)$,  
    and $\vec{B}=(B_1,...,B_m)$, a vector with $m$ $n\times n$ matrix functions defined in $\Omega$, 
    such that $\vec{B}_{\sigma^t}\vec{f}\in L_c^r(\Omega,\mathbb{F}^n)$
    for all $\sigma\in C(m)$,
     there exist $C>0$,
    a collection of subsets of $\Omega$, $\mathcal{G}$, 
    and functions $K_Q:Q\times Q\rightarrow \mathcal{M}_n(\mathbb{F})$
    with $\|K_Q(x,\cdot)\|_{L^{r'}\left(\frac{dy}{|Q|}\right)}\leq 1$ a.e. $x\in Q$, for $Q\in \mathcal{G}$ (all provided by 
    $(P_1)$ applied to a function dependent on $\vec{f}$ and $\vec{B}$), such that
    \begin{equation}\label{paux33}
    \begin{split}
        \vec{T}_{\vec{B}}\vec{f}(x)=C \sum_{Q\in \mathcal{G}}\chi_Q(x)\sum_{\sigma\in C(m)}(-1)^{m-|\sigma|}  \vec{B}_{\sigma}(x)\integ{K_Q(x,\cdot)\vec{B}_{(\sigma^c)^t}\vec{f}}_Q
    \end{split}
    \end{equation}
    for a.e. $x\in \Omega$.
\end{theorem}
Here, {$C(m)$ is the set of all increasingly ordered tuples composed of elements of 
$\{1,2,...,m\}$}, and $\vec{B}_{\sigma}$, for $\sigma\in C(m)$, is the product of the 
components of $\vec{B}$ that appear in $\sigma$, in reverse order.
All the notation regarding these tuples will be explained with greater detail in Section \ref{sect:combinatorics}.

When $m=0$, $(P_1)$ is the natural convex body domination condition for operators of the form $\vec{T}$, as it 
takes into consideration
all the components that appear in such operators. 
In Section \ref{sect:apps}, we will show how sparse domination proofs
can be adapted to obtain 
conditions like $(P_1)$ easily.

Let us make some observations on the theorem. If we define $\vec{T}_{Q}\vec{g}(x)=\integ{K_{Q}(x,\cdot)\vec{g}}_{Q}$,
(\ref{paux33}) can be rewritten as 
\[
\begin{split}\vec{T}_{\vec{B}}\vec{f}(x)=C\sum_{Q\in\mathcal{G}}\chi_{Q}(x)(\vec{T}_{Q})_{\vec{B}}\vec{f}(x),\end{split}
\]
so that the commutator of $\vec{T}$ is controlled by a sum of commutators
of integral operators. This can be exploited in many ways. For example,
if for each $Q$, we take $\vec{a}_{Q}=(a_{Q}^{1},...,a_{Q}^{m})\in\mathbb{F}^{m}$,
and $\vec{A}_{Q}=(a_{Q}^{1}I_{n},...,a_{Q}^{m}I_{n})$, the linearity
of the commutator clearly implies that 
\[
(\vec{T}_{Q})_{\vec{B}}\vec{f}=(\vec{T}_{Q})_{\vec{B}-\vec{A}_{Q}}\vec{f}
\]
so that 
\[
\begin{split}\vec{T}_{\vec{B}}\vec{f}(x)=C\sum_{Q\in\mathcal{G}}\chi_{Q}(x)\sum_{\sigma\in C(m)}(-1)^{m-|\sigma|}(\vec{B}-\vec{A_{Q}})_{\sigma}(x)\integ{K_{Q}(x,\cdot)(\vec{B}-\vec{A_{Q}})_{(\sigma^{c})^{t}}\vec{f}}_{Q}.\end{split}
\]
This is particularly interesting if we consider $\vec{b}=(b_{1},...,b_{m})$
a vector of functions defined in $\Omega$ such that $\vec{b}_{\sigma}\vec{f}\in L_{c}^{1}(\Omega,\mathbb{F}^{n})$
for all $\sigma\in\bigcup_{j=0}^{n}C_{j}(m)$, and $\vec{B}=(b_{1}I_{n},...,b_{m}I_{n})$.
Here we could take $\vec{a}_{Q}=\integ{\vec{b}}_{Q}$, and obtain
that 
\[
\begin{split}\vec{T}_{\vec{b}}\vec{f}(x) & =C\sum_{Q\in\mathcal{G}}\chi_{Q}(x)(\vec{T}_{Q})_{\vec{b}-\integ{\vec{b}}_{Q}}\vec{f}(x)\\
 & =C\sum_{Q\in\mathcal{G}}\chi_{Q}(x)\sum_{\sigma\in C(m)}(-1)^{m-|\sigma|}(\vec{b}-\integ{\vec{b}}_{Q})_{\sigma}(x)\integ{K_{Q}(x,\cdot)(\vec{b}-\integ{\vec{b}}_{Q})_{\sigma^{c}}\vec{f}}_{Q},
\end{split}
\]
which extends the sparse domination result obtained in \cite{MR4734976}. \begin{corollary} Let $\Omega$
be a measurable subset of $\mathbb{R}^{d}$, $\mathbb{F}=\mathbb{R}$
or $\mathbb{C}$, $r\in[1,\infty]$ and assume $\vec{T}=(T_{1},...,T_{n})$
is a linear operator as in Theorem \ref{th:tocho}. Then, given any
function $\vec{f}\in L_{c}^{r}(\Omega,\mathbb{F}^{n})$, and $\vec{b}=(b_{1},...,b_{m})$,
a vector of functions defined in $\Omega$, such that $\vec{b}_{\sigma}\vec{f}\in L_{c}^{r}(\Omega,\mathbb{F}^{n})$
for all $\sigma\in C(m)$, 
 there exist $C>0$, a collection of subsets of $\Omega$, $\mathcal{G}$,
and functions $K_{Q}:Q\times Q\rightarrow\mathcal{M}_{n}(\mathbb{F})$
with $\|K_{Q}(x,\cdot)\|_{L^{r'}\left(\frac{dy}{|Q|}\right)}\leq1$
a.e. $x\in Q$, for $Q\in\mathcal{G}$ (all provided by $(P_{1})$
applied to a function dependent on $\vec{f}$ and $\vec{b}$), such
that 
\[
\begin{split}\vec{T}_{\vec{b}}\vec{f}(x)=C\sum_{Q\in\mathcal{G}}\chi_{Q}(x)\sum_{\sigma\in C(m)}(-1)^{m-|\sigma|}(\vec{b}-\integ{\vec{b}}_{Q})_{\sigma}(x)\integ{K_{Q}(x,\cdot)(\vec{b}-\integ{\vec{b}}_{Q})_{\sigma^{c}}\vec{f}}_{Q}\end{split}
\]
for a.e. $x\in\Omega$. \end{corollary}

Note that an analogous result for a matrix $B$ does not necessarily hold
 even if $\vec{T}$ 
is linear over matrices, due to the non-commutativity of matrices.

It should also be observed that, when we state in the theorem that $C$, $\mathcal{G}$ and
 the $K_Q$ are provided by 
    $(P_1)$ applied to a function dependent on $\vec{f}$ and $\vec{B}$, what we mean is that they 
    retain any property independent of those functions (this can be seen in the proof).
    For example, 
    if an operator satisfies $(P_1)$ with $C$ independent 
       of the function $\vec{f}$, the $C$ provided by the theorem will also be independent of $\vec{f}$ and $\vec{B}$.

    More interestingly, if the operator $\vec{T}$ is such that the corresponding
     $K_Q$ in $(P_1)$ satisfy some relationship that does not depend  
     on $\vec{f}$, like if
    \begin{equation*}
    \begin{split}
        K_Q = 
        k_QI_n
    \end{split}
    \end{equation*}
    for some $k_Q:Q\times Q\rightarrow \mathbb{F}$ with $\|k_Q\|_{L^{r'}}\leq1$ and for all $Q$ (which happens when $\vec{T} = T\otimes I_n$ for some $T$),
    then the 
    $K_Q$ provided by the theorem take the form
    \begin{equation*}
    \begin{split}
        K_Q = k_QI_n,
    \end{split}
    \end{equation*}
    with $k_Q$ in the same conditions as before.
This allows us to recover a more direct extension of Theorem 1.8 in \cite{Cit2SaMWS}.

\begin{corollary}\label{cor:Tochodiago}
        Let $\Omega$ be a measurable subset of $\mathbb{R}^d$, 
        $\mathbb{F}=\mathbb{R}$ or $\mathbb{C}$, $r\in [1,\infty]$, and
        assume $T$ is a scalar linear operator defined for scalar functions with the following property:
        \begin{itemize}
            \item[$(P_2)$] For every $\vec{f}\in L_c^r(\Omega,\mathbb{F}^{2^m n})$,
            there exist $C>0$,
            a collection of subsets of $\Omega$, $\mathcal{G}$, and 
            functions $k_Q:Q\times Q\rightarrow \mathbb{F}$, for $Q\in \mathcal{G}$, 
            with $\|k_Q(x,\cdot)\|_{L^{r'}\left(\frac{dy}{|Q|}\right)}\leq 1$ a.e. $x\in Q$,
            such that
            \begin{equation}\label{eq:convT}
            \begin{split}
                T\otimes I_{2^mn}\vec{f}(x) = C \sum_{Q\in \mathcal{G}}\integ{k_Q(x,\cdot)\vec{f}}_Q \chi_Q(x)
            \end{split}
            \end{equation}
            a.e. on $\Omega$.
        \end{itemize}
        Then, given any function $\vec{f}\in L_c^r(\Omega,\mathbb{F}^n)$, 
        and $\vec{B}=(B_1,...,B_m)$, a vector with $m$ $n\times n$ matrix functions defined in $\Omega$, 
        such that $\vec{B}_{\sigma^t}\vec{f}\in L_c^r(\Omega,\mathbb{F}^n)$ for all $\sigma\in C(m)$,
         there exist $C>0$,
        a collection of subsets of $\Omega$, $\mathcal{G}$, 
        and functions $k_Q:Q\times Q\rightarrow \mathbb{F}$ with $\|k(x,\cdot)\|_{L^{r'}\left(\frac{dy}{|Q|}\right)}\leq 1$, for $Q\in \mathcal{G}$ (all provided by 
        $(P_1)$ applied to a function dependent on $\vec{f}$ and $\vec{B}$), such that
        \begin{equation}\label{aux6}
            \begin{split}
                (T\otimes I_n)_{\vec{B}}\vec{f}(x)=C \sum_{Q\in \mathcal{G}}\chi_Q(x) \sum_{\sigma\in C(m)}(-1)^{m-|\sigma|} \vec{B}_{\sigma}(x)\integ{k_Q(x,\cdot)\vec{B}_{(\sigma^c)^t}\vec{f}}_Q
            \end{split}
        \end{equation}
        a.e. on $\Omega$.
\end{corollary}

It is known that many types of operators satisfy 
property $(P_2)$ for the previous corollary. 
For example, for linear operators $T$ such that 
$T$ and $M_T$ are of weak type $(1,1)$
(Theorem 3.4 in \cite{cbdwczo}) or for 
the Bergman projection in certain domains
(Theorem 3.2 in \cite{WEotBPwMW}). 
Here, $M_T$ is the maximal operator
\begin{equation*}
\begin{split}
    {M}_Tf(x) = \sup_{Q\ni x} \sup_{y\in Q}|T(f \chi_{\rn\setminus 3Q})(y)|.
\end{split}
\end{equation*}
Note that the same reasoning can be applied to
 similar variations of property $(P_1)$ to obtain 
 analogous results.

Some vector-valued operators may not admit convex body domination, like, for example,
rough singular integral operators. For such operators, 
one may attempt to obtain a weaker integral version of convex body domination, that 
generally is good enough to obtain interesting consequences. One such result is the 
following theorem from \cite{muller-rr}.

\begin{theorem}
    Let $1\leq q \leq r$,
     $s\geq 1$, $v$ such that $\frac{1}{v}=\frac{1}{r}+\frac{1}{s}$,
    and ${T}$ be a linear operator of weak type $(q,q)$ such that ${M}_{T}$ maps $L^r\times L^s$ into $L^{v,\infty}$ boundedly.
     Then, for any $\vec{f}$ with compact support and $n$ components such that 
     $|\vec{f}|\in L^r(\mathbb{R}^d)$ and $\vec{g}$ with $n$ components and $|\vec{g}|\in L_{Loc}^s(\mathbb{R}^d)$,
    there exists a $\frac{1}{2\cdot3^n}-$sparse family $\mathcal{S}$ such that 
    \begin{equation*}
    \begin{split}
        \int_{\mathbb{R}^d}&\left\lvert  \esc{{T}\otimes I_n\vec{f}}{\vec{g}}\right\rvert
         \leq     C_{n,d,\nu,q,\varepsilon}
        \left(\|T\|_{L^q\rightarrow L^{q,\infty}}  + \|M_T\|_{L^r\times L^s\rightarrow L^{\nu,\infty}}\right)
        \sum_{Q\in \mathcal{S}}|Q| \iinteg{\vec{f}}_{r,Q_0}\iinteg{\vec{g}}_{s,Q_0}.
    \end{split}
    \end{equation*}
\end{theorem}

Here,
\begin{equation*}
\begin{split}
    {M}_T(f,g)(x) = \sup_{Q\ni x}\integ{|T(f \chi_{\rn\setminus 3Q})g|}_Q.
\end{split}
\end{equation*}

In order to account for this case of operators, we will also provide an integral convex body  
domination result which works for the commutator of an operator, whenever said operator 
admits a form of integral convex body domination. In this case we just managed to provide a suitable result for operators of the form $T\otimes I_n$, extending \cite[Proposition 7.1]{La}

\begin{theorem}\label{th:tocho3}
    Let $\Omega$ be a measurable subset of $\mathbb{R}^d$,
     $\mathbb{F}=\mathbb{R}$ or $\mathbb{C}$,
    $r,s\in[1,\infty]$, and
    assume ${T}$ is a scalar linear operator with the following property:
    \begin{itemize}
        \item[$(P_2)$] For every $\vec{f}\in L_c^r(\Omega,\mathbb{F}^{2^mn})$
         and $\vec{g}\in L_{Loc}^s(\Omega,\mathbb{F}^{2^mn})$, 
        there exist $C>0$ and
        a collection of subsets of $\Omega$, $\mathcal{G}$ such that
        \begin{equation*}
        \begin{split}
            \int_{\mathbb{R}^d} \left\lvert \esc{{T\otimes I_{2^mn}}\vec{f}}{\vec{g}}\right\rvert 
        &\leq C \sum_{Q\in \mathcal{G}}|Q| 
        \iinteg{\vec{f}}_{r,Q}\iinteg{\vec{g}}_{s,Q}
        \end{split}
        \end{equation*}
    \end{itemize}
   
    Then, given functions $\vec{f}\in L_c^r(\Omega,\mathbb{F}^n)$, $\vec{g}\in L_{Loc}^s(\Omega,\mathbb{F}^n)$ 
    and $\vec{B}=(B_1,...,B_m)$, a vector with $m$ $n\times n$ matrix functions defined in $\Omega$, 
    such that $\vec{B}_{\sigma^t}\vec{f}\in L_c^r(\Omega,\mathbb{F}^n)$ and 
    $\vec{B}_{\sigma}^*\vec{g}\in L_{Loc}^s(\Omega,\mathbb{F}^n)$ for all $\sigma\in C(m)$,
    there exist $C>0$,
    a collection of subsets of $\Omega$, $\mathcal{G}$, 
    (all provided by 
    $(P_2)$ applied to a function dependent on $\vec{f}$, $\vec{g}$ and $\vec{B}$), such that
    \begin{equation*}
    \begin{split}
        \int_{\Omega}\left\lvert \esc{{(T\otimes I_n)}_{\vec{B}}\vec{f}(x)}{\vec{g}(x)}\right\rvert dx
    \leq C\sum_{Q\in \mathcal{G}}|Q| 
    \iinteg{\tilde{\Psi}\vec{f}}_{r,Q}\iinteg{\Psi^*\vec{g}}_{s,Q}.
    \end{split}
    \end{equation*}
\end{theorem}

Here, for a vector of $m$ matrices $\vec{B}$, we use the matrices
\begin{equation}\label{psi}
    \begin{split}
        \Psi(x) = \Psi_{\vec{B}}(x) = \begin{pNiceArray}{cc|cc|cc|cc|cc}
            (-1)^m I_n &&  \dots && (-1)^{m-k_0(j)}\vec{B}_{\tilde{\sigma}_j}(x) && \dots && \vec{B}_{\tilde{\sigma}_{2^m}}(x)
          \end{pNiceArray}
    \end{split}
    \end{equation}
    and
    \begin{equation}\label{tildepsi}
    \begin{split}
        \tilde{\Psi}(x)=\tilde{\Psi}_{\vec{B}}(x)=\begin{pmatrix}
              \vec{B}_{\tilde{\sigma}_{2^m}^t} (x)  \\
            \hline \\
            \vdots \\
            \hline  \\
            \vec{B}_{(\tilde{\sigma}_j^c)^t}(x)\\
            \hline \\
            \vdots\\
            \hline\\
            I_n
        \end{pmatrix},
    \end{split}
    \end{equation}
    where $\left\{ \tilde{\sigma}_j\right\}_{j=1}^{2^m}$ are the tuples of $C(m)$, but written in an 
    orderly manner (see Section \ref{sect:combinatorics} for the precise definition).
    When the index $\vec{B}$ is clear, it will be omitted.




Later, in Section \ref{sect:str}, we will apply
a lemma
 to deduce alternate versions of the previous theorem
  and 
 Corollary \ref{cor:Tochodiago},
  where the domination is written more conveniently.

\subsection{Quantitative weighted estimates}
Convex body domination results, much like sparse domination ones, can be employed to 
obtain estimates for those operators that admit them. For example, 
in Theorem 1 in \cite{muller-rr}, the authors obtain a weighted strong-type estimate for rough
singular integral operators in the vector setting by means of the integral convex
 body domination result we discussed earlier. Such a reasoning can be adapted to study the strong-type 
boundedness of the commutator of vector-valued rough
singular integrals in the two-weight setting, which is another one of the results that 
we present in this paper.

\begin{theorem}\label{th:strTocho3}
    If $\Omega\in L^{\infty}(\mathbb{S}^{d-1})$, $\int_{\mathbb{S}^{d-1}}\Omega=0$, $T=T_{\Omega}$, $1<p<\infty$,
    $\vec{B}$ is a vector of locally integrable matrix functions,
    and $U\in A_{p,\infty}^{sc}$, $V^{-\frac{p'}{p}}\in A_{p',\infty}^{sc}$, $(U,V)\in A_p$,
    \begin{align}\label{estelle}
        { \int_{\Omega}}&{ \left\lvert \esc{V^{\frac{1}{p}}(x){(T\otimes I_n)}_{\vec{B}}(U^{\frac{-1}{p}}\vec{h})(x)}{\vec{g}(x)}\right\rvert}\nonumber
        \\ & { \lesssim \|\Omega\|_{L^{\infty}(\mathbb{S}^{d-1})}  [U,V]_{A_p}^{\frac{1}{p}}
        \left(\sum_{\sigma\in C(m)} \|\vec{B}\|_{BMO^{s \gamma p}_{V^{s \gamma},U^{s \gamma},\sigma,1}}\|\vec{B}\|_{BMO^{{(rp')'},*}_{V^{\frac{(rp')'}{p}},U^{\frac{(rp')'}{p}},\sigma^c,2}}\right) } 
        \\ &  \;\;\;\; {
         \times \;t_1^{\frac{1}{p}}[V^{\frac{-p'}{p}}]^{\frac{1}{p}}_{A^{sc}_{\infty,p'}}t_2^{\frac{1}{p'}}[U]^{\frac{1}{p'}}_{A^{sc}_{\infty,p}}\min \left\{ t_2[U]_{A^{sc}_{\infty,p}},t_1[V^{\frac{-p'}{p}}]_{A^{sc}_{\infty,p'}}\right\}{\|\vec{h}\|_{L^p}\|\vec{g}\|_{L^{p'}}}},\nonumber
    \end{align}
    where  $r= 1+\frac{1}{2^{d+11}t_1[V^{\frac{-p'}{p}}]_{A^{sc}_{p',\infty}}}$,
    $\gamma = 1+\frac{1}{\left(\frac{p'+1}{2}\right)2^{d+11}t_2[U]_{A^{sc}_{p,\infty}}}$,
$s = \left(\frac{p'+1}{2}\frac{1+2^{d+11}t_2[U]_{A_{p,\infty}^{sc}}}{1+\left(\frac{p'+1}{2}\right)2^{d+11}t_2[U]_{p,\infty}^{sc}}\right)$,
  and $t_1\geq1$, $t_2\geq1$ can be chosen arbitrarily.
\end{theorem}
Using this result, we can obtain strong-type weighted
 estimates for the appropriate operators,
employing 
duality and switching $h$ for $U^{\frac{1}{p}}f$.
The $BMO$ and $A_p$ spaces that appear in this result will be defined in later sections.

In order to achieve this estimate, one of the fundamental things that we had to take into 
consideration was the way the symbols and the weights interacted in the computations, which led to 
the definition of suitable $BMO$ spaces. These $BMO$ spaces are reminiscent of 
those that appear in \cite[Corollary 4.7]{Cit2SaMWS}, which provides several characterizations of the standard definition of 
$BMO$ in the case of one matrix symbol, in terms of reducing matrices. We have attempted to obtain an 
analogue for this result that is valid for multiple symbols, but we were not able to do so
in a completely successful way.

In \cite[Lemma 1.3]{Cit2SaMWS}, a matrix analogue of the so called conjugation method (we remit to \cite{brcbmofwe} for a detailed account of that result in the scalar setting) was obtained. Our following result generalizes that result to symbol multilinear commutators.

\begin{theorem}\label{th:strTocho}
Assume $T$ is a linear scalar operator such that
 \begin{align*}
    \|T\otimes I_{2^mn}\|_{L^p(W)\rightarrow L^p(W)}\leq \phi([W]_{A_p}) 
 \end{align*}
for some positive increasing function $\phi$, and every $W\in A_p$.
Then, if $\vec{B}=(B_1,...,B_m)$ is a vector of locally integrable matrices,
then, for any $U,V\in A_p$,
\[\|(T\otimes I_{n})_{\vec{B}}\|_{L^{p}(U)\rightarrow L^{p}(V)}
\leq\min\{\beta_1,\beta_2\}^\frac{m}{p}\phi\left(c_{m,p}([U]_{A_{p}}+[V]_{A_{p}})\right)\]
where
\begin{align*}
    \beta_1 &=\max\left\{\max_{\sigma\in C(m),\sigma\not=\varnothing,\tilde{\sigma}_{2^m}} \|\vec{B}\|_{\widetilde{BMO}_{V,V,\sigma}^{p}}^{\frac{1}{|\sigma|}},\max_{\sigma\in C(m),\sigma\not=\varnothing}\|\vec{B}\|_{\widetilde{BMO}_{V,U,\sigma}^{p}}^{\frac{1}{|\sigma|}}\right\} \\
\beta_2 &=\max\left\{\max_{\sigma\in C(m),\sigma\not=\varnothing,\tilde{\sigma}_{2^m}} \|\vec{B}\|_{\widetilde{BMO}_{U,U,\sigma}^{p}}^{\frac{1}{|\sigma|}},\max_{\sigma\in C(m),\sigma\not=\varnothing}\|\vec{B}\|_{\widetilde{BMO}_{V,U,\sigma}^{p}}^{\frac{1}{|\sigma|}}\right\} 
\end{align*}
\end{theorem}

\begin{remark}
    The conditions on the $BMO$ norms on the theorem above imply that, for every operator and vector of symbols that satisfy the assumptions of the theorem with $\beta_1<\infty$ (resp. $\beta_2<\infty)$, every possible lower order commutator is bounded in $L^p(V)$ (resp. $L^p(U)$) and from $L^p(U)$ to $L^p(V)$. See Remark \ref{rem:allcomm} for additional details.
\end{remark}

\begin{remark}
The argument provided to settle the previous theorem allows to conclude that
\begin{align*}
 & \|(T\otimes I_{n})_{\vec{B}}\|_{L^{p}(U)\rightarrow L^{p}(V)}\\
\leq & \phi\left(4^{{mp}{}}\left((2^{m}-1)[V]_{A_p}+[U]_{A_p}   
        +{4^m\sum_{\underrel{\sigma\neq \varnothing,\tilde{\sigma}_{2^m}}{\sigma\in C(m)}}\|\vec{B}\|_{\widetilde{BMO}^p_{V,V,\sigma}}}    + \sum_{\underrel{\sigma\neq \varnothing}{\sigma\in C(m)}}\|\vec{B}\|_{\widetilde{BMO}^p_{V,U,\sigma}}\right)\right)
\end{align*}
and that 
\begin{align*}
 & \|(T\otimes I_{n})_{\vec{B}}\|_{L^{p}(U)\rightarrow L^{p}(V)}\\
\leq & \phi\left(4^{{mp}{}}\left((2^{m}-1)[V]_{A_p}+[U]_{A_p}   
        +{4^m\sum_{\underrel{\sigma\neq \varnothing,\tilde{\sigma}_{2^m}}{\sigma\in C(m)}}\|\vec{B}\|_{\widetilde{BMO}^p_{U,U,\sigma}}}    + \sum_{\underrel{\sigma\neq \varnothing}{\sigma\in C(m)}}\|\vec{B}\|_{\widetilde{BMO}^p_{V,U,\sigma}}\right)\right)
\end{align*}
which upon rescaling provides the statement above.  See the proof of Theorem \ref{th:strTocho} for details.  
\end{remark}

{ We have obtained as well a result that allows us to obtain estimates in terms of bumped weighted $BMO$ spaces extending results in \cite{Cit2SaMWS}. We remit the reader to Section \ref{sec:Bumped} for further details}

In the special case when $\vec{B} = (b, \ldots, b)$ for a scalar $b$ then we will write $(T \otimes I_n)_b ^m$ instead of $(T\otimes I_{n})_{\vec{B}}$. In this case we can prove a two matrix weighted generalization of Theorem $1.2$ in \cite{HW18}. Namely

\begin{theorem} \label{HWMatrixThm} If $T$ satisfies the same conditions of Theorem \ref{th:strTocho} then \begin{equation*}
\|(T\otimes I_{n})_b ^m\|_{L^{p}(U)\rightarrow L^{p}(V)} \leq  C \|b\|_{\text{BMO}} ^{m-1} \|b\|_{\text{BMO}_{V, U, 1}^p} \end{equation*} 
 where $C$ depends on  $m, p, [U]_{\text{A}_p},$  and $[V]_{\text{A}_p}$.
\end{theorem}

We refer the reader to Definition \ref{BMOScalarMatrixDef} for the precise definition of $\|b\|_{BMO_{V, U, 1}^p}$.  Also, we refer to Theorem \ref{StrongHigherJN} for a proof that $\|b\|_{BMO_{V, U, 1}^p} \approx \|b\|_{BMO(\nu)}$ when $V = v I_{n \times n}$ and $U = u I_{n \times n}$ for $u, v$ scalar $A_p$ weights, where $\nu = (u/v)^\frac{1}{p}$ and $\|b\|_{BMO(\nu)}$
is the standard Bloom BMO norm $$\|b\|_{BMO(\nu)} = \sup_Q \frac{1}{\nu(Q)} \int_{Q} |b (x) - \langle b\rangle_Q| \, dx. $$ \noindent In particular, in the special case $U = V$, Theorem $\ref{HWMatrixThm}$ says that $(T\otimes I_{n})_b ^m$ is bounded on $L^p(U)$ for a matrix weight A${}_p$ weight $U$ and any $m \in \mathbb{N}$ if $b \in \text{BMO},$ and in fact, by the proof of Theorem \ref{HWMatrixThm} we will have in this case that $$\|(T\otimes I_{n})_b ^m\|_{L^{p}(U)\rightarrow L^{p}(U)} \leq  C \|b\|_{\text{BMO}} ^m$$ where $C $ depends on $m, p,$ and  $[U]_{\text{A}_p}.$ 

While we are not able to prove it, we will conjecture in fact that \begin{equation} \|(T\otimes I_{n})_b ^m\|_{L^{p}(U)\rightarrow L^{p}(V)} \leq C \|b\|_{BMO_{V, U, m}^p} \label{LOR19Conj}\end{equation}where $C$ depends on $\|U\|_{\text{A}_p}$ and $\|V\|_{\text{A}_p}$.   In particular, if true, this would provide a two matrix weighted generalization of Theorem $1.1 (i)$ in \cite{LOR19} since Theorem \ref{StrongHigherJN} says that $\|b\|_{BMO_{V, U, m}^p} \approx \|b\|_{BMO(\nu^\frac{1}{m})}$   if $V = v I_{n \times n}$ and $U = u I_{n \times n}$ (and again $\nu = (u/v)^\frac{1}{p}$ and $\|b\|_{BMO(\nu^\frac{1}{m})}$
is the standard Bloom BMO norm for $\nu^\frac{1}{m}$.)  See Remark \ref{LOR19Rem} for a disscusion of this as it pertains to the proof of Theorem \ref{HWMatrixThm}. 


The remainder of the paper will be organized as follows. 
In Section \ref{sec:Prelim} we provide some definitions and preliminaries that we will need to build upon.
Section \ref{sec:CombOps} is devoted to settle some key lemmatta on combinatorics, relating symbols and operators.
In Section \ref{sec:BMO} we discuss relations between the BMO classes involved in our main results.
We provide the proofs of our main results in Section \ref{sec:ProfMain}.
We continue giving some further strong-type estimates in terms of bumped weighted $BMO$ classes in Section \ref{sec:Bumped}.
Finally in Section \ref{sect:apps} we gather some applications of Theorem \ref{th:tocho}.

We will end this introduction by noting that related convex body domination results for commutators $T_{\vec{b}}$ when $\vec{b} = (b_1, \ldots, b_m)$ for scalar $b_1, \ldots, b_m$ was proved in \cite{Hyt24}.  Furthermore, it was commented that one can likely extend the results and ideas in \cite{Hyt24}  to the case of matrix valued symbols $\vec{B}$ and commutators $(T\otimes I_n)_{\vec{B}}$ acting on vector valued functions, and it was communicated to us that D. Cruz-Uribe, A. Laukkarinen, and K. Moen are pursuing precisely this in work in progress.   It would be interesting to see how such results would compare to the ones proved in this paper.

\section{Definitions and preliminaries}\label{sec:Prelim}

We devote this section to introducing the main tools and preliminary results that will be vital in the following sections.
\subsection{Matrix weights}

For a matrix $A$, we will denote by $\|A\|$ the operator norm of $A$, that is,
\begin{equation*}
\begin{split}
    \|A\| = \sup_{|\vec{e}|=1}|A\vec{e}| = \sup_{\vec{e}\neq 0}\frac{|A\vec{e}|}{|\vec{e}|},
\end{split}
\end{equation*}
where $|\cdot|$ is the corresponding euclidean norm.
Recall that, for an arbitrary matrix $A$, the identity
\begin{align*}
    \|A\|=\sqrt{\rho(A^*A)}
\end{align*}
is always satisfied, where $\rho(A)$ is the eigenvalue of $A$ with the biggest absolute value.
Observe that 
$\|A\| = \|A^*\|$, where $A^*$ is the conjugate matrix of $A$.

For any positive definite matrix $A$ (which we always assume is self-adjoint), it is 
known that there exists a diagonal matrix with positive diagonal entries, $D$, and
 some orthogonal matrix $P$ such that $A = P D P^*$. Applying this, we 
 may define $A^{\alpha}$, for any $\alpha\in \mathbb{R}$, as 
 $A^\alpha = P D^{\alpha} P^*$, where $D^{\alpha}$ is the diagonal matrix whose 
 diagonal entries are those of $D$ to the power of $\alpha$.
 It can be checked that this is well-defined, that the usual definitions of $A^{-1}$ and $A^{n}$, for $n\in \mathbb{N}$,
  coincide with this one, and that 
 $A^{\alpha + \beta}=A^{\alpha}A^{\beta}$ for any choices of $\alpha$ and $\beta$.

Now we will introduce 
matrix weights, which are simply functions
$W:\mathbb{R}^d\rightarrow \mathcal{M}_n(\mathbb{R})$
such that $W(x)$ is positive definite for a.e $x\in \mathbb{R}^d$.
In the following definition we recall $A_p$ classes following Roudenko's presentation from \cite{roudenko}.

\begin{definition}
    Given two weights $U,V$ in $\mathbb{R}^n$, and $p\in (1,\infty)$, we say 
    that $U\in A_p$ if 
    \begin{equation*}
        \begin{split}
            [U]_{A_p} = \sup_{Q \text{ cube in }\mathbb{R}^n} \fint_Q \left(\fint_Q \|U^{\frac{1}{p}}(x)U^{\frac{-1}{p}}(y)\|^{p'}dy\right)^{\frac{p}{p'}}dx<\infty
        \end{split}
        \end{equation*}
        and that $(U,V)\in A_{p}$ if
    \begin{equation*}
    \begin{split}
        [U,V]_{A_p} = \sup_{Q \text{ cube in }\mathbb{R}^n}\fint_Q\left(\fint_Q \|U^{\frac{1}{p}}(x)V^{\frac{-1}{p}}(y)\|^{p'}dy\right)^{\frac{p}{p'}}dx<\infty.
    \end{split}
    \end{equation*}
\end{definition}

It is known that if $W\in A_p$, then, for every $\vec{e}\neq 0$, 
$|W^{\frac{1}{p}}(x)\vec{e}|^p$
is a scalar $A_p$ weight satisfying 
that $[|W^{\frac{1}{p}}(x)\vec{e}|^p]_{A_p}\lesssim [W]_{A_p}$
(\cite{Volberg-S}, Lemma 5.3).
Also, it is easy to check, using Lemma 5 in \cite{muller-rr} and 
Jensen's Inequality, that if $1<p<q<\infty$, then $[U,V]_{A_q}\leq [U,V]_{A_p}$ for any pair of 
weights $U,V$.

Next, we introduce yet another class of weights.
\begin{definition}
    Given a weight $W$, we say $W\in A_{p,\infty}^{sc}$ if 
    \begin{equation*}
    \begin{split}
        [W]_{A_{p,\infty}^{sc}} = \sup_{e\in \mathbb{F}^n\setminus \{0\}} [|W^{\frac{1}{p}}(\cdot)\vec{e}|^p]_{A_{\infty}}<\infty.
    \end{split}
    \end{equation*}
\end{definition}
This class acts, to some extent, as a substitute for the $A_\infty$
class in the matrix setting. It is clear that $[W]_{A_{p,\infty}^{sc}}\lesssim [W]_{A_p}$,
and 
it can be shown
that the following version of the Reverse Hölder Inequality holds.
\begin{theorem}[Reverse Hölder Inequality for matrices]
    Suppose $W\in A_{p,\infty}^{sc}$ for some $p\in (1,\infty)$, 
    and $A$ is a matrix. Then, if $r\leq 1 + \frac{1}{2^{d+11}[W]_{A_{p,\infty}^{sc}}}$,
    \begin{equation*}
    \begin{split}
        \left(\fint_Q\|W^{\frac{1}{p}}(x)A\|^{rp}dx\right)^{\frac{1}{rp}}\lesssim
        \left(\fint_Q\|W^{\frac{1}{p}}(x)A\|^{p}dx\right)^{\frac{1}{p}}.
    \end{split}
    \end{equation*}
\end{theorem}
The key to prove this result
is the following proposition, which, 
despite its simplicity, is very handy. 
\begin{proposition}\label{prop:truco_orto}
    Let $A$ be a matrix and $\{\vec{e}_\ell\}$ be an orthonormal basis of $\mathbb{F}^n$. Then 
    \begin{equation*}
    \begin{split}
        \frac{1}{n}\sum_{\ell=1}^n |A\vec{e}_\ell|\leq \|A\| \leq\sum_{\ell=1}^n |A\vec{e}_\ell|.
    \end{split}
    \end{equation*}
\end{proposition}
A general version for this proposition can be found in 
Lemma 3.2 in \cite{roudenko}.
This result is very useful when we would like to have $|A(x)\vec{e}|$, for some 
$\vec{e}\in \mathbb{F}^n$ and some matrix function $A(x)$, instead of $\|A(x)\|$. 
In the particular 
case of the matrix Reverse Hölder Inequality, doing this allows us to exploit the fact that 
$|W^{\frac{1}{p}}e|^{p}\in A_{\infty}$ to use the scalar Reverse Hölder Inequality.
In any case, the proof of the Reverse Hölder Inequality can be found in Lemma 2 in \cite{muller-rr}.
{The only difference is that that result states that $A$ has to be 
positive definite, but this has no influence on the proof.}

In \cite{Basicomatrices} and \cite{roudenko}, it is shown that $W\in A_p$ if and only if 
$W^{-\frac{p'}{p}}\in A_{p'}$, in analogy to what happens in the scalar case.
This implies that if $W\in A_p$, then, in particular, 
$W\in A_{p,\infty}^{sc}$ and $W^{-\frac{p'}{p}}\in A_{p',\infty}^{sc}$.

One of the main tools used in this article to obtain
weighted norm inequalities will be reducing matrices.
This concept was first introduced in 
\cite{GoldbergAp}.

\begin{definition}
    Given a weight $W$, $p\in[1,\infty)$ and a set $Q$ of finite non-zero measure, 
    we say that the positive definite matrix $\mathcal{R}_{Q,p,W}$ is an associated 
    reducing matrix if for every $\vec{e}\in \mathbb{F}^n$,
    \begin{equation*}
    \begin{split}
        |\mathcal{R}_{Q,p,W}\vec{e}|\simeq \left(\fint_Q \left\lvert W^{\frac{1}{p}}\vec{e}\right\rvert^{p} \right)^{\frac{1}{p}}
    \end{split}
    \end{equation*} 
    We will also denote by $\mathcal{R}_{Q,p,W}'$ any positive definite matrix such that, for every $\vec{e}\in \mathbb{F}^n$,
    \begin{equation*}
    \begin{split}
        |\mathcal{R}_{Q,p,W}'\vec{e}|\simeq \left(\fint_Q \left\lvert W^{\frac{-1}{p}}\vec{e}\right\rvert^{p'} \right)^{\frac{1}{p'}}.
    \end{split}
    \end{equation*}
\end{definition}

It should be noticed that reducing matrices are not unique, if they exist. Still, 
we will employ the notations $\mathcal{R}_{Q,p,W}$
and $\mathcal{R}_{Q,p,W}'$ to refer to any corresponding 
reducing matrix.
It is clear that some $\mathcal{R}_{Q,p,W}$ exists if 
$\fint_Q \left\lvert W^{\frac{1}{p}}\vec{e}\right\rvert^{p}$ is finite for all $e\in \mathbb{F}^n$
(and respectively $\mathcal{R}_{Q,p,W}'$ exists if the corresponding integrals are finite),
as it is shown in Proposition 1.2 in \cite{GoldbergAp}, in a more general case.
This condition is guaranteed, for instance, if $W\in A_{p,\infty}^{sc}$ (or, respectively, $W^{-\frac{p'}{p}}\in A_{p',\infty}^{sc}$).
{Reducing matrices are very useful to decouple certain terms in integrals, as we will see in 
section \ref{sect:str}}, for example. 

Observe that any reducing matrix of the form $\mathcal{R}_{Q,p,W}'$
is also of the form $\mathcal{R}_{Q,p',W^{-\frac{p'}{p}}}$ and vice versa. 
Likewise,
"$\mathcal{R}_{Q,p',W^{-\frac{p'}{p}}}' = \mathcal{R}_{Q,p,W}$".

In the context of reducing matrices, Proposition \ref{prop:truco_orto} is again very 
useful, as it shows that one can substitute the vector $\vec{e}$ in the previous definition 
for a matrix, as well as some other tricks.
Some equivalences that one can show using reducing matrices and this idea are the following:
\begin{enumerate}
    \item $|\mathcal{R}_{Q,p,W}\vec{e}|\simeq |m_Q(W^{\frac{1}{p}})\vec{e}|$.
    \item $[W]_{A_p}^{\frac{1}{p}}\simeq  \sup_{Q \text{ cube in }\mathbb{R}^n}  \|\mathcal{R}_{Q,p,W}\mathcal{R}_{Q,p,W}'\|\simeq [W^{-\frac{p'}{p}}]_{A_{p'}}^{\frac{1}{p'}}$.
    \item $[U,V]_{A_p}\simeq \sup_{Q \text{ cube in }\mathbb{R}^n}  \|\mathcal{R}_{Q,p,U}\mathcal{R}_{Q,p,V}'\|^{p}$.
\end{enumerate}

The equivalence in $i)$ is proven in Lemma 2.2 in \cite{paraprods}, while the others can be proven 
easily by means of Proposition \ref{prop:truco_orto}. 
These equivalences will be used 
repeatedly throughout the paper, as they are very useful 
when simplifying many lines of reasoning.
As an example of application, observe that, for weights $U,V$ and any cube $Q$, 
\begin{equation*}
\begin{split}
    \|\mathcal{R}_{Q,p,U}\mathcal{R}_{Q,p,U}'\|
    &\leq \fint_Q \|V^{-\frac{1}{p}}(x)\mathcal{R}_{Q,p,U}\|\|V^{\frac{1}{p}}(x)\mathcal{R}_{Q,p,U}'\|dx
    \\ & \leq \left(\fint_Q \|V^{-\frac{1}{p}}(x)\mathcal{R}_{Q,p,U}\|^{p'}dx\right)^{\frac{1}{p'}}\left(\fint_Q\|V^{\frac{1}{p}}(x)\mathcal{R}_{Q,p,U}'\|^{p}dx\right)^{\frac{1}{p}}
    \\ & \simeq \|\mathcal{R}_{Q,p,U}\mathcal{R}_{Q,p,V}'\| \|\mathcal{R}_{Q,p,V}\mathcal{R}_{Q,p,U}'\|\simeq [U,V]_{A_p}^{\frac{1}{p}}[V,U]_{A_p}^{\frac{1}{p}},
\end{split}
\end{equation*}
so $[U]_{A_p}\lesssim [U,V]_{A_p}[V,U]_{A_p}$.
Also, the equivalences can be used to prove that, if 
$U\in A_{p,\infty}^{sc}$, $V$ is a weight such that 
$V^{-\frac{p'}{p}}\in A_{p',\infty}^{sc}$, and 
$r\leq 1+\frac{1}{2^{d+11}\max\{[U]_{A_{p,\infty}^{sc}},[V^{-\frac{p'}{p}}]_{A_{p',\infty}^{sc}}\}}$,
then  
\begin{align*}
    \fint_Q\left(\fint_Q \|U^{\frac{1}{p}}(x)V^{-\frac{1}{p}}(y)\|^{(pr)'}dy\right)^{\frac{pr}{(pr)'}}dx
    &\leq\fint_Q\left(\fint_Q \|U^{\frac{1}{p}}(x)V^{-\frac{1}{p}}(y)\|^{rp'}dy\right)^{\frac{pr}{rp'}}dx
    \\ & \lesssim \fint_Q\left(\fint_Q \|U^{\frac{1}{p}}(x)V^{-\frac{1}{p}}(y)\|^{p'}dy\right)^{\frac{pr}{p'}}dx
    \\ & \simeq \fint_Q\|U^{\frac{1}{p}}(x)\mathcal{R}_{Q,p,V}'\|^{pr}dx
    \\ & \lesssim \left(\fint_Q\|U^{\frac{1}{p}}(x)\mathcal{R}_{Q,p,V}'\|^{p}dx\right)^{\frac{1}{p}pr}
    \\ & \simeq \|\mathcal{R}_{Q,p,U}\mathcal{R}_{Q,p,V}'\|^{pr}
    \simeq [U,V]_{A_p}^{r},
\end{align*}
so $[U^r,V^r]_{A_{pr}}\lesssim [U,V]_{A_p}^{r}$. In the first inequality 
we applied that $(rp)'\leq rp'$ and Hölder's inequality, and in the second and third ones 
we applied the Reverse Hölder Inequality.

\subsection{Combinatorics}\label{sect:combinatorics}

    In order to write the main results of this work, we will need to introduce a certain set of tuples that 
    will simplify the notation significantly. We shall first define
\begin{equation*}
\begin{split}
    C(m)=\left\{ \text{increasingly ordered tuples composed of different elements of }\left\{ 1,2,...,m\right\}\right\}\cup \{\varnothing\},
\end{split}
\end{equation*}
and, given $j=1,2,...,m$,
$$C_j(m) = \left\{ \text{increasingly ordered }j-\text{tuples composed of different elements of }\left\{ 1,2,...,m\right\}\right\}.$$
If $j=0$, we will consider the empty tuple to be the only element of $C_0(m)$.
As an example, $C(3) = \left\{ \varnothing, (1),(2),(3),(1,2),(1,3),(2,3),(1,2,3)\right\}$
and $C_1(3)=\left\{(1),(2),(3) \right\}$.

Now, given $\alpha ,\beta  \in C(m)$, for some $m$, 
    we will write $\alpha\leq \beta$ to indicate that 
    $\alpha$ is an increasingly ordered tuple whose elements are 
    elements of $\beta$. If $\alpha\leq \beta$ and $\alpha\neq \beta$, we will 
    write $\alpha<\beta$ instead. 
    In addition, $|\alpha|$ will denote the number of elements 
    $\alpha$ has, and if $\alpha\leq \beta$, $\beta-\alpha$ will be 
    the increasingly ordered tuple whose elements are those of 
    $\beta$ that are not in $\alpha$.
    It is clear that $\leq$ is a partial order, and that,
     if $\alpha\leq \beta$, 
    $|\beta-\alpha| = |\beta|-|\alpha|$.
    Also, for a given $\sigma\in C(m)$, for some $m$, and $j$ such that 
    $1\leq j \leq |\sigma|$, $\sigma(j)$ will denote the element in the 
    $j-$th position of $\sigma$.

    We will also consider, for an increasingly ordered tuple $\theta$ of natural numbers, the sets
    \begin{equation*}
    \begin{split}
        C(\theta) = \left\{ \sigma \text{ increasingly ordered tuple such that }\sigma\leq \theta\right\}
    \end{split}
    \end{equation*}
    and, for $j=0,1,...,|\theta|$,
    \begin{equation*}
    \begin{split}
        C_j(\theta) = \left\{ \sigma\in C(\theta):|\sigma|=j\right\}.
    \end{split}
    \end{equation*}
Observe that each $C_j(\theta)$ has $\binom{|\theta|}{j}$ elements, and as a result, there are 
$2^{|\theta|}$ tuples in total in $C(\theta)$.

For convenience, we will name the elements of $C_j(m)$
as follows:
\begin{equation*}
\begin{split}
    C_j(m)=\left\{ \sigma^j_1,...,\sigma^j_{\binom{m}{j}}\right\}.
\end{split}
\end{equation*}
Now we shall define, for a given $j=1,2,...,2^m$, $k_0(j)$ as the only number that satisfies the 
identity 
\begin{equation*}
\begin{split}
    1+\sum_{k=0}^{k_0(j)-1}\binom{m}{k}\leq j \leq \sum_{k=0}^{k_0(j)}\binom{m}{k}.
\end{split}
\end{equation*}
This will allow us to provide a different labeling for the tuples of $C(m)$, 
given by
\begin{equation}
    \begin{split}
        \tilde{\sigma}_j=\sigma_{j-\sum_{k=0}^{k_0(j)-1}\binom{m}{k}}^{k_0(j)},
    \end{split}
    \end{equation}
    which will be useful when working with matrices.
    This notation allows us to put all the tuples together in an orderly manner
    (first the empty tuple, then all the tuples with one element of $\left\{ 1,2,...,m\right\}$, and so on until we arrive at the tuple with all the elements
    of the aforementioned set), so that
    \begin{equation*}\label{aux4}
    \begin{split}
        C_j(m) = \left\{ \tilde{\sigma}_i\right\}_{i = 1+\sum_{k=0}^{j-1}\binom{m}{k}}^{\sum_{k=0}^{j}\binom{m}{k}}.
    \end{split}
    \end{equation*}

Also, for $\sigma\in C(\theta)$, for some $\theta$, $\sigma^c$ will denote the $(|\theta|-|\sigma|)-$tuple whose elements are 
those of $\theta$ that are not present in $\sigma$, increasingly ordered (i.e. $\theta-\sigma$), 
and $\sigma^t$ will denote the tuple with the same elements as $\sigma$, but in reverse order.
Notice that, if $\alpha,\beta\in C(\theta)$, $\alpha\leq \beta $ if and only if $\beta^c\leq \alpha^c$.
As an example, if $m=7$, and $\sigma=(1,3,4,6)\in C(7)$, $\sigma^c = (2,5,7)$, $\sigma^t=(6,4,3,1)$ and $(\sigma^c)^t=(7,5,2)$.
When we use the notation $\sigma^c$, the tuple $\theta$ such that $\sigma\in C(\theta)$ will be clear from the context.


Finally, for $\sigma\in C_j(m)$, and a vector $\vec{B}=(B_1,...,B_m)$ of matrices or numbers,
$\vec{B}^* = (B_1^*,...,B_m^*)$, and  
$${\vec{B}}_\sigma = \prod_{k=1}^j {B}_{\sigma(k)}=B_{\sigma(j)}B_{\sigma(j-1)}...B_{\sigma(1)}.$$
Observe that the order in which the products are computed is relevant in the matrix case.

\section{Key Lemma relating combinatorics, symbols and operators}\label{sec:CombOps}
The main result of this section is the following lemma, which will play a key role in settling some of the strong-type estimates and in studying the $BMO$ classes in the following section.
\begin{lemma}\label{lemma:perm2}\label{prop:commbueno}
    Let $\theta\in C(m)$ for some $m$, and 
    $\vec{B} = (B_1,...,B_m)$, where 
    each $B_j\in M_{n}(\mathbb{F})$ is locally integrable. Then, for a linear operator $T$ and a function $f$, formally,
    \begin{equation*}
    \begin{split}
        (-1)^{|\theta|}&\sum_{\sigma\in C(\theta)}(-1)^{|\theta|-|\sigma|}\vec{B}_{\sigma}(x)(T\otimes I_n)\left(\vec{B}_{(\theta-\sigma)^t}f\right)(y) 
    \\ & = \sum_{\sigma\in C(\theta)}\left(\sum_{\alpha\in C(\sigma)}(-1)^{|\theta|-|\alpha|}\vec{B}_{\alpha}(x)(m_Q\vec{B})_{(\sigma-\alpha)^t}\right)(T\otimes I_n)\left(\sum_{\beta\in C(\sigma^c)}(-1)^{|\theta|-|\beta|}(m_Q\vec{B})_{\beta}\vec{B}_{(\sigma^c-\beta)^t}f\right)(y).
    \end{split}
    \end{equation*}
    In particular, if $T$ is the identity and $f\equiv 1$,  
    \begin{equation*}
        \begin{split}
            (-1)^{|\theta|}\sum_{\sigma\in C(\theta)}&(-1)^{|\theta|-|\sigma|}\vec{B}_{\sigma}(x)\vec{B}_{(\theta-\sigma)^t}(y) 
        \\ & = \sum_{\sigma\in C(\theta)}\left(\sum_{\alpha\in C(\sigma)}(-1)^{|\theta|-|\alpha|}\vec{B}_{\alpha}(x)(m_Q\vec{B})_{(\sigma-\alpha)^t}\right)\left(\sum_{\beta\in C(\sigma^c)}(-1)^{|\theta|-|\beta|}(m_Q\vec{B})_{\beta}\vec{B}_{(\sigma^c-\beta)^t}(y)\right).
        \end{split}
        \end{equation*}
\end{lemma}
In order to show Lemma \ref{prop:commbueno}, we will first prove the following lemma.
\begin{lemma}\label{lemma:perm1}\label{lemma:tralalero}
    Let $\theta\in C(m)$ for some $m$, and fix 
    $\alpha,\beta\in C(\theta)$ such that 
    $\alpha < \beta^c$. Then, given
    $\vec{B} = (B_{\theta(1)},...,B_{\theta(|\theta|)})$, where 
    each $B_j\in M_{n}(\mathbb{F})$, we have that
    \begin{equation*}
    \begin{split}
        \sum_{\underrel{\alpha\leq \sigma \leq \beta^c}{\sigma\in C(\theta)}}(-1)^{|\sigma|} \vec{B}_{(\sigma-\alpha)^t}(x)\vec{B}_{\sigma^c-\beta}(x) = 0
    \end{split}
    \end{equation*}
    (in particular, this works for a constant $\vec{B}$).
\end{lemma}
\begin{proof}
    To prove the result, we shall proceed by induction over the 
    difference $d=|\beta^c|-|\alpha|$. If $d=1$,
    then $\beta^c-\alpha=(j)$, and the only possible choices of $\sigma$ are 
    $\alpha$ and $\beta^c$, so that 
    \begin{equation*}
        \begin{split}
            \sum_{\underrel{\alpha\leq \sigma \leq \beta^c}{\sigma\in C(\theta)}}
            (-1)^{|\sigma|} \vec{B}_{(\sigma-\alpha)^t}(x)\vec{B}_{\sigma^c-\beta}(x)
            = (-1)^{|\alpha|}\vec{B}_{j}(x) + (-1)^{|\alpha|+1} \vec{B}_{j}(x)=0.
        \end{split}
        \end{equation*}

        If we assume that the result is true for any choice of $\theta,\alpha, \beta$ and $\vec{B}$ in the conditions of 
        the statement, for a given 
        $d=n$. Then for $d=n+1$, we have that 
        $\beta^c-\alpha = (j_1,...,j_{n+1})$, and that, in particular, 
        $j_1$ belongs to either $\sigma$ or $\sigma^c$. In the first case, 
        we have that $j_1$ belongs to $\sigma-\alpha$, and, in fact, it is the 
        first element of said tuple, since $\sigma-\alpha \leq \beta^c-\alpha=(j_1,...,j_{n+1})$. 
        Likewise, if $j_1$ belongs to $\sigma^c$, then $j_1$ also takes the first position
        of $\sigma^c-\beta$. As a result, 
        \begin{equation*}
            \begin{split}
                \sum_{\underrel{\alpha\leq \sigma \leq \beta^c}{\sigma\in C(\theta)}}
                (-1)^{|\sigma|}& \vec{B}_{(\sigma-\alpha)^t}(x)\vec{B}_{\sigma^c-\beta}(x)
                = \sum_{\underrel{j_1\in \sigma}{\underrel{\alpha\leq \sigma \leq \beta^c}{\sigma\in C(\theta)}}}
                (-1)^{|\sigma|} \vec{B}_{(\sigma-\alpha)^t}(x)\vec{B}_{\sigma^c-\beta}(x)
                 + \sum_{\underrel{j_1\in \sigma^c}{\underrel{\alpha\leq \sigma \leq \beta^c}{\sigma\in C(\theta)}}}
                (-1)^{|\sigma|} \vec{B}_{(\sigma-\alpha)^t}(x)\vec{B}_{\sigma^c-\beta}(x)
                \\ & =B_{j_1}(x)\sum_{\underrel{j_1\in \sigma}{\underrel{\alpha\leq \sigma \leq \beta^c}{\sigma\in C(\theta)}}}
                (-1)^{|\sigma-(j_1)|+1} \vec{B}_{((\sigma-(j_1))-\alpha)^t}(x)\vec{B}_{((\theta-(j_1))-(\sigma-(j_1)))-\beta}(x)
                \\ & + \left(\sum_{\underrel{j_1\in \sigma^c}{\underrel{\alpha\leq \sigma \leq \beta^c}{\sigma\in C(\theta)}}}
                (-1)^{|\sigma|} \vec{B}_{(\sigma-\alpha)^t}(x)\vec{B}_{((\theta-(j_1))-\sigma)-\beta}(x)\right)B_{j_1}(x)
                \\ & =-B_{j_1}(x)\sum_{{\underrel{\alpha\leq \sigma \leq (\theta-(j_1))-\beta}{\sigma\in C(\theta-(j_1))}}}
                (-1)^{|\sigma|} \vec{B}_{(\sigma-\alpha)^t}(x)\vec{B}_{((\theta-(j_1))-\sigma)-\beta}(x)
                \\ & + \left(\sum_{{\underrel{\alpha\leq \sigma \leq (\theta-(j_1))-\beta}{\sigma\in C(\theta-(j_1))}}}
                (-1)^{|\sigma|} \vec{B}_{(\sigma-\alpha)^t}(x)\vec{B}_{((\theta-(j_1))-\sigma)-\beta}(x)\right)B_{j_1}(x)
                =0,
            \end{split}
            \end{equation*}
            where in the last equality we used the induction hypothesis. This concludes the result.
\end{proof}

Now we are ready to prove Lemma \ref{prop:commbueno}.

\begin{proof}[Proof of Lemma \ref{lemma:perm2}]
    The proof follows from some computations and the fact that an operator 
    of the form $T\otimes I_n$, where $T$ is linear, is linear over matrices. In particular,
    \begin{align*}
        \sum_{\sigma\in C(\theta)}
        &\left(\sum_{\alpha\in C(\sigma)} (-1)^{|\theta|-|\alpha|}\vec{B}_{\alpha}(x)(m_Q\vec{B})_{(\sigma-\alpha)^t}\right)
        (T\otimes I_n){\left(\sum_{\beta\in C(\sigma^c)} (-1)^{|\theta|-|\beta|} (m_Q\vec{B})_{\beta}\vec{B}_{(\sigma^c-\beta)^t}f\right)}(y)
        \\ & = \sum_{\sigma\in C(\theta)}
        \left(\sum_{\alpha\in C(\sigma)} (-1)^{|\theta|-|\alpha|}\vec{B}_{\alpha}(x)(m_Q\vec{B})_{(\sigma-\alpha)^t}\right)
        (T\otimes I_n){\left(\sum_{\beta\in C(\sigma^c)} (-1)^{-|\sigma|-|\beta|}( m_Q\vec{B})_{\sigma^c-\beta}\vec{B}_{\beta^t}f\right)}(y)
        \\ & = \sum_{\sigma\in C(\theta)}
        \sum_{\underrel{\beta\in C(\sigma^c)}{\alpha\in C(\sigma)}}(-1)^{|\theta|-|\alpha|-|\sigma|-|\beta|}
        {  \vec{B}_{\alpha}(x)(m_Q\vec{B})_{(\sigma-\alpha)^t}(m_Q\vec{B})_{\sigma^c-\beta}(T\otimes I_n)\left(\vec{B}_{\beta^t}f\right)(y)}
        { }
        \\ & = \sum_{{\alpha,\beta\in C(\theta)}}\sum_{\underrel{\alpha\leq \sigma\leq \beta^c}{\sigma\in C(\theta)}}(-1)^{{|\theta|-|\alpha|-|\sigma|-|\beta|}}
        { \vec{B}_{\alpha}(x)(m_Q\vec{B})_{(\sigma-\alpha)^t}(m_Q\vec{B})_{\sigma^c-\beta}(T\otimes I_n)\left(\vec{B}_{\beta^t}f\right)(y)}
        { }
        \\ & = \sum_{\underrel{\alpha<\beta^c}{\alpha,\beta\in C(\theta)}}\vec{B}_{\alpha}(x)
        {\left(\sum_{\underrel{\alpha\leq \sigma\leq \beta^c}{\sigma\in C(\theta)}}(-1)^{{|\theta|-|\alpha|-|\sigma|-|\beta|}} (m_Q\vec{B})_{(\sigma-\alpha)^t}(m_Q\vec{B})_{\sigma^c-\beta}\right)(T\otimes I_n)\left(\vec{B}_{\beta^t}f\right)(y)}
        { }
        \\ & +\sum_{{\sigma\in C(\theta)}}(-1)^{{|\sigma|}}
        {  \vec{B}_{\sigma}(x)(T\otimes I_n)\left(\vec{B}_{(\sigma^c)^t}f\right)(y)}
        {}
        \\ & =
        (-1)^{|\theta|}\sum_{{\sigma\in C(\theta)}}(-1)^{{|\theta|-|\sigma|}}
        {  \vec{B}_{\sigma}(x)(T\otimes I_n)\left(\vec{B}_{(\sigma^c)^t}f\right)(y)}.
    \end{align*}
    In the first equality, we made the change of variable 
    $\beta'=\sigma^c-\beta$, and in the last equality, we applied the previous lemma, which guarantees that 
    the first term in the last member of the upper term 
    vanishes. We thus arrive at the desired conclusion.
\end{proof}

\section{$BMO$ spaces}\label{sec:BMO}
In {\cite{Cit2SaMWS}}, the $BMO$ norm of one matrix symbol $B$ with
respect to two weights $U,V$ was defined as 
\begin{equation}
\begin{split}\|{B}\|_{BMO_{V,U}^{p}} & =\sup_{Q}\left(\fint_{Q}\|(m_{Q}V^{\frac{1}{p}})\left({B}(x)-m_{Q}{B}\right)(m_{Q}U^{\frac{1}{p}})^{-1}\|dx\right),\end{split}
\label{def_BMO_unidim}
\end{equation}
where $m_{Q}$ denotes the integral mean $\fint_{Q}$. In the setting
we will be working on, where we have $\vec{B}=(B_{1},...,B_{m})$
a vector of locally integrable matrices, and $\sigma\in C(m)$, we
will consider two natural extensions of (\ref{def_BMO_unidim}) given
by 
\begin{equation}
\begin{split}\|\vec{B}\|_{BMO_{V,U,\sigma}^{p,av}} & =\sup_{Q}\left(\fint_{Q}\|(m_{Q}V^{\frac{1}{p}})\left(\sum_{\alpha\in C(\sigma)}(-1)^{|\sigma|-|\alpha|}\vec{B}_{\alpha}(x)(m_{Q}\vec{B})_{(\sigma-\alpha)^{t}}\right)(m_{Q}U^{\frac{1}{p}})^{-1}\|^{{\frac{1}{|\sigma|}}}dx\right),\\
\|\vec{B}\|_{BMO_{V,U,\sigma}^{p,av,*}} & =\sup_{Q}\left(\fint_{Q}\|(m_{Q}U^{\frac{1}{p}})^{-1}\left(\sum_{\alpha\in C(\sigma)}(-1)^{|\sigma|-|\alpha|}(m_{Q}\vec{B})_{\alpha}\vec{B}_{(\sigma-\alpha)^{t}}(x)\right)^{*}(m_{Q}V^{\frac{1}{p}})\|^{{\frac{1}{|\sigma|}}}dx\right).
\end{split}
\label{def:BMO}
\end{equation}
Notice that the only change between these two norms is the order in
which the symbols and their integral means appear in the corresponding
expressions, which is irrelevant in the case $m=1$. In the matrix
weighted setting, it is more suitable to work with definitions in
terms of reducing matrices. Hence we define
\begin{equation}
\begin{split}\|\vec{B}\|_{BMO_{V,U,\sigma}^{p}} & =\sup_{Q}\left(\fint_{Q}\|\mathcal{R}_{Q,p,V}\left(\sum_{\alpha\in C(\sigma)}(-1)^{|\sigma|-|\alpha|}\vec{B}_{\alpha}(x)(m_{Q}\vec{B})_{(\sigma-\alpha)^{t}}\right)\mathcal{R}_{Q,p,U}^{-1}\|^{{\frac{1}{|\sigma|}}}dx\right)\\
\|\vec{B}\|_{BMO_{V,U,\sigma}^{p,*}} & =\sup_{Q}\left(\fint_{Q}\|\mathcal{R}_{Q,p,U}^{-1}\left(\sum_{\alpha\in C(\sigma)}(-1)^{|\sigma|-|\alpha|}(m_{Q}\vec{B})_{\alpha}\vec{B}_{(\sigma-\alpha)^{t}}(x)\right)^{*}\mathcal{R}_{Q,p,V}\|^{{\frac{1}{|\sigma|}}}dx\right)
\end{split}
\label{def_BMO_red}
\end{equation}

\begin{proposition}
Let $U,V\in A_{p}$. Then we have that 
\begin{enumerate}
\item ${\displaystyle \frac{1}{[U]_{A_{p}}^{\frac{n}{{|\sigma|}{p}}}}}\|\vec{B}\|_{BMO_{V,U,\sigma}^{p,av}}\leq\|\vec{B}\|_{BMO_{V,U,\sigma}^{p}}\leq[V]_{A_{p}}^{\frac{n}{{|\sigma|}p}}\|\vec{B}\|_{BMO_{V,U,\sigma}^{p,av}}$, 
\item ${\displaystyle \frac{1}{[U]_{A_{p}}^{\frac{n}{{|\sigma|}p}}}}\|\vec{B}\|_{BMO_{V,U,\sigma}^{p,av,*}}\leq\|\vec{B}\|_{BMO_{V,U,\sigma}^{p,*}}\leq[V]_{A_{p}}^{\frac{n}{{|\sigma|}p}}\|\vec{B}\|_{BMO_{V,U,\sigma}^{p,av,*}}$. 
\end{enumerate}
\end{proposition}

\begin{proof}
We begin recalling that it follows from \cite[Lemma 2.2]{paraprods},
that if $W\in A_{p}$, then for every $\vec{e}\in\mathbb{F}^{n}$
\[
|m_{Q}(W^{\frac{1}{p}})\vec{e}|\leq|\mathcal{R}_{Q,p,W}\vec{e}|\leq[W]_{A_{p}}^{\frac{n}{p}}|m_{Q}(W^{\frac{1}{p}})\vec{e}|.
\]
Bearing this in mind we settle $i)$ since $ii)$ is analogous. Note
that 
\begin{align*}
 & \fint_{Q}\|(m_{Q}V^{\frac{1}{p}})\left(\sum_{\alpha\in C(\sigma)}(-1)^{|\sigma|-|\alpha|}\vec{B}_{\alpha}(x)(m_{Q}\vec{B})_{(\sigma-\alpha)^{t}}\right)(m_{Q}U^{\frac{1}{p}})^{-1}\|^{{\frac{1}{|\sigma|}}}dx\\
\leq & \|(m_{Q}V^{\frac{1}{p}})\mathcal{R}_{Q,p,V}^{-1}\|^{{\frac{1}{|\sigma|}}}\|\mathcal{R}_{Q,p,U}(m_{Q}U^{\frac{1}{p}})^{-1}\|^{{\frac{1}{|\sigma|}}}\|\vec{B}\|_{BMO_{V,U,\sigma}^{p}}\\
\leq & [U]_{A_{p}}^{\frac{n}{{|\sigma|}p}}\|\vec{B}\|_{BMO_{V,U,\sigma}^{p}}
\end{align*}
from which the leftmost estimate readily follows. For the rightmost
estimate
\begin{align*}
 & \fint_{Q}\|\mathcal{R}_{Q,p,V}\left(\sum_{\alpha\in C(\sigma)}(-1)^{|\sigma|-|\alpha|}\vec{B}_{\alpha}(x)(m_{Q}\vec{B})_{(\sigma-\alpha)^{t}}\right)\mathcal{R}_{Q,p,U}^{-1}\|^{{\frac{1}{|\sigma|}}}dx\\
\leq & \|\mathcal{R}_{Q,p,V}(m_{Q}V^{\frac{1}{p}})^{-1}\|^{{\frac{1}{|\sigma|}}}\|(m_{Q}U^{\frac{1}{p}})\mathcal{R}_{Q,p,U}^{-1}\|^{{\frac{1}{|\sigma|}}}\|\vec{B}\|_{BMO_{V,U,\sigma}^{p,av}}\\
\leq & [V]_{A_{p}}^{\frac{n}{{|\sigma|}p}}\|\vec{B}\|_{BMO_{V,U,\sigma}^{p,av}}
\end{align*}
from which the desired inequality follows.
\end{proof}

The next proposition establishes an interesting relationship between the norms in \eqref{def_BMO_red}.

\begin{proposition}\label{prop:relBMO-BMO*}
Let $1<p<\infty$. 
\begin{enumerate}
\item If $V\in A_{p}$, then $\|\vec{B}\|_{BMO_{V,U,\sigma}^{p}}\lesssim[V]_{A_{p}}^{\frac{1}{{|\sigma|}p}}\|\vec{B}^{*}\|_{BMO_{U^{\frac{-p'}{p}},V^{\frac{-p'}{p}},\sigma}^{p',*}}$.
\item If $U\in A_{p}$ then $\|\vec{B}^{*}\|_{BMO_{U^{\frac{-p'}{p}},V^{\frac{-p'}{p}},\sigma}^{p',*}}\lesssim[U]_{A_{p}}^{\frac{1}{{|\sigma|}p}}\|\vec{B}\|_{BMO_{V,U,\sigma}^{p}}$.
\item If $U,V\in A_{p}$, then 
\[
\frac{1}{[U]_{A_{p}}^{\frac{1}{{|\sigma|}p}}}\|\vec{B}^{*}\|_{BMO_{U^{\frac{-p'}{p}},V^{\frac{-p'}{p}},\sigma}^{p',*}}\lesssim\|\vec{B}\|_{BMO_{V,U,\sigma}^{p}}\lesssim[V]_{A_{p}}^{\frac{1}{{|\sigma|}p}}\|\vec{B}^{*}\|_{BMO_{U^{\frac{-p'}{p}},V^{\frac{-p'}{p}},\sigma}^{p',*}}.
\]
\end{enumerate}
\end{proposition}

\begin{proof}
It suffices to settle $i)$, since $ii)$ is analogous and $iii)$ is a
direct consequence of the combination of $i)$ and $ii)$. We follow
ideas in the proof of Corollary 4.7 in \cite{Cit2SaMWS}.
\begin{align*}
\fint_{Q} & \|\mathcal{R}_{Q,p,V}\left(\sum_{\alpha\in C(\sigma)}(-1)^{|\sigma|-|\alpha|}\vec{B}_{\alpha}(x)(m_{Q}\vec{B})_{(\sigma-\alpha)^{t}}\right)\mathcal{R}_{Q,p,U}^{-1}\|^{{\frac{1}{|\sigma|}}}dx\\
 & \leq\fint_{Q}\|\mathcal{R}_{Q,p,V}\mathcal{R}_{Q,p,V}'\|^{{\frac{1}{|\sigma|}}}\|(\mathcal{R}_{Q,p,V}')^{-1}\left(\sum_{\alpha\in C(\sigma)}(-1)^{|\sigma|-|\alpha|}\vec{B}_{\alpha}(x)(m_{Q}\vec{B})_{(\sigma-\alpha)^{t}}\right)\mathcal{R}_{Q,p,U}^{-1}\|^{{\frac{1}{|\sigma|}}}dx\\
 & \lesssim[V]_{A_{p}}^{\frac{1}{{|\sigma|}p}}\fint_{Q}\fint_{Q}\|(\mathcal{R}_{Q,p,V}')^{-1}\left(\sum_{\alpha\in C(\sigma)}(-1)^{|\sigma|-|\alpha|}\vec{B}_{\alpha}(x)(m_{Q}\vec{B})_{(\sigma-\alpha)^{t}}\right)U^{\frac{-1}{p}}(y)U^{\frac{1}{p}}(y)\mathcal{R}_{Q,p,U}^{-1}\|^{{\frac{1}{|\sigma|}}}dydx%
\\
 & {\leq}[V]_{A_{p}}^{\frac{1}{{|\sigma|}p}}\fint_{Q}\left(\fint_{Q}\|(\mathcal{R}_{Q,p,V}')^{-1}\left(\sum_{\alpha\in C(\sigma)}(-1)^{|\sigma|-|\alpha|}\vec{B}_{\alpha}(x)(m_{Q}\vec{B})_{(\sigma-\alpha)^{t}}\right)U^{\frac{-1}{p}}(y)\|^{\frac{p'}{{|\sigma|}}}dy\right)^{\frac{1}{p'}}\\
 & \;\;\;\;\times\left(\fint_{Q}\|U^{\frac{1}{p}}(y)\mathcal{R}_{Q,p,U}^{-1}\|^{\frac{p}{{|\sigma|}}}dy\right)^{\frac{1}{p}}dx.
 \end{align*}
 {We may now apply Jensen's inequality to ensure that}
 \begin{align*}
 \fint_{Q} & \|\mathcal{R}_{Q,p,V}\left(\sum_{\alpha\in C(\sigma)}(-1)^{|\sigma|-|\alpha|}\vec{B}_{\alpha}(x)(m_{Q}\vec{B})_{(\sigma-\alpha)^{t}}\right)\mathcal{R}_{Q,p,U}^{-1}\|^{{\frac{1}{|\sigma|}}}dx\\
 & \lesssim [V]_{A_{p}}^{\frac{1}{{|\sigma|}p}}\fint_{Q}\left(\fint_{Q}\|(\mathcal{R}_{Q,p,V}')^{-1}\left(\sum_{\alpha\in C(\sigma)}(-1)^{|\sigma|-|\alpha|}\vec{B}_{\alpha}(x)(m_{Q}\vec{B})_{(\sigma-\alpha)^{t}}\right)U^{\frac{-1}{p}}(y)\|^{{p'}}dy\right)^{\frac{1}{{|\sigma|}p'}}\\
 & \;\;\;\;\times\left(\fint_{Q}\|U^{\frac{1}{p}}(y)\mathcal{R}_{Q,p,U}^{-1}\|^{p}dy\right)^{\frac{1}{{|\sigma|}p}}dx.
 \\
 & \lesssim[V]_{A_{p}}^{\frac{1}{{|\sigma|}p}}\fint_{Q}\|(\mathcal{R}_{Q,p,V}')^{-1}\left(\sum_{\alpha\in C(\sigma)}(-1)^{|\sigma|-|\alpha|}\vec{B}_{\alpha}(x)(m_{Q}\vec{B})_{(\sigma-\alpha)^{t}}\right)\mathcal{R}_{Q,p,U}'\|^{{\frac{1}{{|\sigma|}}}}dx\\
 & =[V]_{A_{p}}^{\frac{1}{{|\sigma|}p}}\fint_{Q}\|\mathcal{R}_{Q,p',U^{\frac{-p'}{p}}}\left(\sum_{\alpha\in C(\sigma)}(-1)^{|\sigma|-|\alpha|}(m_{Q}(\vec{B}^{*}))_{\alpha}\vec{B}_{(\sigma-\alpha)^{t}}^{*}(x)\right)(\mathcal{R}_{Q,p',V^{\frac{-p'}{p}}})^{-1}\|^{{\frac{1}{|\sigma|}}}dx,
\end{align*}
so that $\|\vec{B}\|_{BMO_{V,U,\sigma}^{p}}\lesssim[V]_{A_{p}}^{\frac{1}{{|\sigma|}p}}\|\vec{B}^{*}\|_{BMO_{U^{\frac{-p'}{p}},V^{\frac{-p'}{p}},\sigma}^{p',*}}$.
\end{proof}

Equivalent expressions for the $BMO$ norm (\ref{def_BMO_unidim}),
which employ reducing matrices instead of weights, were introduced
in \cite{Cit2SaMWS} (Corollary 4.7). Such expressions were very useful
in obtaining estimates for certain operators, since working with reducing
matrices is much easier than working with matrix functions. In order
to emulate this, we introduce the following norms: 
\begin{definition}
Let $U,V$ be matrix weights, $\vec{B}$ a vector of locally integrable
matrices, and $\sigma\in C(m)$. We define 
\begin{align*}
\|\vec{B}\|_{BMO_{V,U,\sigma,1}^{p}} & =\sup_{Q}\left(\fint_{Q}\left\Vert V^{\frac{1}{p}}(y)\left(\sum_{\alpha\in C(\sigma)}(-1)^{|\sigma|-|\alpha|}\vec{B}_{\alpha}(y)(m_{Q}\vec{B})_{(\sigma-\alpha)^{t}}\right)\mathcal{R}_{Q,p,U}^{-1}\right\Vert ^{p}dy\right)^{\frac{1}{p}}\\
\|\vec{B}\|_{BMO_{V,U,\sigma,2}^{p}} & =\sup_{Q}\left(\fint_{Q}\left\Vert U^{\frac{-1}{p}}(y)\left(\sum_{\alpha\in C(\sigma)}(-1)^{|\sigma|-|\alpha|}\vec{B}_{\alpha}(y)(m_{Q}\vec{B})_{(\sigma-\alpha)^{t}}\right)^{*}(\mathcal{R}_{Q,p,V}')^{-1}\right\Vert ^{p'}dy\right)^{\frac{1}{p'}}\\
\|\vec{B}\|_{BMO_{V,U,\sigma,1}^{p,*}} & =\sup_{Q}\left(\fint_{Q}\left\Vert V^{\frac{1}{p}}(y)\left(\sum_{\alpha\in C(\sigma)}(-1)^{|\sigma|-|\alpha|}(m_{Q}\vec{B})_{\alpha}\vec{B}_{(\sigma-\alpha)^{t}}(y)\right)\mathcal{R}_{Q,p,U}^{-1}\right\Vert ^{p}dy\right)^{\frac{1}{p}}\\
\|\vec{B}\|_{BMO_{V,U,\sigma,2}^{p,*}} & =\sup_{Q}\left(\fint_{Q}\left\Vert U^{\frac{-1}{p}}(y)\left(\sum_{\alpha\in C(\sigma)}(-1)^{|\sigma|-|\alpha|}(m_{Q}\vec{B})_{\alpha}\vec{B}_{(\sigma-\alpha)^{t}}(y)\right)^{*}(\mathcal{R}_{Q,p,V}')^{-1}\right\Vert ^{p'}dy\right)^{\frac{1}{p'}}\\
\|\vec{B}\|_{\widetilde{BMO}_{V,U,\sigma}^{p}} & =\sup_{Q}\fint_{Q}\left(\fint_{Q}\|V^{\frac{1}{p}}(x)\left(\sum_{\alpha\in C(\sigma)}(-1)^{|\sigma|-|\alpha|}\vec{B}_{\alpha}(x)\vec{B}_{(\sigma-\alpha)^{t}}(y)\right)U^{\frac{-1}{p}}(y)\|^{p'}dy\right)^{\frac{p}{p'}}dx\\
\|\vec{B}\|_{\widetilde{BMO}_{V,U,\sigma}^{p,*}} & =\sup_{Q}\fint_{Q}\left(\fint_{Q}\left\Vert U^{\frac{-1}{p}}(y)\left(\sum_{\alpha\in C(\sigma)}(-1)^{|\sigma|-|\alpha|}\vec{B}_{\alpha}(y)\vec{B}_{(\sigma-\alpha)^{t}}(x)\right)^{*}V^{\frac{1}{p}}(x)\right\Vert ^{p'} \, dy\right)^{\frac{p}{p'}} \, dx.
\end{align*}
 For $\sigma=\widetilde{\sigma}_{2^{m}}$, we will typically write
$m$ instead of $\widetilde{\sigma}_{2^{m}}$ in the previous definition
and in \eqref{def:BMO} and \eqref{def_BMO_red}.
\end{definition}

The aforementioned corollary from \cite{Cit2SaMWS} shows that, for
$m=1$, the ${BMO}_{V,U,\sigma}^{p}$, $BMO_{V,U,\sigma,1}^{p}$,
$BMO_{V,U,\sigma,2}^{p}$ and $\widetilde{BMO}_{V,U,\sigma}^{p}$
norms are equivalent, with constants possibly dependent on the $A_{p}$
norms of the corresponding weights. It therefore seems logical that
an equivalence, at least to some extent, is also satisfied for each
case in the general scenario. We were, however, unsuccessful in obtaining
a full positive result. In the following lines we gather some the
relationships that we have managed to obtain.

\begin{proposition}\label{prop:relBMO1-BMO2}
If $U,V$ are matrix weights, and $1<p<\infty$, then 
\[
\|\vec{B}\|_{BMO_{V,U,\sigma,1}^{p}}=\|\vec{B}^{*}\|_{BMO_{U^{\frac{-p'}{p}},V^{\frac{-p'}{p}},\sigma,2}^{p',*}}\qquad\text{and}\qquad\|\vec{B}\|_{BMO_{V,U,\sigma,2}^{p}}=\|\vec{B}^{*}\|_{BMO_{U^{\frac{-p'}{p}},V^{\frac{-p'}{p}},\sigma,1}^{p',*}.}
\]
\end{proposition}

\begin{proof}
Note that
\begin{align*}
 & \left(\fint_{Q}\left\Vert V^{\frac{1}{p}}(y)\left(\sum_{\alpha\in C(\sigma)}(-1)^{|\sigma|-|\alpha|}\vec{B}_{\alpha}(y)(m_{Q}\vec{B})_{(\sigma-\alpha)^{t}}\right)\mathcal{R}_{Q,p,U}^{-1}\right\Vert ^{p}dy\right)^{\frac{1}{p}}\\
 & =\left(\fint_{Q}\left\Vert (V^{\frac{-p'}{p}})^{\frac{-1}{p'}}(y)\left(\sum_{\alpha\in C(\sigma)}(-1)^{|\sigma|-|\alpha|}\vec{B}_{\alpha}(y)(m_{Q}\vec{B})_{(\sigma-\alpha)^{t}}\right)^{**}(\mathcal{R}'_{Q,p',U^{\frac{-p'}{p}}})^{-1}\right\Vert ^{(p')'}dy\right)^{\frac{1}{(p')'}}\\
 & =\left(\fint_{Q}\left\Vert (V^{\frac{-p'}{p}})^{\frac{-1}{p'}}(y)\left(\sum_{\alpha\in C(\sigma)}(-1)^{|\sigma|-|\alpha|}(m_{Q}\vec{B}^{*})_{\alpha}(\vec{B}^{*})_{(\sigma-\alpha)^{t}}(y)\right)^{*}(\mathcal{R}'_{Q,p',U^{\frac{-p'}{p}}})^{-1}\right\Vert ^{(p')'}dy\right)^{\frac{1}{(p')'}},
\end{align*}
so $\|\vec{B}\|_{BMO_{V,U,\sigma,1}^{p}}=\|\vec{B}^{*}\|_{BMO_{U^{\frac{-p'}{p}},V^{\frac{-p'}{p}},\sigma,2}^{p',*}}$.
The other case is analogous.
\end{proof}
A direct consequence from propositions \ref{prop:relBMO-BMO*} and \ref{prop:relBMO1-BMO2} is that showing the
equivalence between $BMO_{V,U,\sigma}^{p}$ and $BMO_{V,U,\sigma,2}^{p}$
is equivalent to showing the equivalence between $BMO_{V,U,\sigma}^{p,*}$
and $BMO_{V,U,\sigma,1}^{p,*}$, and the same applies to $BMO_{V,U,\sigma}^{p}$,
$BMO_{V,U,\sigma,1}^{p}$ and $BMO_{V,U,\sigma}^{p,*}$,$BMO_{V,U,\sigma,2}^{p,*}$.

The next result will provide two inequalities that relate the $\widetilde{BMO}_{V,U,\sigma}^{p}$
and a sum of untilded $BMO$ norms.
\begin{proposition}
\label{cor:BmotildeBMO} If $\vec{B}$ is a vector of locally integrable
matrices, and $\theta\in C(m)$, then:
\begin{enumerate}
\item If $U,V,(U,V)\in A_{p}$, 
\[
\begin{split}\|\vec{B}\|_{\widetilde{BMO}_{V,U,\theta}^{p}}\lesssim[V]_{A_{p}}^{\frac{1}{p}}{\|\vec{B}\|_{BMO_{V,U,\theta,2}^{p,*}}^{{p}}}+\sum_{\underrel{\sigma\neq\varnothing,\theta}{\sigma\in C(\theta)}}[U,V]_{A_{p}}^{\frac{1}{p}}\|\vec{B}\|_{BMO_{V,U,\sigma,1}^{p}}^{{p}}{\|\vec{B}\|_{BMO_{V,U,\theta-\sigma,2}^{p,*}}^{{p}}}+\|\vec{B}\|_{{BMO}_{V,U,\theta,1}^{p}}^{{p}}[U]_{A_{p}}^{\frac{1}{p}}.\end{split}
\]
\item If $(U,V)\in A_{p}$, 
\[
\begin{split}\|\vec{B}\|_{\widetilde{BMO}_{V,U,\theta}^{p}}\lesssim[U,V]_{A_{p}}^{\frac{1}{p}}\sum_{{\sigma\in C(\theta)}}\|\vec{B}\|_{BMO_{V,U,\sigma,1}^{p}}^{{p}}\|\vec{B}\|_{BMO_{V,U,\theta-\sigma,2}^{p,*}}^{{p}}.\end{split}
\]
\end{enumerate}
\end{proposition}

\begin{proof}
Observe that, for any cube $Q$, applying Lemma \ref{lemma:perm2}
in the first equality, 
\begin{align*}
 & {\fint_{Q}\left[\fint_{Q}\|V^{\frac{1}{p}}(x)\left(\sum_{\sigma\in C(\theta)}(-1)^{|\theta|-|\sigma|}\vec{B}_{\sigma}(x)\vec{B}_{(\theta-\sigma)^{t}}(y)\right)U^{\frac{-1}{p}}(y)\|^{p'}dy\right]^{\frac{p}{p'}}dx}\\
 & {=\fint_{Q}\left[\fint_{Q}\|V^{\frac{1}{p}}(x)\left(\sum_{\sigma\in C(\theta)}\left(\sum_{\alpha\in C(\sigma)}(-1)^{|\theta|-|\alpha|}\vec{B}_{\alpha}(x)(m_{Q}\vec{B})_{(\sigma-\alpha)^{t}}\right)\right.\right.}\\
 & \;\;\;\;{\times\;\left.\left.\left(\sum_{\beta\in C(\sigma^{c})}(-1)^{|\theta|-|\beta|}(m_{Q}\vec{B})_{\beta}\vec{B}_{(\sigma^{c}-\beta)^{t}}(y)\right)\right)U^{\frac{-1}{p}}(y)\|^{p'}dy\right]^{\frac{p}{p'}}dx}\\
 & {\leq\sum_{\underrel{\sigma\neq\varnothing,\theta}{\sigma\in C(\theta)}}\fint_{Q}\left[\fint_{Q}\|V^{\frac{1}{p}}(x)\left(\sum_{\alpha\in C(\sigma)}(-1)^{|\theta|-|\alpha|}\vec{B}_{\alpha}(x)(m_{Q}\vec{B})_{(\sigma-\alpha)^{t}}\right)\mathcal{R}_{Q,p,U}^{-1}\|^{p'}\|\mathcal{R}_{Q,p,U}\mathcal{R}_{Q,p,V}'\|^{p'}\right.}\\
 & \;\;\;\;{\times\;\left.\|(\mathcal{R}_{Q,p,V}')^{-1}\left(\sum_{\beta\in C(\sigma^{c})}(-1)^{|\theta|-|\beta|}(m_{Q}\vec{B})_{\beta}\vec{B}_{(\sigma^{c}-\beta)^{t}}(y)\right)U^{\frac{-1}{p}}(y)\|^{p'}dy\right]^{\frac{p}{p'}}dx}\\
 & {\;\;\;\;+\fint_{Q}\left[\fint_{Q}\|V^{\frac{1}{p}}(x)\left(\sum_{\beta\in C(\theta)}(-1)^{|\theta|-|\beta|}(m_{Q}\vec{B})_{\beta}\vec{B}_{(\theta-\beta)^{t}}(y)\right)U^{\frac{-1}{p}}(y)\|^{p'}dy\right]^{\frac{p}{p'}}dx}\\
 & {\;\;\;\;+\fint_{Q}\left[\fint_{Q}\|V^{\frac{1}{p}}(x)\left(\sum_{\alpha\in C(\theta)}(-1)^{|\theta|-|\alpha|}\vec{B}_{\alpha}(x)(m_{Q}\vec{B})_{(\theta-\alpha)^{t}}\right)U^{\frac{-1}{p}}(y)\|^{p'}dy\right]^{\frac{p}{p'}}dx}\\
 & {=\sum_{\underrel{\sigma\neq\varnothing,\theta}{\sigma\in C(\theta)}}C_{\sigma}+C_{\varnothing}+C_{\theta},}
\end{align*}
where in the first equality we applied Lemma \ref{lemma:perm2}. 
We will now control all the $C_{\sigma}$ separately. First, we will
observe that 
\begin{align*}
C_{\theta} & \leq\fint_{Q}\left(\fint_{Q}\|V^{\frac{1}{p}}(x)\left(\sum_{\alpha\in C(\theta)}(-1)^{|\theta|-|\alpha|}\vec{B}_{\alpha}(x)(m_{Q}\vec{B})_{(\theta-\alpha)^{t}}\right)\mathcal{R}_{Q,p,U}^{-1}\|^{p'}\|\mathcal{R}_{Q,p,U}U^{\frac{-1}{p}}(y)\|^{p'}dy\right)^{\frac{p}{p'}}dx\\
 & =\fint_{Q}\|V^{\frac{1}{p}}(x)\left(\sum_{\alpha\in C(\theta)}(-1)^{|\theta|-|\alpha|}\vec{B}_{\alpha}(x)(m_{Q}\vec{B})_{(\theta-\alpha)^{t}}\right)\mathcal{R}_{Q,p,U}^{-1}\|^{p}dx\left(\fint_{Q}\|\mathcal{R}_{Q,p,U}U^{\frac{-1}{p}}(y)\|^{p'}dy\right)^{\frac{p}{p'}}\\
 & \lesssim\|\vec{B}\|_{{BMO}_{V,U,\theta,1}^{p}}^{p}[U]_{A_{p}}
\end{align*}
Similarly, 
\[
\begin{split}C_{\varnothing}\leq[V]_{A_{p}}{\|\vec{B}\|_{BMO_{V,U,\theta,2}^{p,*}}^{p}}\end{split}
\]

As for the rest of the $C_{\sigma}$, we have that 
\begin{align*}
{\fint_{Q}} & {\left(\fint_{Q}\|V^{\frac{1}{p}}(x)\left(\sum_{\alpha\in C(\sigma)}(-1)^{|\theta|-|\alpha|}\vec{B}_{\alpha}(x)(m_{Q}\vec{B})_{(\sigma-\alpha)^{t}}\right)\mathcal{R}_{Q,p,U}^{-1}\|^{p'}\|\mathcal{R}_{Q,p,U}\mathcal{R}_{Q,p,V}'\|^{p'}\right.}\\
 & \;\;\;\;{\times\;\left.\|(\mathcal{R}_{Q,p,V}')^{-1}\left(\sum_{\beta\in C(\sigma^{c})}(-1)^{|\theta|-|\beta|}(m_{Q}\vec{B})_{\beta}\vec{B}_{(\sigma^{c}-\beta)^{t}}(y)\right)U^{\frac{-1}{p}}(y)\|^{p'}dy\right)^{\frac{p}{p'}}dx}\\
 & {=\|\mathcal{R}_{Q,p,U}\mathcal{R}_{Q,p,V}'\|^{p}\fint_{Q}\|V^{\frac{1}{p}}(x)\left(\sum_{\alpha\in C(\sigma)}(-1)^{|\theta|-|\alpha|}\vec{B}_{\alpha}(x)(m_{Q}\vec{B})_{(\sigma-\alpha)^{t}}\right)\mathcal{R}_{Q,p,U}^{-1}\|^{p}dx}\\
 & \;\;\;\;{\times\;\left(\fint_{Q}\|(\mathcal{R}_{Q,p,V}')^{-1}\left(\sum_{\beta\in C(\sigma^{c})}(-1)^{|\theta|-|\beta|}(m_{Q}\vec{B})_{\beta}\vec{B}_{(\sigma^{c}-\beta)^{t}}(y)\right)U^{\frac{-1}{p}}(y)\|^{p'}dy\right)^{\frac{p}{p'}}}\\
 & {\simeq[U,V]_{A_{p}}\|\vec{B}\|_{BMO_{V,U,\sigma,1}^{p}}^{p}{\|\vec{B}\|_{BMO_{V,U,\sigma^{c},2}^{p,*}}^{p}}}.
\end{align*}
Note this is also true for $\sigma=\varnothing$, $\theta$.

Once this is done, standard arguments allow us to conclude that 
\[
\begin{split}\|\vec{B}\|_{\widetilde{BMO}_{V,U,\theta}^{p}}\lesssim[V]_{A_{p}}^{\frac{1}{p}}{\|\vec{B}\|_{BMO_{V,U,\theta,2}^{p,*}}^{{p}}}+\sum_{\underrel{\sigma\neq\varnothing,\theta}{\sigma\in C(\theta)}}[U,V]_{A_{p}}^{\frac{1}{p}}\|\vec{B}\|_{BMO_{V,U,\sigma,1}^{p}}^{{p}}{\|\vec{B}\|_{BMO_{V,U,\sigma^{c},2}^{p,*}}^{{p}}}+\|\vec{B}\|_{{BMO}_{V,U,\theta,1}^{p}}^{{p}}[U]_{A_{p}}^{\frac{1}{p}}\end{split}
\]
and 
\[
\begin{split}\|\vec{B}\|_{\widetilde{BMO}_{V,U,\theta}^{p}}\lesssim[U,V]_{A_{p}}^{\frac{1}{p}}\sum_{{\sigma\in C(\theta)}}\|\vec{B}\|_{BMO_{V,U,\sigma,1}^{p}}^{{p}}{\|\vec{B}\|_{BMO_{V,U,\sigma^{c},2}^{p,*}}^{{p}}},\end{split}
\]
so the desired inequalities are proven. 
\end{proof}

Notice that the first inequality in this proposition
is more precise than the second one, but the latter is more readily
applicable. These inequalities are not completely satisfactory,
but they suggest that a potential equivalence between the proposed
$BMO$ norms may have to involve terms with tuples of lower order.

Lastly, let us just observe that, if $(U,V)\in A_{p}$, it is easy
to show that $\|\vec{B}\|_{\widetilde{BMO}_{V,V,\sigma}^{p}},\|\vec{B}\|_{\widetilde{BMO}_{U,U,\sigma}^{p}}\lesssim[U,V]_{A_{p}}^{\frac{1}{p}}\|\vec{B}\|_{\widetilde{BMO}_{V,U,\sigma}^{p}}$
and $\|\vec{B}\|_{\widetilde{BMO}_{U,V,\sigma}^{p}}\lesssim[U,V]_{A_{p}}^{\frac{2}{p}}\|\vec{B}\|_{\widetilde{BMO}_{V,U,\sigma}^{p}}$. These inequalities will be useful to prove Theorem \ref{th:strTocho}. 

\subsection{A particular case of interest}

In the first part of this section we devoted our efforts to studying
the relations between several $BMO$ classes in the most general scenario.
However, even though we were able to control the $\widetilde{BMO}$ norms by $BMO_{1,2}$ norms, we were unable to obtain a reciprocal control, nor could we relate those norms to the $BMO$ norm we began defining.
In order to provide a positive result relating those norms, we shall restrict ourselves to a particular case, namely, the case $\vec{B}=b(I_{n},\dots,I_{n})$
where $b$ is a scalar function. As we will soon see, in this case
we will be able to relate each couple of definitions. For this purpose,
we begin noting that following the notation above, for $\vec{B}=b(I_{n},\dots,I_{n})$, it is
not hard to check that 
\begin{align*}
a) \, \|\vec{B}\|_{BMO_{V,U,\sigma}^{p}} & \simeq\sup_{Q}\fint_{Q}\|\mathcal{R}_{Q,p,V}\mathcal{R}_{Q,p,U}^{-1}\|^{{\frac{1}{|\sigma|}}}|b(x)-b_{Q}|\,dx,\\
b) \, \|\vec{B}\|_{BMO_{V,U,\sigma,1}^{p}} & =\left(\sup_{Q}\fint_{Q}\|V^{1/p}(x)\mathcal{R}_{Q,p,U}^{-1}\|^{p}|b(x)-b_{Q}|^{p|\sigma|}\,dx\right)^{1/p},\\
c) \, \|\vec{B}\|_{BMO_{V,U,\sigma,2}^{p}} & =\left(\sup_{Q}\fint_{Q}\|U^{-1/p}(x)(\mathcal{R}'_{Q,p,V})^{-1}\|^{p'}|b(x)-b_{Q}|^{p'|\sigma|}\,dx\right)^{1/p'} ,\\
d) \, \|\vec{B}\|_{\widetilde{BMO}_{V,U,\sigma}^{p}} & =\sup_{Q}\fint_{Q}\left(\fint_{Q}\|V^{1/p}(x)U^{-1/p}(y)\|^{p'}|b(x)-b(y)|^{p|\sigma|}\,dy\right)^{p'/p} \, dx. \\ 
\end{align*}
Note that the $*$ variants are not needed anymore, since, the matrices
involved are self-adjoint and the function $b$ is scalar. What we
have just written motivates providing the following definition for the sake of simplifying the notation.
\begin{definition} \label{BMOScalarMatrixDef}
Let $j\in\mathbb{N}$, and $U,V$ matrix weights. We define
\begin{align*}
\|b\|_{BMO_{V,U,j}^{p}} & =\sup_{Q}\fint_{Q}\|\mathcal{R}_{Q,p,V}\mathcal{R}_{Q,p,U}^{-1}\|^{{\frac{1}{j}}}|b(x)-b_{Q}|\,dx,\\
\|b\|_{BMO_{V,U,j,1}^{p}} & =\left(\sup_{Q}\fint_{Q}\|V^{1/p}(x)\mathcal{R}_{Q,p,U}^{-1}\|^{p}|b(x)-b_{Q}|^{pj}\,dx\right)^{1/p},\\
\|b\|_{BMO_{V,U,j,2}^{p}} & =\left(\sup_{Q}\fint_{Q}\|U^{-1/p}(x)(\mathcal{R}'_{Q,p,V})^{-1}\|^{p'}|b(x)-b_{Q}|^{p'j}\,dx\right)^{1/p'},\\
\|b\|_{\widetilde{BMO}_{V,U,j}^{p}} & =\sup_{Q}\fint_{Q}\left(\fint_{Q}\|V^{1/p}(x)U^{-1/p}(y)\|^{p'}|b(x)-b(y)|^{p'j}\,dy\right)^{p/p'}dx \\
\|b\|_{\widetilde{BMO}_{V,U,j,2}^{p}} & = \sup_Q \fint_{Q}\left(\fint_{Q}\|V^{\frac{1}{p}}(x)U^{-\frac{1}{p}}(y)\|^{p}|b(x)-b(y)|^{jp}\,dx\right)^{\frac{p'}{p}}\,dy
\end{align*}
\end{definition}
The following results shows that all the definitions above give rise to the same $BMO$ type space.
\begin{theorem}
\label{StrongHigherJN} If $U,V$ are matrix weights A$_{p}$ weights
and $j\in\mathbb{N}$ then $$\|b\|_{BMO_{V,U,j}^{p}}^{jp} \approx \|b\|_{BMO_{V,U,j,1}^{p}} ^p \approx \|b\|_{BMO_{V,U,j,2}^{p}} ^{p'} 
\approx \|b\|_{\widetilde{BMO}_{V,U,j}^{p}}  \approx \|b\|_{\widetilde{BMO}_{V,U,j, 2}^{p'}} $$
\end{theorem}

Note that in the case $V = v I_{n \times n}$ and $U = u I_{n \times n}$ for $u, v $ being scalar A${}_p$ weights, we have that $\|b\|_{BMO_{V,U,j}^{p}} = \|b\|_{BMO(\nu^\frac{1}{j})}$ where $\nu=u^{\frac{1}{p}}v^{-\frac{1}{p}}$ and $BMO(\nu^\frac{1}{j})$ is the by now standard Bloom BMO (see \cite{HW18, LOR19} for example) since 
if $u,v\in\text{A}_{p}$ then $u,v\in A_{jp}$ and thus 
\[
\|\mathcal{R}_{Q,p,V}\mathcal{R}_{Q,p,U}^{-1}\|^{\frac{1}{j}}=\langle v\rangle_{Q}^{\frac{1}{jp}}\langle u\rangle_{Q}^{-\frac{1}{jp}}\approx\left\{ \langle u\rangle_{Q}^{\frac{1}{jp}}\langle v^{-\frac{(pj)'}{pj}}\rangle_{Q}^{\frac{1}{(jp)'}}\right\} ^{-1}\approx\frac{1}{\langle\nu^{\frac{1}{j}}\rangle_{Q}}.
\]

The proof is similar to the proof of the John-Nirenberg
Theorem from \cite{Cit2SaMWS} but also clarifies and simplifies some arguments from \cite{Cit2SaMWS}. The proof of Theorem \ref{StrongHigherJN}
will require a sequence of preliminary results, the first being interesting
on its own right and is a $p\neq2$ generalization of Lemma $3.5$
in \cite{TV}.
\begin{proposition}
    
\label{TVLem} Let $W$ be an $n\times n$ matrix A$_{p}$ weight
and $A$ be any constant invertible $n\times n$ matrix. If 
\[
\mathcal{W}=(A^{*}W^{\frac{2}{p}}A)^{\frac{p}{2}},
\]
then $\mathcal{W}$ is a matrix A$_{p}$ weight with $[\mathcal{W}]_{\text{A}_{p}}=[W]_{\text{A}_{p}}$. 
\end{proposition}

\begin{proof}
Clearly 
\[
\mathcal{W}^{\frac{1}{p}}=(A^{*}W^{\frac{2}{p}}A)^{\frac{1}{2}}=\left[(W^{\frac{1}{p}}A)^{*}(W^{\frac{1}{p}}A)\right]^{\frac{1}{2}}
\]
so that by the polar decomposition there exists a unitary $\mathcal{U}$
with $\mathcal{U}\mathcal{W}^{\frac{1}{p}}=W^{\frac{1}{p}}A$. Thus, we have
that 
\[
\|\mathcal{W}^{\frac{1}{p}}(x)\mathcal{W}^{-\frac{1}{p}}(y)\|=\|(\mathcal{U}(x)W^{\frac{1}{p}}(x)A)(A^{-1}W^{-\frac{1}{p}}(y)\mathcal{U}(y))\|=\|W^{\frac{1}{p}}(x)W^{-\frac{1}{p}}(y)\|.
\]
\end{proof}
We will need the following simple result that is a special case of
Theorem $2.2$ in \cite{I}.
\begin{theorem}
\label{embedthm} Let $U$ be a matrix A$_{p}$ weight and let $\{a_{Q}\}_{Q\in\mathcal{D}}$
be a nonnegative Carleson sequence, meaning that 
\[
\|\{a_{Q}\}\|^{2}=\sup_{J\in\mathcal{D}}\frac{1}{|J|}\sum_{Q\in\mathcal{D}(J)}a_{Q}^{2}<\infty.
\]
Then for any $\vec{f}\in L^{p}$ we have 
\begin{equation}
\int_{\mathbb{R}^{d}}\left(\sum_{Q\in\mathcal{D}}\frac{a_{Q}^{2}}{|Q|}\langle|\mathcal{R}_{Q,p,U}U^{-\frac{1}{p}}\vec{f}|\rangle_{Q}^{2}\chi_{Q}(t)\right)^{\frac{p}{2}}\,dt\lesssim\|\{a_{Q}\}\|^{p}\|\vec{f}\|_{L^{p}}^{p}.\label{embedineq}
\end{equation}
\end{theorem}

\begin{lemma} \label{VIlem} For a fixed cube $I$ and $V_{I}(x)$ defined
by 
\[
V_{I}(x)=\left(\mathcal{R}_{I,p,U}^{-1}V^{\frac{2}{p}}(x)\mathcal{R}_{I,p,U}^{-1}\right)^{\frac{p}{2}},
\]
we have that $V_{I}$ is a matrix $A_{p}$ weight with the same $A_{p}$
characteristic as $V$, and, for any $q>1$, we have 
\[
\|\mathcal{R}_{Q,q,V_{I}}\|{\approx}\|\mathcal{R}_{Q,p,V}\mathcal{R}_{I,p,U}^{-1}\|^{\frac{p}{q}}.
\]
\end{lemma}
\begin{proof}
The first part follows immediately from the Proposition \ref{TVLem},
since $\mathcal{R}_{I,p,U}$ is self-adjoint. On the other hand, for
any orthonormal basis $\{\vec{e_{\ell}}\}_{\ell=1}^{n}$ we have 
\begin{align*}
\|\mathcal{R}_{Q,q,V_{I}}\| & \approx\sum_{\ell=1}^{n}|\mathcal{R}_{Q,q,V_{I}}\vec{e_{\ell}}|\approx\sum_{\ell=1}^{n}\left(\fint_{Q}|V_{I}^{\frac{1}{q}}(x)\vec{e_{\ell}}|^{q}\,dx\right)^{\frac{1}{q}}\mathrel{\approx}\mathinner{\left(\fint_{Q}\|V_{I}^{\frac{2}{p}}(x)\|^{\frac{p}{2}}\,dx\right)^{\frac{1}{q}}}\\
 & =\left(\fint_{Q}\|\mathcal{R}_{I,p,U}^{-1}V^{\frac{2}{p}}(x)\mathcal{R}_{I,p,U}^{-1}\|^{\frac{p}{2}}\right)^{\frac{1}{q}}\approx\left(\fint_{Q}\|V^{\frac{1}{p}}(x)\mathcal{R}_{I,p,U}^{-1}\|^{p}\,dx\right)^{\frac{1}{q}}\approx\|\mathcal{R}_{Q,p,V}\mathcal{R}_{I,p,U}^{-1}\|^{\frac{p}{q}}.
\end{align*}
\end{proof}
In what follows $\{h_{Q}^{\varepsilon}:Q\in\mathcal{D},\varepsilon\in\S\}$
refers to any Haar basis for $L^{2}(\mathbb{R}^{d})$ with $|\S|=2^{d-1}$.
Note that the ``Triebel-Lizorkin imbedding theorem'' (i.e. matrix
weighted dyadic Littlewood-Paley Theory, see \cite{isralowitz2015,NT,Volberg-S}) states
that if $U$ is a matrix $A_{p}$ weight and $\vec{f}_{Q}^{\varepsilon}$
is the Haar coefficient of $\vec{f}$ then 
\begin{equation}
\|\vec{f}\|_{L^{p}(U)}\approx\left(\int_{\mathbb{R}^{d}}\left(\sum_{\substack{Q\in\mathcal{D}(I)\\
\varepsilon\in\S
}
}\frac{|\mathcal{R}_{Q,p,U}\vec{f}_{Q}^{\varepsilon}|^{2}\chi_Q(x)}{|Q|}\right)^{\frac{p}{2}}\,dx\right)^{\frac{1}{p}}.\label{TriebLizThm}
\end{equation}

\begin{lemma} \label{HigherJN1} For any cube $J$ we have 
\begin{equation}
\sup_{I\in\mathcal{D}(J)}\,\fint_{I}\|V^{\frac{1}{p}}(x)\mathcal{R}_{I,p,U}^{-1}\|^{p}|b(y)-\langle b\rangle_{I}|^{jp}\,dy\lesssim\sup_{I\in\mathcal{D}(J)}\,\frac{1}{|I|}\sum_{\substack{Q\in\mathcal{D}(I)\\
\varepsilon\in\S
}
}\|\mathcal{R}_{Q,p,V}\mathcal{R}_{Q,p,U}^{-1}\|^{\frac{2}{j}}|b_{Q}^{\epsilon}|^{2}.\label{HigherJN1Ineq}
\end{equation}
\end{lemma}
\begin{proof}
Let $U_{I}(x)=\left(\mathcal{R}_{I,p,U}^{-1}U^{\frac{2}{p}}(x)\mathcal{R}_{I,p,U}^{-1}\right)^{\frac{p}{2}}$
and notice that 
\begin{align}
\int_{\mathbb{R}^{d}}\|U_{I}^{\frac{1}{jp}}(y)\chi_{{I}}(y)\|^{jp}\,dy & =\int_{I}\|U_{I}^{\frac{1}{p}}(y)\|^{p}\,dy=\int_{I}\|\mathcal{R}_{I,p,U}^{-1}U^{\frac{2}{p}}(y)\mathcal{R}_{I,p,U}^{-1}\|^{\frac{p}{2}}\,dy\nonumber \\
 & =\int_{I}\|U^{\frac{1}{p}}(y)\mathcal{R}_{I,p,U}^{-1}\|^{p}\,dy\approx|I|.\label{HigherJN1comp}
\end{align}

Similarly let $V_{I}(x)=\left(\mathcal{R}_{I,p,U}^{-1}V^{\frac{2}{p}}(x)\mathcal{R}_{I,p,U}^{-1}\right)^{\frac{p}{2}}$
and note that $\|V^{\frac{1}{p}}(y)\mathcal{R}_{I,p,U}^{-1}\|^{p}=\|V_{I}(y)\|$ 
and $V_{I}\in A_{p}\subseteq A_{jp}$ with $A_{jp}$ characteristic
independent of $I$ (see Proposition $5.5$ in \cite{G}). Thus, by the
matrix weighted Triebel-Lizorkin imbedding theorem and Theorem \ref{embedthm}
with $a_{Q}=\|\mathcal{R}_{Q,p,V}\mathcal{R}_{Q,p,U}^{-1}\|^{\frac{1}{j}}|b_{Q}^{\varepsilon}|$
and $U=U_{I}$ with $I\in\mathcal{D}(J)$ fixed, we have that
\begin{align*}
\fint_{I} & \|V^{\frac{1}{p}}(y)\mathcal{R}_{I,p,U}\|^{p}|b(y)-\langle b\rangle_{I}|^{jp}\,dy\\
 & \approx\sum_{\ell=1}^{n}\fint_{I}|V_{I}^{\frac{1}{jp}}(y)(b(y)-\langle b\rangle_{I})\vec{e}_{\ell}|^{jp}\,dy\\
 & \lesssim\sum_{\ell=1}^{n}\frac{1}{|I|}\int_{\mathbb{R}^{d}}\left(\sum_{\substack{Q\in\mathcal{D}(I)\\
\varepsilon\in\S
}
}\frac{|\mathcal{R}_{Q,jp,V_{I}}b_{Q}^{\varepsilon}\vec{e}_{\ell}|^{2}}{|Q|}\chi_{{Q}}(y)\right)^{\frac{jp}{2}}\,dy\\
 & \approx\frac{1}{|I|}\int_{\mathbb{R}^{d}}\left(\sum_{\substack{Q\in\mathcal{D}(I)\\
\varepsilon\in\S
}
}\frac{\|\mathcal{R}_{Q,p,V}\mathcal{R}_{I,p,U}^{-1}\|^{\frac{2}{j}}|b_{Q}^{\varepsilon}|^{2}}{|Q|}\chi_{{Q}}(y)\right)^{\frac{jp}{2}}\,dy\\
 & \lesssim\frac{1}{|I|}\int_{\mathbb{R}^{d}}\left(\sum_{\substack{Q\in\mathcal{D}(I)\\
\varepsilon\in\S
}
}\frac{(\|\mathcal{R}_{Q,p,V}\mathcal{R}_{Q,p,U}^{-1}\|^{\frac{1}{j}}|b_{Q}^{\varepsilon}|)^{2}\|\mathcal{R}_{Q,p,U}^{-1}\mathcal{R}_{I,p,U}\|^{\frac{2}{j}}}{|Q|}\chi_{{Q}}(y)\right)^{\frac{jp}{2}}\,dy\\
 & \lesssim \sum_{\ell=1}^{n} \frac{1}{|I|}\int_{\mathbb{R}^{d}}\left(\sum_{\substack{Q\in\mathcal{D}(I)\\
\varepsilon\in\S
}
}\frac{(\|\mathcal{R}_{Q,p,V}\mathcal{R}_{Q,p,U}^{-1}\|^{\frac{1}{j}}|b_{Q}^{\varepsilon}|)^{2}}{|Q|}\langle\|[(\mathcal{R}_{Q,jp,U_{I}}U_{I}^{-\frac{1}{jp}})U_{I}^{\frac{1}{jp}} \vec{e}_\ell \chi_{{I}}]\|\rangle_{Q}^{2}\chi_{{Q}}(y)\right)^{\frac{jp}{2}}\,dy.\\
 & \leq\left(\frac{1}{|I|}\int_{\mathbb{R}^{d}}\|U_{I}^{\frac{1}{jp}}(y)\chi_{{I}}(y)\|^{jp}\,dy\right)\sup_{I\in\mathcal{D}(J)}\,\frac{1}{|I|}\sum_{\substack{Q\in\mathcal{D}(I)\\
\varepsilon\in\S
}
}\|\mathcal{R}_{Q,p,V}\mathcal{R}_{Q,p,U}^{-1}\|^{\frac{2}{j}}|b_{Q}^{\epsilon}|^{2}\\
 & \lesssim\sup_{I\in\mathcal{D}(J)}\,\frac{1}{|I|}\sum_{\substack{Q\in\mathcal{D}(I)\\
\varepsilon\in\S
}
}\|\mathcal{R}_{Q,p,V}\mathcal{R}_{Q,p,U}^{-1}\|^{\frac{2}{j}}|b_{Q}^{\epsilon}|^{2}
\end{align*}
thanks to \eqref{HigherJN1comp}.
\end{proof}
\begin{lemma} \label{HigherJN2} For $\epsilon'>0$ small enough we
have that 
\[
\left(\fint_{I}\|\mathcal{R}_{I,jp,V_{I}}(b(y)-\langle b\rangle_{I})\|^{1+\epsilon'}\,dy\right)^{\frac{1}{1+\epsilon'}}\lesssim\left(\fint_{I}\|V^{\frac{1}{p}}(y)\mathcal{R}_{I,p,U}^{-1}\|^{p}|b(y)-\langle b\rangle_{I}|^{jp}\,dy\right)^{\frac{1}{jp}}.
\]
\end{lemma}
\begin{proof}
As was mentioned before, 
\[
\fint_{I}\|V^{\frac{1}{p}}(y)\mathcal{R}_{I,p,U}^{-1}\|^{p}|b(y)-\langle b\rangle_{I}|^{jp}\,dy\approx\fint_{I}\|V_{I}^{\frac{1}{jp}}(y)(b(y)-\langle b\rangle_{I})\|^{jp}\,dy.
\]
To that end, since again $V_{I}\in A_{p}\subseteq A_{jp}$ with constant independent
of $I$ we have by the reverse Hölder inequality 
\begin{align*}
 & \left(\fint_{I}\|\mathcal{R}_{I,jp,V_{I}}(b(y)-\langle b\rangle_{I})\|^{1+\epsilon'}\,dy\right)^{\frac{1}{1+\epsilon'}}\\
 & =\left(\fint_{I}\|\mathcal{R}_{I,jp,V_{I}}V_{I}^{-\frac{1}{jp}}(y)\|^{1+\epsilon'}\|V_{I}^{\frac{1}{jp}}(y)(b(y)-\langle b\rangle_{I})\|^{1+\epsilon'}\,dy\right)^{\frac{1}{1+\epsilon'}}\\
 & \leq\left(\fint_{I}\|\mathcal{R}_{I,jp,V_{I}}V_{I}^{-\frac{1}{jp}}(y)\|^{\frac{jp(1+\epsilon')}{jp-1-\epsilon'}}\right)^{\frac{jp-1-\epsilon'}{jp(1+\epsilon')}}\left(\fint_{I}\|V_{I}^{\frac{1}{jp}}(y)(b(y)-\langle b\rangle_{I})\|^{jp}\,dy\right)^{\frac{1}{jp}}\\
 & \lesssim\left(\fint_{I}\|V_{I}^{\frac{1}{jp}}(y)(b(y)-\langle b\rangle_{I})\|^{jp}\,dy\right)^{\frac{1}{jp}}.
\end{align*}
\end{proof}
For the next proof we need to introduce a stopping time from \cite{I}
which is a modification of the one from \cite{paraprods}. Assume that $U,V$
are a matrix $A_{p}$ weights and that $\lambda$ is large. For any
cube $I\in\mathcal{D}$, let $\J(I)$ be the collection of maximal $J\in\mathcal{D}(I)$
such that either of the two conditions 
\[
\|\mathcal{R}_{J,p,U}\mathcal{R}_{I,p,U}^{-1}\|>\lambda\ \text{ or }\ \|\mathcal{R}_{J,p,U}^{-1}\mathcal{R}_{I,p,U}\|>\lambda,
\]
or either of the two conditions 
\begin{equation}
\|\mathcal{R}_{J,p,V}\mathcal{R}_{I,p,V}^{-1}\|>\lambda\ \text{ or }\ \|\mathcal{R}_{I,p,V}\mathcal{R}_{J,p,V}^{-1}\|>\lambda
\end{equation}
are true. Also, let $\F(I)$ be the collection of dyadic subcubes
of $I$ not contained in any cube $J\in\J(I)$, so that clearly $J\in\F(J)$
for any $J\in\mathcal{D}$.

Let $\J^{0}(I):=\{I\}$ and inductively define $\J^{j}(I)$ and $\F^{j}(I)$
for $j\geq1$ by 
\[
\J^{j}(I):=\{R\in\J(Q):Q\in\J_{j-1}(I)\}
\]
and $\F^{j}(I)=\{J'\in\F(J):J\in\J^{j-1}(I)\}$. Clearly the cubes
in $\J^{j}(I)$ for $j>0$ are pairwise disjoint. Furthermore, since
$J\in\F(J)$ for any $J\in\mathcal{D}(I)$, we have that $\mathcal{D}(I)=\bigcup_{j=1}^{\infty}\mathcal{F}^{j}(I)$.
We will slightly abuse notation and write $\bigcup\J(I)$ for the
set $\bigcup_{J\in\J(I)}J$ and write $|\bigcup\J(I)|$ for $|\bigcup_{J\in\J(I)}J|$.
By easy arguments we can pick $\lambda$ depending on $U$ and $V$
so that 
\begin{equation}
\left|\bigcup\J^{j}(I)\right|\leq2^{-j}|I|\label{decay}
\end{equation}
for every $I\in\mathcal{D}$.

\begin{lemma} \label{HigherJN3} For any $0<\epsilon'\leq1$ and fixed
$J\in\mathcal{D}$ there exists $C>0$ independent of $J$ and $b$
where 
\[
\sup_{I\in\mathcal{D}(J)}\,\left(\frac{1}{|I|}\sum_{\substack{Q\in\mathcal{D}(I)\\
\varepsilon\in\S
}
}\|\mathcal{R}_{Q,p,V}\mathcal{R}_{Q,p,U}^{-1}\|^{\frac{2}{j}}|b_{Q}^{\epsilon}|^{2}\right)^{\frac{1}{2}}\leq C\sup_{I\in\mathcal{D}(J)}\left(\fint_{I}\|\mathcal{R}_{I,jp,V_{I}}(b(y)-\langle b\rangle_{I})\|^{1+\epsilon'}\,dy\right)^{\frac{1}{1+\epsilon'}}
\]
\end{lemma}
\begin{proof}
Fix $I\in\mathcal{D}(J)$. 

By the John-Nirenberg lemma and unweighted dyadic Littlewood-Paley theory, it is enough to 
prove that 

\begin{align}
\sup_{I\in\mathcal{D}(J)} & \left(\fint_{I}\left(\sum_{\substack{Q\in\mathcal{D}(I)\\
\varepsilon\in\S
}
}\frac{\|\mathcal{R}_{Q,p,V}\mathcal{R}_{Q,p,U}^{-1}\|^{\frac{2}{j}}|b_{Q}^{\epsilon}|^{2}}{|Q|}\chi_{Q}(x)\right)^{\frac{1+\epsilon'}{2}}\,dx\right)^{\frac{1}{1+\epsilon'}}\\
 & {\color{red}\mathrel{\normalcolor \leq}{\normalcolor C\sup_{I\in\mathcal{D}(J)}}\mathinner{\normalcolor \left(\fint_{I}\left(\sum_{\substack{Q\in\mathcal{D}(I)\\
\ \varepsilon\in\S
}
}\frac{\|\mathcal{R}_{I,p,V}\mathcal{R}_{I,p,U}^{-1}\|^{\frac{2}{j}}|b_{Q}^{\varepsilon}|^{2}}{|Q|}\chi_{Q}(x)\right)^{\frac{1+\epsilon'}{2}}\,dx\right)^{\frac{1}{1+\epsilon'}}}} \label{MatrixJNIneq}
\end{align} for $I\in\mathcal{D}(J)$ where $C$ is independent of $I,J$ and
$B$.

Namely, if $$F_J = \sum_{\substack{Q\in\mathcal{D}\\ \varepsilon\in\S}} \|\mathcal{R}_{Q,p,V}\mathcal{R}_{Q,p,U}^{-1}\|^{\frac{2}{j}}|\tilde{b}_{Q}^{\epsilon}|^{2} h_Q ^{\epsilon}  $$ where 

\[ \tilde{b}_{Q}^{\epsilon}  = \begin{cases} 
      {b}_{Q}^{\epsilon}  & \text{ if } Q \subseteq J\\
     0 & \text{ otherwise }       
   \end{cases}
\] then 
\begin{align*} 
& \left(\sup_{I\in\mathcal{D}(J)}\, \frac{1}{|I|}\sum_{\substack{Q\in\mathcal{D}(I)\\
\varepsilon\in\S
}
}\|\mathcal{R}_{Q,p,V}\mathcal{R}_{Q,p,U}^{-1}\|^{\frac{2}{j}}|b_{Q}^{\epsilon}|^{2} \right)^\frac12 = \left(\sup_{I\in\mathcal{D}}\, \frac{1}{|I|}\sum_{\substack{Q\in\mathcal{D}(I)\\
\varepsilon\in\S
}
}\|\mathcal{R}_{Q,p,V}\mathcal{R}_{Q,p,U}^{-1}\|^{\frac{2}{j}}|\tilde{b}_{Q}^{\epsilon}|^{2} \right)^\frac12
\\ & =\sup_{I\in\mathcal{D}}\,  \left(\fint_I |F_J (x)  - \langle F_J \rangle _I |^2 \, dx \right)^\frac12 \approx 
\sup_{I\in\mathcal{D}}\,  \left(\fint_I |F_J (x)  - \langle F_J \rangle _I |^{1+\epsilon'} \, dx \right)^\frac{1}{1+\epsilon'}
\\ & \approx \sup_{I\in\mathcal{D}}  \left(\fint_{I}\left(\sum_{\substack{Q\in\mathcal{D}(I)\\
\varepsilon\in\S
}
}\frac{\|\mathcal{R}_{Q,p,V}\mathcal{R}_{Q,p,U}^{-1}\|^{\frac{2}{j}}|\tilde{b}_{Q}^{\epsilon}|^{2}}{|Q|}\chi_{Q}(x)\right)^{\frac{1+\epsilon'}{2}}\,dx\right)^{\frac{1}{1+\epsilon'}}
\\ & =  \sup_{I\in\mathcal{D}(J)}  \left(\fint_{I}\left(\sum_{\substack{Q\in\mathcal{D}(I)\\
\varepsilon\in\S
}
}\frac{\|\mathcal{R}_{Q,p,V}\mathcal{R}_{Q,p,U}^{-1}\|^{\frac{2}{j}}|b_{Q}^{\epsilon}|^{2}}{|Q|}\chi_{Q}(x)\right)^{\frac{1+\epsilon'}{2}}\,dx\right)^{\frac{1}{1+\epsilon'}}
\\ & \leq C\sup_{I\in\mathcal{D}(J)} \left(\fint_{I}\left(\sum_{\substack{Q\in\mathcal{D}(I)\\
\ \varepsilon\in\S
}
}\frac{\|\mathcal{R}_{I,p,V}\mathcal{R}_{I,p,U}^{-1}\|^{\frac{2}{j}}|b_{Q}^{\varepsilon}|^{2}}{|Q|}\chi_{Q}(x)\right)^{\frac{1+\epsilon'}{2}}\,dx\right)^{\frac{1}{1+\epsilon'}} \\ & = C \sup_{I\in\mathcal{D}(J)}\left(\fint_{I}\|\mathcal{R}_{I,jp,V_{I}}(b(y)-\langle b\rangle_{I})\|^{1+\epsilon'}\,dy\right)^{\frac{1}{1+\epsilon'}}
\end{align*} \noindent as desired.

Let $C_J$ denote the right hand side of \eqref{MatrixJNIneq}.  Using the stopping time discussed above, we have for any $I \in \mathcal{D}(J)$ that 
\begin{align*}
\fint_{I} & \left(\sum_{\substack{Q\in\mathcal{D}(I)\\
\varepsilon\in\S
}
}\frac{\|\mathcal{R}_{Q,p,V}\mathcal{R}_{Q,p,U}^{-1}\|^{\frac{2}{j}}|b_{Q}^{\epsilon}|^{2}}{|Q|}\chi_{Q}(x)\right)^{\frac{1+\epsilon'}{2}}\,dx\\
 & \leq\fint_{I}\left(\sum_{j=1}^{\infty}\sum_{K\in\J^{j-1}(I)}\sum_{\substack{Q\in\mathcal{F}(K)\\
\varepsilon\in\S
}
}\frac{\|\mathcal{R}_{Q,p,V}\mathcal{R}_{K,p,V}^{-1}\|^{\frac{2} {j}} \|\mathcal{R}_{K,p,V}\mathcal{R}_{K,p,U}^{-1}\|^{\frac{2} {j}}|b_{Q}^{\epsilon}|^{2}\|\mathcal{R}_{K,p,U}\mathcal{R}_{Q,p,U}^{-1}\|^{\frac{2}{j}}}{|Q|}\chi_{Q}(x)\right)^{\frac{1+\epsilon'}{2}}\,dx\\
 & \lesssim\fint_{I}\left(\sum_{j=1}^{\infty}\sum_{K\in\J^{j-1}(I)}\sum_{\substack{Q\in\mathcal{F}(K)\\
\varepsilon\in\S
}
}\frac{\|\mathcal{R}_{K,p,V}\mathcal{R}_{K,p,U}^{-1}\|^{\frac{2} {j}}|b_{Q}^{\epsilon}|^{2}}{|Q|}\chi_{Q}(x)\right)^{\frac{1+\epsilon'}{2}}\,dx\\
  & \leq\frac{1}{|I|}\sum_{j=1}^{\infty}\sum_{K\in\mathcal{J}^{j-1}(I)}\int_{K}\left(\sum_{\substack{Q\in\mathcal{F}(K)\\
\varepsilon\in\S
}
}\frac{\|\mathcal{R}_{K,p,V}\mathcal{R}_{K,p,U}^{-1}\|^{\frac{2} {j}}|b_{Q}^{\epsilon}|^{2}}{|Q|}\chi_{Q}(x)\right)^{\frac{1+\epsilon'}{2}}\,dx\\
 & \leq C_J ^{1+\epsilon'} \frac{1}{|I|}\sum_{j=1}^{\infty}\sum_{K\in\mathcal{J}^{j-1}(I)} |K| \\
 &  =  C_J ^{1+\epsilon'} \frac{1}{|I|}\sum_{j=1}^{\infty} \left|\bigcup\J^{j}(I)\right|
 \\ & \leq  C_J ^{1+\epsilon'}
\end{align*}
\end{proof}
The next lemma follows immediately from the last three proved. \begin{lemma}\label{eJNHigherLem}
If $U$ and $V$ are matrix $A_{p}$ weights then there exists $\epsilon>0$
small enough where for any $0<\epsilon'<\epsilon$ we have 
\[
\sup_{I\in\mathcal{D}(J)}\left(\fint_{I}\|V^{\frac{1}{p}}(y)\mathcal{R}_{I,p,U}^{-1}\|^{p}|b(y)-\langle b\rangle_{I}|^{jp}\,dy\right)^{\frac{1}{jp}}\leq C\sup_{I\in\mathcal{D}(J)}\left(\fint_{I}\|\mathcal{R}_{I,jp,V_{I}}(b(y)-\langle b\rangle_{I})\|^{1+\epsilon'}\,dy\right)^{\frac{1}{1+\epsilon'}}
\]
and 
\[
\sup_{I\in\mathcal{D}(J)}\left(\fint_{I}\|\mathcal{R}_{I,jp,V_{I}}(b(y)-\langle b\rangle_{I})\|^{1+\epsilon}\,dy\right)^{\frac{1}{1+\epsilon}}\leq C\sup_{I\in\mathcal{D}(J)}\left(\fint_{I}\|V^{\frac{1}{p}}(x)\mathcal{R}_{I,p,U}^{-1}\|^{p}|b(y)-\langle b\rangle_{I}|^{jp}\,dy\right)^{\frac{1}{jp}}.
\]
where $C$ is independent of $J$ (but depends on $\epsilon'$.) \end{lemma}

\begin{lemma} \label{HigherEqnLem}\ Let $\mathcal{D}$ be a fixed
dyadic grid. If $U,V$ are matrix $A_{p}$ weights and $\epsilon'>0$
is small enough then 
\begin{equation}
\sup_{I\in\mathcal{D}}\left(\fint_{I}\|\mathcal{R}_{I,jp,V_{I}}(b(y)-\langle b\rangle_{I})\|^{1+\epsilon'}\,dy\right)^{\frac{1}{1+\epsilon'}}\lesssim\sup_{I\in\mathcal{D}}\fint_{I}\|\mathcal{R}_{I,jp,V_{I}}(b(y)-\langle b\rangle_{I})\|\,dy.\label{HigherEqn}
\end{equation}
\end{lemma}
\begin{proof}
Without loss of generality we may assume $U$ and $V^{-1}$ are bounded. For more details, see Remark \ref{JNRem}. Let $\|b\|_{*}$ be the right hand side of \eqref{HigherEqn}.  For
fixed $R\in\mathbb{N}$ let $P_{R}$ be the canonical projection operator
\[
P_{R}b(x)=\sum_{|I|=2^{-R}}\chi_{{I}}\langle b\rangle_{I}.
\]
For notational ease let 
\[
F_{Q}(x)=\chi_Q(x)\mathcal{R}_{Q,jp,V_{Q}}(b(x)-\langle b\rangle_{Q}),
\]
let 
\[
F_{Q}^{R}(x)= P_{R} F_Q (x)  ,
\]
and let $d\mu_{Q}(x)=|Q|^{-1}\chi_{Q}(x)\,dx$. Assume that $\tnorm{U^{-1}}_{L^\infty} < \infty$ and $\tnorm{V}_{L^\infty} < \infty$, where, for a measurable matrix function $A$, $\tnorm{A}_{L^\infty} = \|\,\|A(x)\|\,\|_{L^\infty(dx)}$. Clearly we have  for any $J \in \mathcal{D}$ that 
\begin{align*}
\sup_{Q \in \mathcal{D}(J)}& \|F_{Q}^{R}\|_{L^{1+\epsilon'}(d\mu_{Q})}\,dx \leq \sup_{Q \in \mathcal{D}(J)} \|F_{Q}^{R}\|_{L^{\infty}} \leq 2^{R+1}  \sup_{Q \in \mathcal{D}(J)}  \|\mathcal{R}_{Q,jp,V_{Q}}\|\int_Q |b|   
\\ & {\lesssim} 2^R \tnorm{U^{-1}}_{L^\infty} ^{jp} \tnorm{V}_{L^\infty} ^{jp} \int_J |b| <  \infty.
\end{align*}
Also, clearly 
\[
\sup_{Q\in\mathcal{D}}\|F_{Q}^{R}\|_{L^{1}(d\mu_{Q})}\lesssim\|b\|_{*}
\]
independent of $R$ by the boundedness of $P_{R}$ on $L^{1}(\mathbb{R}^{d})$
independent of $R$. On the other hand, it is easy to check  by direct calculation (and in particular looking at the two separate cases $|Q| \leq 2^{-R}$ and $|Q| > 2^{-R}$) that $$ F_Q ^R (x)   = \chi_Q(x)\mathcal{R}_{Q,jp,V_{Q}}(P_R b (x)-\langle P_R b\rangle_{Q})$$

Let $\epsilon>0$ be from Lemma \ref{eJNHigherLem} . Also for $0<\epsilon'<\epsilon$
let $p_{1}=1+\epsilon',\ p_{2}=1+\epsilon,\alpha=\frac{p_{2}-p_{1}}{p_{2}-1},\ \beta=\frac{p_{1}-\alpha}{\alpha}$,
and let $C$ be the constant in Lemma \ref{eJNHigherLem}. Then by 
a use of Hölder's inequality and  Lemma \ref{eJNHigherLem} applied to $F_Q ^R$,  we have for any $J \in \mathcal{D}$ that 
\begin{align*}
\sup_{Q \in \mathcal{D}(J)} \|F_{Q}^{R}\|_{L^{p_{1}}(d\mu_{Q})}^{p_{1}} & \leq \sup_{Q \in \mathcal{D}(J)} \|F_{Q}^{R}\|_{L^{1}(d\mu_{Q})}^{\alpha}\|F_{Q}^{R}\|_{L^{p_{2}}(d\mu_{Q})}^{p_{1}-\alpha}\\
 & \leq\|b\|_{*}^{\alpha}C^{p_{1}-\alpha}\sup_{Q \in \mathcal{D}(J)} \|F_{Q}^{R}\|_{L^{p_{1}}(d\mu_{Q})}^{p_{1}-\alpha}
\end{align*}
which says that 
\[
 \|F_{Q}^{R}\|_{L^{p_{1}}(d\mu_{Q})}\leq\|b\|_{*}C^{\beta}.
\]
for any $Q \in \mathcal{D}(J)$.  Letting $R\rightarrow\infty$ first (using Fatou's lemma) and then
taking the supremum over all $Q\in\mathcal{D}(J)$ and then $J\in\mathcal{D}$
completes the proof under the assumption that $\tnorm{U^{-1}}_{L^\infty} < \infty$ and ${\tnorm{V}}_{L^\infty} < \infty$. 
\end{proof}
\begin{remark} \label{JNRem}
    The proof of Lemma \ref{HigherEqnLem} only handles weights $U$ and $V$ where $U^{-1}$ is bounded and $V$ is bounded. To reduce to this case, we modify the truncation arguments in \cite[Remark 3.3]{BPW}.  Namely, let $V(x) = \sum_{j = 1}^n \lambda_j(x) \langle \cdot,  e_j (x){\rangle}_{\Cn} e_j(x)$ where $\{e_j(x)\}_{j = 1}^n$ is an orthonormal eigenbasis of $V(x)$, and $\{\lambda_j(x)\}_{j = 1}^n$ are the corresponding eigenvalues.
    For $k \in \mathbb{N}$, let us consider the sets $N_1 ^k =  \{j \in \{1, \ldots, n\} : \lambda_j(x) < k\}$ and 
    $  N_2^k =   \{j \in \{1, \ldots, n\} : \lambda_j(x) \geq k\}$  and define $ V_k(x) = \widetilde{V}_k(x) + k P_{N_2 ^k} (x)$ where 
$$ \widetilde{V}_k(x) =   \sum_{j \in N_1 ^k} \lambda_j(x) \langle \cdot,  e_j (x){\rangle}_{\Cn} e_j(x) \text{ and } P_{N_2 ^k} (x) =  \sum_{j \in N_2 ^k}  \langle \cdot,  e_j (x){\rangle}_{\Cn} e_j(x).$$ Then   $ V_k ^{-1}(y) = \widetilde{V}_k ^{-1} (y) + k^{-1} P_{N_2 ^k} (y)$ where $$ \widetilde{V} ^{-1} _k(y) =   \sum_{j \in N_1 ^k} \lambda_j^{-1}(y) \langle \cdot,  e_j (y){\rangle}_{\Cn} e_j(y).$$ Clearly $V_k ^q (x)  \leq V^q (x)$ for any $q > 0$ while also $V_k ^q (x) \leq k ^q \, Id_{n \times n}$ and $ V_k ^{-q}(y) = \widetilde{V}_j ^{-q} (y) + k^{-q} P_{N_2 ^k} (y)$ with $\widetilde{V}_k ^{-q} (y) \leq {V} ^{-q} (y) $.    We now show that $[{V_k}]_{\text{A}_p} {\lesssim} [{V}]_{\text{A}_p}$ independent of $k$.

To this end, \begin{align} & \norm{V_k ^\frac{1}{p} (x) V_k ^{-\frac{1}{p}} (y)}^2 = \norm{V_k ^\frac{1}{p} (x) V_k ^{-\frac{2}{p}} (y)V_k ^\frac{1}{p} (x)}  {\leq} \text{tr }\left( V_k ^\frac{1}{p} (x) V_k ^{-\frac{2}{p}} (y) V_k ^\frac{1}{p} (x)\right) = \text{tr } \left(V_k ^\frac{2}{p} (x) V_k ^{-\frac{2}{p}} (y)\right)
\nonumber \\ & = \text{tr } \left(V_k ^\frac{2}{p} (x) \widetilde{V} _k ^{-\frac{2}{p}} (y)\right) + k^{-\frac{2}{p}} \text{tr } \left(V_k ^\frac{2}{p} (x) P_{N_2^k} (y)\right) \leq \text{tr } \left(V ^\frac{2}{p} (x) V ^{-\frac{2}{p}} (y)\right)  + 1 \approx \norm{V ^\frac{1}{p} (x) V ^{-\frac{1}{p}} (y)}^2 + 1 \label{TruncNormEst}
\end{align} \noindent which clearly shows that $[{V_k}]_{\text{A}_p} {\lesssim}[V]_{\text{A}_p}$.  Here, we have used the elementary and easily proved fact that if $A_1, B_1, A_2, B_2 \geq 0$ and $A_i \leq B_i$ for $i = 1, 2$ then $\text{tr} \, (A_1 A_2) \leq \text{tr} \, (B_1 B_2).  $ 

Regarding $U$, again let $U(x) = \sum_{j = 1}^n \lambda_j(x) \langle \cdot,  e_j (x){\rangle}_{\Cn} e_j(x)$ where $\{e_j(x)\}_{j = 1}^n$ is an orthonormal eigenbasis of $U(x)$ with corresponding eigenvalues $\{\lambda_j(x)\}_{j = 1}^n$.  For $k \in \mathbb{N}$, let us define the sets $N_3 ^k =  \{j \in \{1, \ldots, n\} : \lambda_j(x) < k^{-1}\}$ and 
    $  N_4^k =   \{j \in \{1, \ldots, n\} : \lambda_j(x) \geq k^{-1}\}$, and $ U_k(x) = \widetilde{U}_k(x) + k^{-1} P_{N_3 ^k} (x)$ where 
$$ \widetilde{U}_k(x) =   \sum_{j \in N_4 ^k} \lambda_j(x) \langle \cdot,  e_j (x){\rangle}_{\Cn} e_j(x) \text{ and } P_{N_3 ^k} (x) =  \sum_{j \in N_3 ^k}  \langle \cdot,  e_j (x){\rangle}_{\Cn} e_j(x).$$ Then again  $ U_k ^{-1}(y) = \widetilde{U}_k ^{-1} (y) + k P_{N_3 ^k} (y)$ where $$ \widetilde{U} ^{-q} _k(y) =   \sum_{j \in N_4 ^k} \lambda_j ^{-1}(y) \langle \cdot,  e_j (y){\rangle}_{\Cn} e_j(y).$$ Clearly $U_k ^{-q} (y)  \leq U^{-q} (y)$ for any $q > 0$ while also $U_k ^{-q} (y) \leq k ^{q} \, Id_{n \times n}$ and $ U_k ^{q}(y) = \widetilde{U}_k ^{q} (x) + k^{-q} P_{N_3 ^k} (x)$ with $\widetilde{U}_k ^{q} (x) \leq {U} ^{q} (x) $.     Thus, $$ \text{tr } \left(U_k ^\frac{2}{p} (x) U_k ^{-\frac{2}{p}} (y)\right)  = \text{tr }\left( \widetilde{U}_k ^{\frac{2}{p}} (x) U_k ^{-\frac{2}{p}} (y)\right) + k^{-\frac{2}{p}} \text{tr } \left(P_{N_3 ^k}(x) U_k ^{-\frac{2}{p}} (y) \right) \leq    \text{tr } \left({U} ^{\frac{2}{p}} (x) U ^{-\frac{2}{p}} (y)\right)  + 1,  $$ which shows that $[U_k]_{\text{A}_p} {\lesssim}[{U}]_{\text{A}_p}$ independent of $k$, arguing as before.

Now define $V_{I, k}(x)$ 
by 
\[
V_{I, k}(x)=\left(\mathcal{R}_{I,p,U_k}^{-1}V_k ^{\frac{2}{p}}(x)\mathcal{R}_{I,p,U_k}^{-1}\right)^{\frac{p}{2}},
\]

Applying Fatou's Lemma and\eqref{HigherEqn} applied to $V_k$ and $U_k$, we have for fixed $I' \in \mathcal{D}$ that for $0 < \epsilon' < \epsilon$

\begin{align*}
 & \left( \fint_{I'}\|\mathcal{R}_{I',jp,V_{I'}}(b(y) - \langle b\rangle_{I'}) \|^{1+\epsilon'} \, dy \right)^\frac{1}{1+\epsilon'} \leq \liminf_{k \rightarrow \infty }\left(\fint_{I'}  \|\mathcal{R}_{I',jp,V_{I', k}}(b(y) - \langle b\rangle_{I'}) \|^{1+\epsilon'} \, dy \right)^\frac{1}{1+\epsilon'} \\ & \lesssim \liminf_{k \rightarrow \infty } \sup_{I \in \mathcal{D}}  \fint_I\|\mathcal{R}_{I,jp,V_{I, k}}(b(y) - \langle b\rangle_{I} ) \| \, dy
\end{align*}

However, recall that by Lemma \ref{VIlem} we have $\|\mathcal{R}_{I,jp,V_{I, k}}\| =    \|\mathcal{R}_{I,p,V_k}\mathcal{R}_{I,p,U_k}^{-1}\|^{\frac{1}{jp}} $.  Furthermore, arguing as in \eqref{TruncNormEst} and using the fact that $V_k^\frac{1}{p}(x) \leq V^\frac{1}{p}(x)$ and $U_k^{-\frac{1}{p}} (y) \leq U^{-\frac{1}{p}} (y)$ we have 

\begin{align*}& \|\mathcal{R}_{I,p,V_k}\mathcal{R}_{I,p,U_k}^{-1}\|^p \lesssim \|\mathcal{R}_{I,p,V_k}\mathcal{R}_{I,p,U_k} '\| ^p \approx \fint_I \left(\fint_I \| V_k ^\frac{1}{p} (x) U_k ^{-\frac{1}{p}} (y)\|^{p'} \, dy \right)^\frac{p}{p'} \, dx  \\ & \lesssim \fint_I \left(\fint_I \| V ^\frac{1}{p} (x) U ^{-\frac{1}{p}} (y)\|^{p'} \, dy \right)^\frac{p}{p'} \, dx \approx \|\mathcal{R}_{I,p,V}\mathcal{R}_{I,p,U} '\| ^p \leq [U]_{\text{A}_p}  \|\mathcal{R}_{I,p,V}\mathcal{R}_{I,p,U} ^{-1} \| ^p 
\end{align*}

which says that $\|\mathcal{R}_{I,jp,V_{I, k}}\| \lesssim [U]_{\text{A}_p} \|\mathcal{R}_{I,jp,V_{I}}\| $ and clearly completes the proof.

\end{remark}

At this point we are in the position to settle Theorem \ref{StrongHigherJN}.
\begin{proof}[Proof of Theorem \ref{StrongHigherJN}]
Without loss of generality, we can assume that the supremums are
taken over some fixed dyadic grid and that the estimates do not depend
on the dyadic grid used. Let us see first that the supremum in $a)$
and $b)$ are equivalent, namely $\|b\|_{BMO_{V,U,j}^{p}}^{jp}\simeq\|b\|_{BMO_{V,U,j,1}^{p}}^{p}$.
We begin showing that 
\[
\|b\|_{BMO_{V,U,j}^{p}}^{jp}\lesssim\|b\|_{BMO_{V,U,j,1}^{p}}^{p}.
\]
Let us fix a cube $Q$. Then we have that
\begin{align*}
 & \left(\fint_Q\|\mathcal{R}_{Q,p,V}\mathcal{R}_{Q,p,U}^{-1}\|^{\frac{1}{j}}|b(x)-\langle b\rangle_{Q}|\,dx\right)^{jp}\\
= & \left(\fint_Q\|\mathcal{R}_{Q,p,V}V^{-\frac{1}{p}}(x)V^{\frac{1}{p}}(x)\mathcal{R}_{Q,p,U}^{-1}\|^{\frac{1}{j}}|b(x)-\langle b\rangle_{Q}|\,dx\right)^{jp}\\
\leq & \left(\fint_Q\|V^{\frac{1}{p}}(x)\mathcal{R}_{Q,p,U}^{-1}\|^{p}|b(x)-\langle b\rangle_{Q}|^{jp}\,dx\right)\left(\fint_Q\|\mathcal{R}_{Q,p,V}V^{-\frac{1}{p}}(x)\|^{\frac{(jp)'}{j}}\,dx\right)^{\frac{jp}{(jp)'}}\\
\leq & \|b\|_{BMO_{V,U,j,1}^{p}}^{p}\left(\fint_Q\|\mathcal{R}_{Q,p,V}V^{-\frac{1}{p}}(x)\|^{\frac{(jp)'}{j}}\,dx\right)^{\frac{jp}{(jp)'}}.
\end{align*}
Now, since $\frac{(jp)'}{j}\leq p'$, 
\begin{align*}
 & \|b\|_{BMO_{V,U,j,1}^{p}}^{p}\left(\fint_Q\|\mathcal{R}_{Q,p,V}V^{-\frac{1}{p}}(x)\|^{\frac{(jp)'}{j}}\,dx\right)^{\frac{jp}{(jp)'}}\\
\leq & \|b\|_{BMO_{V,U,j,1}^{p}}^{p}\left(\fint_{Q}\|\mathcal{R}_{Q,p,V}V^{-\frac{1}{p}}(x)\|^{p'}\,dx\right)^{\frac{p}{p'}}\\
\simeq & \|b\|_{BMO_{V,U,j,1}^{p}}^{p}\|\mathcal{R}_{Q,p,V}\mathcal{R}_{Q,p,V}^{'}\|^{p}\simeq\|b\|_{BMO_{V,U,j,1}^{p}}^{p}.
\end{align*}
Combining the estimates above, 
\[
\|b\|_{BMO_{V,U,j}^{p}}^{jp}\lesssim\|b\|_{BMO_{V,U,j,1}^{p}}^{p}.
\]
Let us show now that 
\[
\|b\|_{BMO_{V,U,j,1}^{p}}^{p}\lesssim\|b\|_{BMO_{V,U,j}^{p}}^{jp}.
\]
Again we fix a cube $Q$. We have that with the notation in 
Lemma
\ref{VIlem} and using Lemma \ref{eJNHigherLem}, 
\begin{align*}
 & \left(\fint_{Q}\|V^{\frac{1}{p}}(x)\mathcal{R}_{Q,p,U}^{-1}\|^{p}|b(x)-\langle b\rangle_{Q}|^{jp}\,dx\right)^{\frac{1}{jp}}\\
\leq & \sup_{I\in\mathcal{D}(J)}\left(\fint_{I}\|V^{\frac{1}{p}}(y)\mathcal{R}_{I,p,U}^{-1}\|^{p}|b(y)-\langle b\rangle_{I}|^{jp}\,dy\right)^{\frac{1}{jp}}\\
\lesssim & \sup_{I\in\mathcal{D}(J)}\left(\fint_{I}\|\mathcal{R}_{I,jp,V_{I}}(b(y)-\langle b\rangle_{I})\|^{1+\epsilon'}\,dy\right)^{\frac{1}{1+\epsilon'}}.
\end{align*}
Now, by Lemma \ref{HigherEqnLem},
\begin{align*}
 & \sup_{I\in\mathcal{D}(J)}\left(\fint_{I}\|\mathcal{R}_{I,jp,V_{I}}(b(y)-\langle b\rangle_{I})\|^{1+\epsilon'}\,dy\right)^{\frac{1}{1+\epsilon'}}\\
\leq & \sup_{I\in\mathcal{D}}\left(\fint_{I}\|\mathcal{R}_{I,jp,V_{I}}(b(y)-\langle b\rangle_{I})\|^{1+\epsilon'}\,dy\right)^{\frac{1}{1+\epsilon'}}\\
\lesssim & \sup_{I\in\mathcal{D}}\fint_{I}\|\mathcal{R}_{I,jp,V_{I}}(b(y)-\langle b\rangle_{I})\|\,dy.
\end{align*}
At this point, since by Lemma \ref{VIlem}, for every cube $P$, $\|\mathcal{R}_{P,jp,V_{I}}\|\simeq\|\mathcal{R}_{P,p,V}\mathcal{R}_{Q,p,U}^{-1}\|^{\frac{1}{j}}$,
we have that
\begin{align*}
 & \sup_{I\in\mathcal{D}}\fint_{I}\|\mathcal{R}_{I,jp,V_{I}}(b(y)-\langle b\rangle_{I})\|\,dy\lesssim\sup_{I\in\mathcal{D}}\fint_{I}\|\mathcal{R}_{I,jp,V_{I}}\||b(y)-\langle b\rangle_{I}|\,dy\\
\simeq & \sup_{I\in\mathcal{D}}\fint_{I}\|\mathcal{R}_{I,p,V}\mathcal{R}_{Q,p,U}^{-1}\|^{\frac{1}{j}}|b(y)-\langle b\rangle_{I}|\,dy\leq\|b\|_{BMO_{V,U,j}^{p}}^{jp}.
\end{align*}
Gathering the estimates above, it readily follows that $\|b\|_{BMO_{V,U,j,1}^{p}}^{p}\lesssim\|b\|_{BMO_{V,U,j}^{p}}^{jp}$.

The equivalence between $a)$ and $c)$, namely, $\|b\|_{BMO_{V,U,j}^{p}}^{jp}\simeq\|b\|_{BMO_{V,U,j,1}^{p}}^{p'}$
can be settled arguing analogously as we have just done above.

Next we continue showing that the conditions in $b)$ and $c)$ are equivalent.
Indeed, observe that for any cube $Q$, 
\begin{align*}
\fint_{Q} & \|V^{\frac{1}{p}}(x)\mathcal{R}_{Q,p,U}^{-1}\|^{p}|b(x)-\langle b\rangle_{Q}|^{jp}\,dx\\
 & \leq\fint_{Q}\|V^{\frac{1}{p}}(x)\mathcal{R}_{Q,p,U}^{-1}\|^{p}\left(\fint_{Q}|b(x)-b(y)|\right)^{jp}\,dx\\
 & \leq\fint_{Q}\left(\fint_{Q}\|U^{\frac{1}{p}}(y)\mathcal{R}_{Q,p,U}^{-1}\|\|U^{-\frac{1}{p}}(y)V^{\frac{1}{p}}(x)\||b(x)-b(y)|^{j}\,dy\right)^{p}\,dx\\
 & \leq\fint_{Q}\left(\fint_{Q}\|U^{\frac{1}{p}}(y)\mathcal{R}_{Q,p,U}^{-1}\|^{p}dy\right)^{\frac{p}{p}}\left(\fint_{Q}\|V^{\frac{1}{p}}(x)U^{-\frac{1}{p}}(y)\|^{p'}|b(x)-b(y)|^{jp'}\,dy\right)^{\frac{p}{p'}}\,dx\\
  & \lesssim\fint_{Q}\left(\fint_{Q}\|V^{\frac{1}{p}}(x)U^{-\frac{1}{p}}(y)\|^{p'}|b(x)-b(y)|^{jp'}\,dy\right)^{\frac{p}{p'}}\,dx.
\end{align*}
From this it readily follows that $\|b\|_{BMO_{V,U,j,1}^{p}}^{p}\lesssim\|b\|_{\widetilde{BMO}_{V,U,j}^{p}}.$
Reciprocally
\begin{align*}
 & \fint_{Q}\left(\fint_{Q}\|V^{\frac{1}{p}}(x)U^{-\frac{1}{p}}(y)\|^{p'}|b(x)-b(y)|^{jp'}\,dy\right)^{\frac{p}{p'}}\,dx\\
 & \lesssim\fint_{Q}\left(\fint_{Q}\|V^{\frac{1}{p}}(x)U^{-\frac{1}{p}}(y)\|^{p'}(|b(x)-\langle b\rangle_{Q}|^{jp'}+|b(y)-\langle b\rangle_{Q}|^{jp'})\,dy\right)^{\frac{p}{p'}}\,dx\\
 & \lesssim\fint_{Q}|b(x)-\langle b\rangle_{Q}|^{jp}\left(\fint_{Q}\|V^{\frac{1}{p}}(x)U^{-\frac{1}{p}}(y)\|^{p'}\,dy\right)^{\frac{p}{p'}}\,dx\\
 & +\fint_{Q}\left(\fint_{Q}\|V^{\frac{1}{p}}(x)U^{-\frac{1}{p}}(y)\|^{p'}|b(y)-\langle b\rangle_{Q}|^{jp'}\,dy\right)^{\frac{p}{p'}}\,dx\\
 & \lesssim\fint_{Q}|b(x)-\langle b\rangle_{Q}|^{jp}\|V^{\frac{1}{p}}(x)\mathcal{R}_{Q,p,U}^{'}\|^{p}\,dx\\
 & +\left(\fint_{Q}\|V^{\frac{1}{p}}(x)\mathcal{R}_{Q,p,V}^{'}\|^{p}dx\right)\left(\fint_{Q}\|(\mathcal{R}_{Q,p,V}^{'})^{-1}U^{-\frac{1}{p}}(y)\|^{p'}|b(y)-\langle b\rangle_{Q}|^{jp'}\,dy\right)^{\frac{p}{p'}}\\
 & \lesssim\|\mathcal{R}_{Q,p,U}\mathcal{R}_{Q,p,U}^{'}\|^{p}\fint_{Q}|b(x)-\langle b\rangle_{Q}|^{jp}\|V^{\frac{1}{p}}(x)\mathcal{R}_{Q,p,U}^{-1}\|^{p'}\,dx\\
 & +\|\mathcal{R}_{Q,p,V}\mathcal{R}_{Q,p,V}^{'}\|^{p}\left(\fint_{Q}\|U^{-\frac{1}{p}}(y)(\mathcal{R}_{Q,p,V}^{'})^{-1}\|^{p'}|b(y)-\langle b\rangle_{Q}|^{jp'}\,dy\right)^{\frac{p}{p'}}\\
 & \lesssim\|b\|_{BMO_{V,U,j,1}^{p}}^{p}+\|b\|_{BMO_{V,U,j,2}^{p}}^{p}
\end{align*}
and since, as we showed above, $\|b\|_{BMO_{V,U,j,2}^{p}}^{p}\simeq\|b\|_{BMO_{V,U,j,1}^{p}}^{p}$,
this implies the desired inequality. The equivalence of $e)$ with
$a)-c)$ can be obtained arguing analogously as above.
\end{proof}

We end this section noting that an analogue of \eqref{StrongHigherJN} in the more general situation than a single scalar symbol $b$ would be quite interesting.  In particular,  we believe that ``suitable combinations'' of the norms in more general situations may give rise to the same $BMO$ class. However, as we have already noted, we are still unable to prove such a claim.


\section{Proof of the main results}\label{sec:ProfMain}

In this section we will provide proofs for all the results mentioned in the previous sections.

\subsection{Proof of Theorems \ref{th:tocho} and \ref{th:tocho3}}\label{sect:demotocho}

In order to prove the two theorems, we are going to show that two more general analogous results,
which are stronger, but more abstract, are satisfied.  
The proof of these two results is going to be a very simple application of 
the following lemma.

\begin{lemma}\label{lemma:comm_exp}
    Formally, 
    \begin{equation*}
        \begin{split}
            \vec{T}_{\vec{B}}\vec{f}(x) = \sum_{\sigma\in C(m)} (-1)^{m-|\sigma|}\vec{B}_{\sigma}(x) \vec{T}(\vec{B}_{(\sigma^c)^t}\vec{f})(x).
        \end{split}
        \end{equation*}
        
        In particular, if we define $\Psi$, $\tilde{\Psi}$ and $\overline{T}$ as in 
        \eqref{psi}, \eqref{tildepsi} and \eqref{Tbarra}, respectively, 
        we may write 
        \begin{equation*}
        \begin{split}
            \vec{T}_{\vec{B}}\vec{f}(x) = \Psi(x)\overline{T}(\tilde{\Psi}\vec{f})(x).
        \end{split}
        \end{equation*}
\end{lemma}

\begin{proof}
    The result is clear for $m=1$, so, if we assume by induction that 
    it is also true for some $m$, then, 
    for $\vec{B} = (B_1,...,B_{m+1})$ we have that 
\begin{align*}
        \vec{T}_{\vec{B}}\vec{f}(x) &= 
        B_{m+1}(x) \sum_{j=0}^{m} (-1)^{m-j}\sum_{\sigma\in C_j(m)}\vec{B}_{\sigma}(x) \vec{T}(\vec{B}_{(\sigma^c)^t}\vec{f})(x)
    - \sum_{j=0}^{m} (-1)^{m-j}\sum_{\sigma\in C_j(m)}\vec{B}_{\sigma}(x) \vec{T}(\vec{B}_{(\sigma^c)^t}B_{m+1}(x)\vec{f})(x)
   \\ & = \sum_{j=0}^{m} (-1)^{m-j}\sum_{\underrel{\sigma(j+1) = m+1}{\sigma\in C_{j+1}(m+1)}}\vec{B}_{\sigma}(x) \vec{T}(\vec{B}_{(\sigma^c)^t}\vec{f})(x)
   - \sum_{j=0}^{m} (-1)^{m-j}\sum_{\underrel{\sigma(j)\neq m+1}{\sigma\in C_j(m+1)}}\vec{B}_{\sigma}(x) \vec{T}(\vec{B}_{(\sigma^c)^t}\vec{f})(x)
   \\ & =\sum_{j=1}^{m+1} (-1)^{m+1-j}\sum_{\underrel{\sigma(j) = m+1}{\sigma\in C_{j}(m+1)}}\vec{B}_{\sigma}(x) \vec{T}(\vec{B}_{(\sigma^c)^t}\vec{f})(x)
   + \sum_{j=0}^{m} (-1)^{m+1-j}\sum_{\underrel{\sigma(j)\neq m+1}{\sigma\in C_j(m+1)}}\vec{B}_{\sigma}(x) \vec{T}(\vec{B}_{(\sigma^c)^t}\vec{f})(x)
   \\ & =\sum_{j=0}^{m+1} (-1)^{m+1-j}\sum_{{\sigma\in C_{j}(m+1)}}\vec{B}_{\sigma}(x) \vec{T}(\vec{B}_{(\sigma^c)^t}\vec{f})(x).
\end{align*}
\end{proof}


The generalization of Theorem \ref{th:tocho} that we are going to prove is the following.

\begin{theorem}\label{th:tochogen}
    Let $\Omega$ be a measurable subset of some measure space $(X,\mathcal{M},\mu)$, 
     $\mathbb{F}=\mathbb{R}$ or $\mathbb{C}$, and
    $r\in [1,\infty]$. 
    Also consider $\{E_k\}_{k=1}^\infty$ a collection of spaces such that $E_k\subset L^r((\Omega, \mu),\mathbb{F}^k)$ 
    and $E_{k_1}\times E_{k_2} \subset E_{k_1+k_2}$ whenever $k,k_1,k_2\in \mathbb{N}$,
    and
    assume $\vec{T}=(T_1,...,T_n)$ is a {linear} operator 
    with the following property:
    \begin{itemize}
        \item[$(P_1)$] For every $\vec{f}\in E_{2^mn}$,
        there exist $C>0$,
        a collection of subsets of $\Omega$, $\mathcal{G}$, coefficients $\{a_Q\}_{Q\in \mathcal{G}}$, and
        functions $K_Q:Q\times Q\rightarrow \mathcal{M}_n(\mathbb{F})$, for $Q\in \mathcal{G}$, 
        with 
        $\|K_Q(x,\cdot)\|_{L^{r'}\left(\frac{d \mu}{|Q|}\right)}\leq 1$, for a.e $x\in Q$
        such that, for the $2^mn\times 2^m n$ matrix function, given by the $n\times n$ blocks
        \begin{equation*}
        \begin{split}
            \overline{K}_Q = 
\begin{pNiceArray}{cc|cc|cc|cc}
  \quad K_Q &&  {{0}} && \dots && {0}\\
  \hline
  {{0}} && K_Q&& \dots && {0}\\
  \hline
  \vdots && \vdots && \ddots     && \vdots\\
  \hline
  0 && 0 && \dots &&K_Q
\end{pNiceArray}
        \end{split}
        \end{equation*} 
        and the operator $\overline{T}$ given by 
            \begin{equation}\label{Tbarraa}
        \begin{split}
            \overline{T}(\vec{g}) = \begin{pmatrix}
            \vec{T}(\vec{g}^1)\\\vec{T}(\vec{g}^2)\\\vdots\\\vec{T}(\vec{g}^{2^m})
            \end{pmatrix},
        \end{split}
        \end{equation}
        for an appropriate $\vec{g} = (g_1,g_2,...,g_{2^mn})$
        ($\vec{g}^k=(g_{(k-1)n+1},g_{(k-1)n+2},...,g_{(k-1)n+n})$),
        the equality
        \begin{equation}\label{eq:pconvTT}
        \begin{split}
            \overline{T}\vec{f}(x) = C \sum_{Q\in \mathcal{G}}a_Q\chi_Q(x)\int_Q{\overline{K}_Q(x,y)\vec{f}(y)dy}
        \end{split}
        \end{equation}
        is satisfied a.e. on $\Omega$.
    \end{itemize}
    
    Then, given any function $\vec{f}\in E_n$,  
    and $\vec{B}=(B_1,...,B_m)$, a vector with $m$ $n\times n$ matrix functions defined in $\Omega$, 
    such that $\vec{B}_{\sigma^t}\vec{f}\in E_n$
    for all $\sigma\in C(m)$,
     there exist $C>0$,
    a collection of finite measure subsets of $\Omega$, $\mathcal{G}$, 
    coefficients $\{a_Q\}_{Q\in \mathcal{G}}$,
    and functions $K_Q:Q\times Q\rightarrow \mathcal{M}_n(\mathbb{F})$
    with $\|K_Q(x,\cdot)\|_{L^{r'}\left(\frac{d \mu}{|Q|}\right)}\leq 1$ a.e. $x\in Q$, for $Q\in \mathcal{G}$ (all provided by 
    $(P_1)$ applied to a function dependent on $\vec{f}$ and $\vec{B}$), such that
    \begin{equation}\label{paux333}
    \begin{split}
        \vec{T}_{\vec{B}}\vec{f}(x)=C \sum_{Q\in \mathcal{G}}a_Q\chi_Q(x)\sum_{\sigma\in C(m)}(-1)^{m-|\sigma|}  \vec{B}_{\sigma}(x)\int_Q{K_Q(x,y)\vec{B}_{(\sigma^c)^t}(y)\vec{f}(y)dy}
    \end{split}
    \end{equation}
    for a.e. $x\in \Omega$.
\end{theorem}

\begin{proof}
    The proof will be both an extension and a simplification of \cite[Theorem 4]{SA1EfVVO} and \cite[Theorem 1.8]{Cit2SaMWS}
by Isralowitz, Pott and Treil. 
    By Lemma \ref{lemma:comm_exp}, we know that 
    \begin{equation*}
        \begin{split}
            \vec{T}_{\vec{B}}\vec{f}(x) = \Psi(x)\overline{T}(\tilde{\Psi}\vec{f})(x),
        \end{split}
        \end{equation*}
        and by the hypotheses, it is also clear that 
        $\tilde{\Psi}\vec{f}\in E_{2^mn}$, so, by $(P_1)$, 
        we know there exist $C>0$, a collection $\mathcal{G}$, coefficients $\{a_Q\}_{Q\in \mathcal{G}}$, and 
        functions $K_Q:Q\times Q\rightarrow \mathcal{M}_n(\mathbb{F})$ with $\|K_Q(x,\cdot)\|_{L^{r'}\left(\frac{d \mu}{|Q|}\right)}\leq 1$ a.e. $x\in Q$, for $Q\in \mathcal{G}$, such that 
        \begin{equation*}
        \begin{split}
            \overline{T}(\tilde{\Psi}\vec{f})(x) = C \sum_{Q\in \mathcal{G}}a_Q\chi_Q(x)\begin{pmatrix}
            \int_Q{K_Q(x,y)\vec{B}_{\tilde{\sigma}_{2^m}^t}(y)\vec{f}(y)dy}\\
            \vdots\\
            \int_Q{K_Q(x,y)\vec{B}_{(\tilde{\sigma}_{i}^c)^t}(y)\vec{f}(y)dy}  \\
                \vdots\\
                \int_Q{K_Q(x,y)\vec{f}(y)dy}
            \end{pmatrix}.
        \end{split}
        \end{equation*}
        As a result,
        \begin{equation*}
        \begin{split}
            \vec{T}_{\vec{B}}\vec{f}(x) = \Psi(x)\overline{T}(\tilde{\Psi}\vec{f})(x)
            &= C \sum_{Q\in \mathcal{G}}a_Q\chi_Q(x) \Psi(x) \begin{pmatrix}
                \int_Q{K_Q(x,y)\vec{B}_{\tilde{\sigma}_{2^m}^t}(y)\vec{f}(y)dy}\\
                \vdots\\
                \int_Q{K_Q(x,y)\vec{B}_{(\tilde{\sigma}_{i}^c)^t}(y)\vec{f}(y)dy}  \\
                    \vdots\\
                    \int_Q{K_Q(x,y)\vec{f}(y)dy}
                \end{pmatrix}
                \\ & =  C\sum_{Q\in \mathcal{G}}a_Q\chi_Q(x) 
                \sum_{k_0=0}^m (-1)^{m-k_0}\sum_{\sigma\in C_{k_0}(m)} \vec{B}_{\sigma}(x)\int_Q{K_Q(x,y)\vec{B}_{({\sigma}^c)^t}(y)\vec{f}(y)dy},
        \end{split}
        \end{equation*}
        concluding the proof.
\end{proof}


\begin{theorem}\label{th:tocho3gen}
    Let $\Omega$ be a measurable subset of some measure space $(X,\mathcal{M},\mu)$,
     $\mathbb{F}=\mathbb{R}$ or $\mathbb{C}$, and
    $r,s\in[1,\infty]$.
    Also consider $\{E_k\}_{k=1}^\infty$, $\{F_k\}_{k=1}^\infty$ collections of $k-$dimensional function spaces
   such that 
    $E_{k_1}\times E_{k_2} \subset E_{k_1+k_2}$ and $F_{k_1}\times F_{k_2} \subset F_{k_1+k_2}$ whenever $k_1,k_2\in \mathbb{N}$,
    and
    assume ${T}$ is a scalar linear operator with the following property:
    \begin{itemize}
        \item[$(P_2)$] For every $\vec{f}\in E_{2^mn}$
         and $\vec{g}\in F_{2^mn}$,
        there exist $C>0$ and
        a collection of subsets of $\Omega$, $\mathcal{G}$, such that
        $E_{2^mn}\subset L^r((Q,\mu),\mathbb{F}^{2^mn})$ and $F_{2^mn}\subset L^s((Q,\mu),\mathbb{F}^{2^mn})$
        for all $Q\in \mathcal{G}$, and the inequality
        \begin{equation*}
        \begin{split}
            \int_{\mathbb{R}^d} \left\lvert \esc{{T\otimes I_{2^mn}}\vec{f}}{\vec{g}}\right\rvert 
        &\leq C \sum_{Q\in \mathcal{G}}|Q| 
        \iinteg{\vec{f}}_{r,Q}\iinteg{\vec{g}}_{s,Q}
        \end{split}
        \end{equation*}
        is satisfied.
    \end{itemize}
   
    Then, given functions $\vec{f}\in E_n$, $\vec{g}\in F_n$ 
    and $\vec{B}=(B_1,...,B_m)$, a vector with $m$ $n\times n$ matrix functions defined in $\Omega$, 
    such that $\vec{B}_{\sigma^t}\vec{f}\in E_n$ and 
    $\vec{B}^*_{\sigma}\vec{g}\in F_n$ for all $\sigma\in \bigcup_{j=0}^n C_j(m)$,
    there exist $C>0$,
    a collection of subsets of $\Omega$, $\mathcal{G}$, 
    (all provided by 
    $(P_2)$ applied to a function dependent on $\vec{f}$, $\vec{g}$ and $\vec{B}$), such that
    \begin{equation*}
    \begin{split}
        \int_{\Omega}\left\lvert \esc{{(T\otimes I_n)}_{\vec{B}}\vec{f}(x)}{\vec{g}(x)}\right\rvert dx
    \leq C\sum_{Q\in \mathcal{G}}|Q| 
    \iinteg{\tilde{\Psi}\vec{f}}_{r,Q}\iinteg{\Psi^*\vec{g}}_{s,Q}.
    \end{split}
    \end{equation*}
\end{theorem}

\begin{proof}
    By Lemma \ref{lemma:comm_exp}, we know that, using the same notation, 
\begin{equation*}
\begin{split}
    {(T\otimes I_n)}_{\vec{B}}\vec{f}(x) = \Psi(x){T\otimes I_{2^mn}}(\tilde{\Psi}\vec{f})(x),
\end{split}
\end{equation*}
so
\begin{equation*}
\begin{split}
    \int_{\Omega}\left\lvert \esc{{(T\otimes I_n)}_{\vec{B}}\vec{f}}{\vec{g}}\right\rvert
    &=  \int_{\Omega}\left\lvert \esc{\Psi{T\otimes I_{2^mn}}(\tilde{\Psi}\vec{f})}{\vec{g}}\right\rvert
    = \int_{\Omega}\left\lvert \esc{{T\otimes I_{2^mn}}(\tilde{\Psi}\vec{f})}{\Psi^*\vec{g}}\right\rvert.
\end{split}
\end{equation*}
Now, observe that 
\begin{equation*}
\begin{split}
    \tilde{\Psi}(x)\vec{f}(x) = \begin{pmatrix}
     \vec{B}_{\tilde{\sigma}_{2^m}^t }(x)\vec{f}(x)\\
    \hline\\
    \vdots\\
    \hline\\
    \vec{B}_{(\tilde{\sigma}_{j}^c)^t}(x)\vec{f}(x)\\
    \hline\\
    \vdots\\
    \hline\\
    \vec{f}(x)
    \end{pmatrix},\quad
    \Psi^*(x)\vec{g}(x) = \begin{pmatrix}
    (-1)^m\vec{g}(x)\\
    \hline\\
    \vdots\\
    \hline\\
    (-1)^{m-k_0(j)}\vec{B}^*_{\tilde{\sigma}_{j}}(x)\vec{g}(x)\\
    \hline\\
    \vdots\\
    \hline\\
    \vec{B}^*_{\tilde{\sigma}_{2^m}}(x)\vec{g}(x)
    \end{pmatrix},
\end{split}
\end{equation*}
so both functions satisfy the hypotheses of property $(P_2)$, and as a result 
\begin{equation*}
\begin{split}
    \int_{\Omega}\left\lvert \esc{{T\otimes I_{2^mn}}(\tilde{\Psi}\vec{f})}{\Psi^*(x)\vec{g}}\right\rvert 
    &\leq C\sum_{Q\in \mathcal{G}}|Q| 
        \iinteg{\tilde{\Psi}\vec{f}}_{r,Q}\iinteg{\Psi^*\vec{g}}_{s,Q}.
\end{split}
\end{equation*}
\end{proof}

\subsection{Proof of the strong-type estimates}\label{sect:str}

We will start this section by obtaining two direct corollaries of 
Lemma \ref{lemma:perm2}. The first one will allow us to rewrite the sum of operators 
provided by Corollary \ref{cor:Tochodiago} adequately, and the second one will do the same 
for Theorem \ref{th:tocho3}.


\begin{corollary}\label{cor:a}
    Let $Q$ be a set of finite Lebesgue measure, and
    $\vec{B} = (B_{1},...,B_{m})$, where 
    each $B_j\in M_{n}(\mathbb{F})$ is locally integrable. Then
    \begin{equation*}
    \begin{split}
        { \sum_{\sigma\in C(m)}}
        &{ \esc{\sum_{\beta\in C(\sigma^c)}\fint_Q (-1)^{m-|\beta|} (m_Q\vec{B})_{\beta}\vec{B}_{(\sigma^c-\beta)^t}(y)A_1(y)dy}{\sum_{\alpha\in C(\sigma)}\fint_Q (-1)^{m-|\alpha|}(m_Q\vec{B})_{(\sigma-\alpha)^t}^*\vec{B}_{\alpha}^*(x)A_2(x)dx}}
        \\ & { =(-1)^m\esc{\fint_Q  \tilde{\Psi}(y)A_1(y)dy}
        {\fint_Q \Psi^*(x)A_2(x)dx},}
    \end{split}
    \end{equation*}
    provided that $A_1$ and $A_2$ are functions for which the integrals make sense.
\end{corollary}


\begin{corollary}\label{cor:b}
    Let $Q$ be a set of finite Lebesgue measure, $k_Q:Q\times Q\rightarrow \mathbb{F}$, and
    $\vec{B} = (B_{1},...,B_{m})$, where 
    each $B_j\in M_{n}(\mathbb{F})$ is locally integrable. Then
    \begin{equation*}
    \begin{split}
        { \sum_{\sigma\in C(m)}}
        &{ \left(\sum_{\alpha\in C(\sigma)} (-1)^{m-|\alpha|}\vec{B}_{\alpha}(x)(m_Q\vec{B})_{(\sigma-\alpha)^t}\right){\fint_Q \sum_{\beta\in C(\sigma^c)}(-1)^{m-|\beta|} (m_Q\vec{B})_{\beta}\vec{B}_{(\sigma^c-\beta)^t}(y)A_1(y) k_Q(x,y)dy}}
        \\ &  =(-1)^m \sum_{\sigma\in C(m)}(-1)^{m-|\sigma|} \vec{B}_{\sigma}(x)\fint_Q\vec{B}_{(\sigma^c)^t}(y)A_1(y)k_Q(x,y)dy
    \end{split}
    \end{equation*}
    provided that $A_1$ is a function for which the integrals make sense.
\end{corollary}


{In order to prove the first of these corollaries, it suffices to put the integral in $x$ outside of the product, move $\Psi$ to the other side of the product (without the $*$), and apply Lemma \ref{lemma:perm2} with the operator $Tf=\fint_Qf$. The second one is a more immediate consequence of the aforementioned lemma, considering $Tf(x)=\fint_Qk(x,y)f(y)dy$}.
Corollary \ref{cor:a} will be employed to obtain the following result, which sets the stage 
for a boundedness proof. It is based on
Lemma 3 in \cite{muller-rr}.

\begin{lemma}\label{lemma:a}
    If $p,r,s\geq 1$, $U,V$ are weights, 
    $\vec{B}$ is a vector of locally integrable matrices,
    $\mathcal{S}$ is an
    $\eta-$sparse family, and
    $\mathcal{P}=\left\{ P_Q\right\}, \mathcal{R}=\left\{ R_Q\right\}$ are 
    families of positive definite matrices indexed in $\mathcal{S}$,
    \begin{equation*}
    \begin{split} & \sum_{Q\in\mathcal{S}}\iinteg{\tilde{\Psi}U^{\frac{-1}{p}}\vec{f}}_{r,Q}\iinteg{\Psi^{*}V^{\frac{1}{p}}\vec{g}}_{s,Q}|Q|\\
 & \leq\frac{1}{\eta}\sup_{Q}\|R_{Q}P_{Q}\|\sum_{\sigma\in C(m)}\|M_{\mathcal{P},\vec{B},\sigma^{c},U^{\frac{-1}{p}},r}\|_{p}\|M_{\mathcal{R},\vec{B},\sigma,V^{\frac{1}{p}},s,*}\|_{{p'}}\|\vec{f}\|_{L^{p}}\|\vec{g}\|_{L^{p'}},
\end{split}
    \end{equation*}
    where 
    \begin{equation*}
    \begin{split}
        M_{\mathcal{P},\vec{B},\sigma^c,U^{\frac{-1}{p}},r}f(z) &= \sup_{Q\in \mathcal{S},z\in Q}\left(\fint_Q \left\lvert  (P_Q)^{-1}\sum_{\beta\in C(\sigma^c)}(-1)^{m-|\beta|} (m_Q\vec{B})_{\beta}\vec{B}_{(\sigma^c-\beta)^t}(y)U^{\frac{-1}{p}}(y)\vec{f}(y)\right\rvert^r dy\right)^{\frac{1}{r}},
        \\
        M_{\mathcal{R},\vec{B},\sigma,V^{\frac{1}{p}},s,*}g(z) &=\sup_{Q\in \mathcal{S},z\in Q}\left({\fint_Q \left\lvert  (R_Q)^{-1}\sum_{\alpha\in C(\sigma)}(-1)^{m-|\alpha|}(m_Q\vec{B})_{(\sigma-\alpha)^t}^*\vec{B}_{\alpha}^*(x)V^{\frac{1}{p}}(x)\vec{g}(x)\right\rvert^s dx}\right)^{\frac{1}{s}},
    \end{split}
    \end{equation*}
    and $\Psi, \tilde{\Psi}$ are defined as in Lemma \ref{lemma:comm_exp}.
\end{lemma}
\begin{proof}
    By Corollary \ref{cor:a}, we know that, if $\phi\in B_{L^{r'}}, \psi\in B_{L^{s'}}$,
    \begin{align*}
        &{ \left\lvert  \esc{\fint_Q  \tilde{\Psi}(y)U^{\frac{-1}{p}}(y)\vec{f}(y) \phi(y)dy}
        {\fint_Q \Psi^*(x)V^{\frac{1}{p}}(x)\vec{g}(x) \psi(x)dx}\right\rvert}
        \\ &  = \left\lvert  \sum_{\sigma\in C(m)}
        \Big\langle \sum_{\beta\in C(\sigma^c)}\fint_Q (-1)^{m-|\beta|}( m_Q\vec{B})_{\beta}\vec{B}_{(\sigma^c-\beta)^t}(y)U^{\frac{-1}{p}}(y)\vec{f}(y) \phi(y)dy
        ,\right.\\ & 
        \left.\sum_{\alpha\in C(\sigma)}\fint_Q (-1)^{m-|\alpha|}(m_Q\vec{B})_{(\sigma-\alpha)^t}^*\vec{B}_{\alpha}^*(x)V^{\frac{1}{p}}(x)\vec{g}(x) \psi(x)dx
        \Big\rangle
        \right\rvert.
    \end{align*}
    If we now observe that $R_Q$ and $P_Q$ are positive definite matrices, 
    \begin{align*}
        &\left\lvert  \sum_{\sigma\in C(m)}
        \Big\langle{\fint_Q \sum_{\beta\in C(\sigma^c)}(-1)^{m-|\beta|} (m_Q\vec{B})_{\beta}\vec{B}_{(\sigma^c-\beta)^t}(y)U^{\frac{-1}{p}}(y)\vec{f}(y) \phi(y)dy},
        \right. \\ & \;\;\;\;\left.{\fint_Q \sum_{\alpha\in C(\sigma)}(-1)^{m-|\alpha|}(m_Q\vec{B})_{(\sigma-\alpha)^t}^*\vec{B}_{\alpha}^*(x)V^{\frac{1}{p}}(x)\vec{g}(x) \psi(x)dx}\Big\rangle
        \right\rvert
        \\ &  =\left\lvert  \sum_{\sigma\in C(m)}
        \Big\langle{\fint_Q P_Q(P_Q)^{-1}\sum_{\beta\in C(\sigma^c)}(-1)^{m-|\beta|} (m_Q\vec{B})_{\beta}\vec{B}_{(\sigma^c-\beta)^t}(y)U^{\frac{-1}{p}}(y)\vec{f}(y) \phi(y)dy},
        \right. \\ &  \;\;\;\; \left.{\fint_Q R_Q(R_Q)^{-1}\sum_{\alpha\in C(\sigma)}(-1)^{m-|\alpha|}(m_Q\vec{B})_{(\sigma-\alpha)^t}^*\vec{B}_{\alpha}^*(x)V^{\frac{1}{p}}(x)\vec{g}(x) \psi(x)dx}\Big\rangle
        \right\rvert
        \\ &  =\left\lvert  \sum_{\sigma\in C(m)}
        \Big\langle{\fint_Q R_QP_Q(P_Q)^{-1}\sum_{\beta\in C(\sigma^c)}(-1)^{m-|\beta|} (m_Q\vec{B})_{\beta}\vec{B}_{(\sigma^c-\beta)^t}(y)U^{\frac{-1}{p}}(y)\vec{f}(y) \phi(y)dy},
        \right. \\ & \;\;\;\; \left.{\fint_Q (R_Q)^{-1}\sum_{\alpha\in C(\sigma)}(-1)^{m-|\alpha|}(m_Q\vec{B})_{(\sigma-\alpha)^t}^*\vec{B}_{\alpha}^*(x)V^{\frac{1}{p}}(x)\vec{g}(x) \psi(x)dx}\Big\rangle
        \right\rvert
        \\ &  \leq  \sum_{\sigma\in C(m)}
        \left({\fint_Q \left\lvert R_QP_Q(P_Q)^{-1}\sum_{\beta\in C(\sigma^c)}(-1)^{m-|\beta|} (m_Q\vec{B})_{\beta}\vec{B}_{(\sigma^c-\beta)^t}(y)U^{\frac{-1}{p}}(y)\vec{f}(y) \phi(y)\right\rvert dy}\right)
        \\ & \;\;\;\;\left({\fint_Q \left\lvert  (R_Q)^{-1}\sum_{\alpha\in C(\sigma)}(-1)^{m-|\alpha|}(m_Q\vec{B})_{(\sigma-\alpha)^t}^*\vec{B}_{\alpha}^*(x)V^{\frac{1}{p}}(x)\vec{g}(x) \psi(x)\right\rvert dx}\right)
        \\ &  \leq \sup_{Q}\|R_QP_Q\| \sum_{\sigma\in C(m)}
        \left({\fint_Q \left\lvert (P_Q)^{-1}\sum_{\beta\in C(\sigma^c)}(-1)^{m-|\beta|} (m_Q\vec{B})_{\beta}\vec{B}_{(\sigma^c-\beta)^t}(y)U^{\frac{-1}{p}}(y)\vec{f}(y) \phi(y)\right\rvert dy}\right)
        \\ & \;\;\;\;\left({\fint_Q \left\lvert  (R_Q)^{-1}\sum_{\alpha\in C(\sigma)}(-1)^{m-|\alpha|}(m_Q\vec{B})_{(\sigma-\alpha)^t}^*\vec{B}_{\alpha}^*(x)V^{\frac{1}{p}}(x)\vec{g}(x) \psi(x)\right\rvert dx}\right)
        \\ &  \leq \sup_{Q}\|R_QP_Q\| \sum_{\sigma\in C(m)}
        \left({\fint_Q \left\lvert (P_Q)^{-1}\sum_{\beta\in C(\sigma^c)}(-1)^{m-|\beta|} (m_Q\vec{B})_{\beta}\vec{B}_{(\sigma^c-\beta)^t}(y)U^{\frac{-1}{p}}(y)\vec{f}(y)\right\rvert^r dy}\right)^{\frac{1}{r}}
        \\ & \;\;\;\;\left({\fint_Q \left\lvert  (R_Q)^{-1}\sum_{\alpha\in C(\sigma)}(-1)^{m-|\alpha|}(m_Q\vec{B})_{(\sigma-\alpha)^t}^*\vec{B}_{\alpha}^*(x)V^{\frac{1}{p}}(x)\vec{g}(x)\right\rvert^s dx}\right)^{\frac{1}{s}}
        \\ &  \leq \sup_{Q}\|R_Q P_Q\| \sum_{\sigma\in C(m)}
        \inf_{z\in Q}M_{\mathcal{P},\vec{B},\sigma^c,U^{\frac{-1}{p}},r}\vec{f}(z)
        \inf_{z\in Q}M_{\mathcal{R},\vec{B},\sigma,V^{\frac{1}{p}},s,*}\vec{g}(z).
    \end{align*}
    Thus,
    \begin{align*}
        { \sum_{Q\in \mathcal{S}}}&{\iinteg{\tilde{\Psi}U^{\frac{-1}{p}}\vec{f}}_{r,Q}\iinteg{\Psi^* V^{\frac{1}{p}}\vec{g}}_{s,Q}|Q|
        }\\ & 
        { \leq \frac{1}{\eta}\sum_{Q\in \mathcal{S}}\sup_{Q}\|R_QP_Q\| \sum_{\sigma\in C(m)}
        \inf_{z\in Q}M_{\mathcal{P},\vec{B},\sigma^c,U^{\frac{-1}{p}},r}\vec{f}(z)
        \inf_{z\in Q}M_{\mathcal{R},\vec{B},\sigma,V^{\frac{1}{p}},s,*}\vec{g}(z)|E_Q|}
    \\ & { \leq  \frac{1}{\eta}\sum_{Q\in \mathcal{S}}\sup_{Q}\|R_QP_Q\|\sum_{\sigma\in C(m)}
    \int_{E_Q}  M_{\mathcal{P},\vec{B},\sigma^c,U^{\frac{-1}{p}},r}\vec{f}
    M_{\mathcal{R},\vec{B},\sigma,V^{\frac{1}{p}},s,*}\vec{g} }
    \\ & { \leq   \frac{1}{\eta}\sup_{Q}\|R_QP_Q\|\sum_{\sigma\in C(m)}
     \|M_{\mathcal{P},\vec{B},\sigma^c,U^{\frac{-1}{p}},r}\vec{f}\|_{p}
   \| M_{\mathcal{R},\vec{B},\sigma,V^{\frac{1}{p}},s,*}\vec{g}\|_{{p'}}}.
    \end{align*}
\end{proof}

Likewise, a similar reasoning and Corollary \ref{cor:b} can be applied to arrive at the following result.
\begin{lemma}\label{lemma:rodolfo}
    If $p,r\geq 1$, $U,V$ are weights,
    $\vec{B}$ is a vector of locally integrable matrices,
     $\mathcal{S}$ is an
    $\eta-$sparse family, and
    $\mathcal{P}=\left\{ P_Q\right\}, \mathcal{R}=\left\{ R_Q\right\}$ are 
    families of positive definite matrices indexed in $\mathcal{S}$,
    \begin{equation*}
    \begin{split}
        \int_{\Omega}&\left\lvert  \esc{\sum_{Q\in \mathcal{G} }\chi_Q(x)\sum_{\sigma\in C(m)}(-1)^{m-|\sigma|} \vec{B}_{\sigma}(x)\fint_Q \vec{B}_{(\sigma^c)^t}(y)U^{\frac{-1}{p}}(y)\vec{f}(y)k_Q(x,y)dy}
        {V^{\frac{1}{p}}(x)\vec{g}(x)}
        \right\rvert
        \\ & \leq \frac{1}{\eta}\sup_{Q}\|R_QP_Q\|\sum_{\sigma\in C(m)}
        \|M_{\mathcal{P},\vec{B},\sigma^c,U^{\frac{-1}{p}},r}\|_{p}
      \| M_{\mathcal{R},\vec{B},\sigma,V^{\frac{1}{p}},1,*}\|_{{p'}}\|\vec{f}\|_{L^p}\|\vec{g}\|_{L^{p'}}.
    \end{split}
    \end{equation*}
\end{lemma}
\begin{proof}
    By Corollary \ref{cor:b},
    \begin{equation*}
        \begin{split}
             { (-1)^m} &{ \esc{\sum_{Q\in \mathcal{G} }\chi_Q(x)\sum_{\sigma\in C(m)}(-1)^{m-|\sigma|} \vec{B}_{\sigma}(x)\fint_Q \vec{B}_{(\sigma^c)^t}(y)U^{\frac{-1}{p}}(y)\vec{f}(y)k_Q(x,y)dy}
                {V^{\frac{1}{p}}(x)\vec{g}(x)}}
                \\ & 
                =  \sum_{Q\in \mathcal{G} }\chi_Q(x)\sum_{\sigma\in C(m)}
                 \Big\langle{\sum_{\beta\in C(\sigma^c)}\fint_Q (-1)^{m-|\beta|} (m_Q\vec{B})_{\beta}\vec{B}_{(\sigma^c-\beta)^t}(y)U^{\frac{-1}{p}}(y)\vec{f}(y) k_Q(x,y)dy},
                 \\ & \;\;\;\;{\sum_{\alpha\in C(\sigma)} (-1)^{m-|\alpha|}(m_Q\vec{B})_{(\sigma-\alpha)^t}^*\vec{B}_{\alpha}^*(x)V^{\frac{1}{p}}(x)\vec{g}(x)}\Big\rangle.
        \end{split}
        \end{equation*}
    Thus, following the ideas of the previous proof,
    \begin{align*}
        \int_{\Omega}&\left\lvert  \esc{\sum_{Q\in \mathcal{G} }\chi_Q(x)\sum_{\sigma\in C(m)}(-1)^{m-|\sigma|} \vec{B}_{\sigma}(x)\fint_Q \vec{B}_{(\sigma^c)^t}(y)U^{\frac{-1}{p}}(y)\vec{f}(y)k_Q(x,y)dy}
        {V^{\frac{1}{p}}(x)\vec{g}(x)}
        \right\rvert dx
        \\ &  \leq \int_{\Omega}\sum_{Q\in \mathcal{G} }\chi_Q(x)\sum_{\sigma\in C(m)}\left\lvert  \Big\langle{\sum_{\beta\in C(\sigma^c)}\fint_Q (-1)^{m-|\beta|} (m_Q\vec{B})_{\beta}\vec{B}_{(\sigma^c-\beta)^t}(y)U^{\frac{-1}{p}}(y)\vec{f}(y) k_Q(x,y)dy},\right.
        \\ & \left. \;\;\;{\sum_{\alpha\in C(\sigma)} (-1)^{m-|\alpha|}(m_Q\vec{B})_{(\sigma-\alpha)^t}^*\vec{B}_{\alpha}^*(x)V^{\frac{1}{p}}(x)\vec{g}(x)}\Big\rangle
        \right\rvert dx
        \\ &  \leq \sup_{Q}\|R_QP_Q\|\\ & \;\;\;\; \times \;\sum_{Q\in \mathcal{G} }\int_{\Omega}\chi_Q(x)\sum_{\sigma\in C(m)}
        \left({\fint_Q \left\lvert (P_Q)^{-1}\sum_{\beta\in C(\sigma^c)}(-1)^{m-|\beta|} (m_Q\vec{B})_{\beta}\vec{B}_{(\sigma^c-\beta)^t}(y)U^{\frac{-1}{p}}(y)\vec{f}(y) k_Q(x,y)\right\rvert dy}\right)
        \\ & \;\;\;\;\times \;{ \left\lvert  (R_Q)^{-1}\sum_{\alpha\in C(\sigma)}(-1)^{m-|\alpha|}(m_Q\vec{B})_{(\sigma-\alpha)^t}^*\vec{B}_{\alpha}^*(x)V^{\frac{1}{p}}(x)\vec{g}(x)\right\rvert dx}
        \\ &  \leq \sup_{Q}\|R_QP_Q\| \sum_{Q\in \mathcal{G} }\sum_{\sigma\in C(m)}
        \left({\fint_Q \left\lvert (P_Q)^{-1}\sum_{\beta\in C(\sigma^c)}(-1)^{m-|\beta|} (m_Q\vec{B})_{\beta}\vec{B}_{(\sigma^c-\beta)^t}(y)U^{\frac{-1}{p}}(y)\vec{f}(y)\right\rvert^r dy}\right)^{\frac{1}{r}}
        \\ & \;\;\;\; \times \;\left({\fint_Q \left\lvert  (R_Q)^{-1}\sum_{\alpha\in C(\sigma)}(-1)^{m-|\alpha|}(m_Q\vec{B})_{(\sigma-\alpha)^t}^*\vec{B}_{\alpha}^*(x)V^{\frac{1}{p}}(x)\vec{g}(x)\right\rvert dx}\right)|Q| .
    \end{align*}
    The rest follows from the proof of Lemma \ref{lemma:a}.
\end{proof}

Now we are going to prove Theorem \ref{th:strTocho3}.

\begin{proof}[Proof of Theorem \ref{th:strTocho3}]
Applying Theorem 9 in \cite{muller-rr}, Theorem \ref{th:tocho3} and Lemma \ref{lemma:a}, we know that
\begin{equation}\label{auxpal}
\begin{split}
    \int_{\Omega}&\left\lvert \esc{V^{\frac{1}{p}}(x){(T\otimes I_n)}_{\vec{B}}(U^{\frac{-1}{p}}\vec{h})(x)}{\vec{g}(x)}\right\rvert
    = \int_{\Omega}\left\lvert \esc{{(T\otimes I_n)}_{\vec{B}}(U^{\frac{-1}{p}}\vec{h})(x)}{V^{\frac{1}{p}}(x)\vec{g}(x)}\right\rvert
    \\ & \leq C\sum_{Q\in \mathcal{G}}|Q| 
    \iinteg{\tilde{\Psi}U^{\frac{-1}{p}}\vec{h}}_{1,Q}\iinteg{\Psi^*V^{\frac{1}{p}}\vec{g}}_{s,Q}
    \\ & \leq  C \frac{1}{\eta}\sup_{Q}\|R_QP_Q\|\sum_{\sigma\in C(m)}
    \|M_{\mathcal{P},\vec{B},\sigma^c,U^{\frac{-1}{p}},1}\vec{h}\|_{p}
  \| M_{\mathcal{R},\vec{B},\sigma,V^{\frac{1}{p}},s,*}\vec{g}\|_{p'},
\end{split}
\end{equation}
where $C=c_{n,d}\|\Omega\|_{L^{\infty}(\mathbb{S}^{d-1})}s'$.
Let us choose $P_Q$ such that 
\begin{equation*}
\begin{split}
    |P_Q \vec{e}|\simeq \left(\fint_Q \left\lvert V^{\frac{-1}{p}}(x)\vec{e}\right\rvert^{rp'} \right)^{\frac{1}{rp'}},
\end{split}
\end{equation*}
that is, $P_Q = \mathcal{R}_{Q,{(rp')'},V^{\frac{(rp')'}{p}}}'$,
where $r= 1+\frac{1}{2^{d+11}t_1[V^{\frac{-p'}{p}}]_{A^{sc}_{p',\infty}}}$, $t_1\geq 1$,
and 
$R_Q$ satisfying that
\begin{equation*}
\begin{split}
    |R_Q \vec{e}|\simeq \left(\fint_Q \left\lvert U^{\frac{1}{p}}(x)\vec{e}\right\rvert^{s \gamma p} \right)^{\frac{1}{s \gamma p}},
\end{split}
\end{equation*}
that is, $R_Q = \mathcal{R}_{Q,s \gamma p, U^{s \gamma}}$
where 
$\gamma = 1+\frac{1}{\left(\frac{p'+1}{2}\right)2^{d+11}t_2[U]_{A^{sc}_{p,\infty}}}$
and $s = \left(\frac{p'+1}{2}\frac{1+2^{d+11}t_2[U]_{A^{sc}_{p,\infty}}}{1+\left(\frac{p'+1}{2}\right)2^{d+11}t_2[U]_{A^{sc}_{p,\infty}}}\right)$,
$t_2\geq1$,so that 
$\gamma s = 1 + \frac{1}{2^{d+11}t_2[U]_{A^{sc}_{p,\infty}}}$.
For these choices we have that
\begin{equation*}
\begin{split}
    s' = \left(\frac{p'+1}{2}\right)'(1+ 2^{d+11}t_2 [U]_{A_p,\infty}^{sc})
    \lesssim pt_2[U]_{A^{sc}_{p,\infty}},
\end{split}
\end{equation*}
and that
\begin{equation*}
\begin{split}
    \|P_QR_Q\| \simeq \left(\fint_Q \| V^{\frac{-1}{p}}(x)R_Q\|^{rp'} \right)^{\frac{1}{rp'}}
    \lesssim \left(\fint_Q\| V^{\frac{-1}{p}}(x)R_Q\|^{p'} \right)^{\frac{1}{p'}}
    \simeq \|\mathcal{R}_{Q,p,V}'R_Q\|= \|R_Q\mathcal{R}_{Q,p,V}'\|,
\end{split}
\end{equation*}
where in the first inequality we applied the Reverse Hölder Inequality. 
As a result, 
\begin{equation*}
\begin{split}
    \|R_Q\mathcal{R}_{Q,p,V}'\|
    \simeq \left(\fint_Q \| U^{\frac{1}{p}}(x)\mathcal{R}_{Q,p,V}'\|^{s \gamma p} \right)^{\frac{1}{s \gamma p}}
    \lesssim \left(\fint_Q \| U^{\frac{1}{p}}(x)\mathcal{R}_{Q,p,V}'\|^{ p} \right)^{\frac{1}{p}}
    \simeq \|\mathcal{R}_{Q,p,U}\mathcal{R}_{Q,p,V}'\|,
\end{split}
\end{equation*}
so 
$\|P_QR_Q\|  \lesssim \|\mathcal{R}_{Q,p,U}\mathcal{R}_{Q,p,V}'\| \lesssim [U,V]_{A_p}^{\frac{1}{p}}$.

In order to control the first maximal in the last term of \eqref{auxpal}, observe that, for each $\sigma$, 
\begin{equation*}
\begin{split}
    \fint_Q &\left\lvert  P_Q^{-1}\left(\sum_{\beta\in C(\sigma^c)}(-1)^{m-|\beta|}( m_Q\vec{B})_{\beta}\vec{B}_{(\sigma^c-\beta)^t}(y)\right)U^{\frac{-1}{p}}(y)\vec{h}(y)\right\rvert dy
    \\ & \leq \left(\fint_Q \| P_Q^{-1}\left(\sum_{\beta\in C(\sigma^c)}(-1)^{m-|\beta|}( m_Q\vec{B})_{\beta}\vec{B}_{(\sigma^c-\beta)^t}(y)\right)U^{\frac{-1}{p}}(y)\|^{rp'} dy\right)^{\frac{1}{rp'}}\left( \fint_Q\left\lvert\vec{h}(y)\right\rvert^{(rp')'} dy\right)^{\frac{1}{(rp')'}}
    \\ & =\left(\fint_Q \| U^{\frac{-1}{p}}(y)\left(\sum_{\beta\in C(\sigma^c)}(-1)^{m-|\beta|}( m_Q\vec{B})_{\beta}\vec{B}_{(\sigma^c-\beta)^t}(y)\right)^* P_Q^{-1}\|^{rp'} dy\right)^{\frac{1}{rp'}}\left( \fint_Q\left\lvert\vec{h}(y)\right\rvert^{(rp')'} dy\right)^{\frac{1}{(rp')'}}
    \\ & \lesssim \|\vec{B}\|_{BMO^{(rp')',*}_{V^{\frac{(rp')'}{p}},U^{\frac{(rp')'}{p}},\sigma^c,2}}\left( \fint_Q\left\lvert\vec{h}(y)\right\rvert^{(rp')'} dy\right)^{\frac{1}{(rp')'}}.
\end{split}
\end{equation*}
As a result, we deduce that
\begin{equation*}
\begin{split}
    \|M_{\mathcal{P},\vec{B},\sigma^c,U^{\frac{-1}{p}},1}\|_{p} 
    \lesssim \|\vec{B}\|_{BMO^{(rp')',*}_{V^{\frac{(rp')'}{p}},U^{\frac{(rp')'}{p}},\sigma^c,2}}
    \|M_{(rp')'}\|_{p}
    \lesssim \|\vec{B}\|_{BMO^{(rp')',*}_{V^{\frac{(rp')'}{p}},U^{\frac{(rp')'}{p}},\sigma^c,2}}
    \left(\left(\frac{p}{(rp')'}\right)'\right)^{\frac{1}{(rp')'}},
\end{split}
\end{equation*}
where the rightmost term can be controlled by means of Lemma 6 in \cite{muller-rr},
so that 
\begin{equation*}
\begin{split}
    \left(\left(\frac{p}{(rp')'}\right)'\right)^{\frac{1}{(rp')'}}
    \lesssim t_1^{\frac{1}{p}}[V^{\frac{-p'}{p}}]^{\frac{1}{p}}_{A^{sc}_{\infty,p'}},
\end{split}
\end{equation*}
and, as a result,
\begin{equation*}
\begin{split}
    \|M_{\mathcal{P},\vec{B},\sigma^c,U^{\frac{-1}{p}},r}\|_{L^p}
    \lesssim \|\vec{B}\|_{BMO^{(rp')',*}_{V^{\frac{(rp')'}{p}},U^{\frac{(rp')'}{p}},\sigma^c,2}}
    t_1^{\frac{1}{p}}[V^{\frac{-p'}{p}}]^{\frac{1}{p}}_{A^{sc}_{\infty,p'}}.
\end{split}
\end{equation*}
Likewise, the second maximal in the last term of \eqref{auxpal} can be controlled by observing that
\begin{equation*}
\begin{split}
    &\left(\fint_Q \left\lvert  (R_Q)^{-1}\left(\sum_{\alpha\in C(\sigma)}(-1)^{m-|\alpha|}( m_Q\vec{B})_{(\sigma-\alpha)^t}^*\vec{B}^*_{\alpha}(x)\right)V^{\frac{1}{p}}(x)\vec{g}(x)\right\rvert^s dx\right)^{\frac{1}{s}}
    \\ & \leq \left(\fint_Q \|  (R_Q)^{-1}\left(\sum_{\alpha\in C(\sigma)}(-1)^{m-|\alpha|}( m_Q\vec{B})_{(\sigma-\alpha)^t}^*\vec{B}^*_{\alpha}(x)\right)V^{\frac{1}{p}}(x)\|^{s \gamma p} dx\right)^{\frac{1}{s \gamma p }}\left(\fint_Q |\vec{g}|^{s(\gamma p)'}\right)^{\frac{1}{s (\gamma p)'}}
    \\ & = \left(\fint_Q \|  V^{\frac{1}{p}}(x)\left(\sum_{\alpha\in C(\sigma)}(-1)^{m-|\alpha|}(\vec{B}_{\alpha}(x) m_Q\vec{B})_{(\sigma-\alpha)^t}\right)(R_Q)^{-1}\|^{s \gamma p} dx\right)^{\frac{1}{s \gamma p }}\left(\fint_Q |\vec{g}|^{s(\gamma p)'}\right)^{\frac{1}{s (\gamma p)'}},
    \\ & \lesssim  \|\vec{B}\|_{BMO^{s \gamma p}_{V^{s \gamma},U^{s \gamma},\sigma,1}}   \left(\fint_Q |\vec{g}|^{s(\gamma p)'}\right)^{\frac{1}{s (\gamma p)'}},
\end{split}
\end{equation*}
so that 
\begin{equation*}
\begin{split}
    \| M_{\mathcal{R},\vec{B},\sigma,V^{\frac{1}{p}},s,*}\|_{p'}
    \lesssim \|\vec{B}\|_{BMO^{s \gamma p}_{V^{s \gamma},U^{s \gamma},\sigma,1}} 
    \|M_{s(\gamma p)'}\|_{p'}  
    \lesssim \|\vec{B}\|_{BMO^{s \gamma p}_{V^{s \gamma},U^{s \gamma},\sigma,1}} 
    \left(\left(\frac{p'}{s(\gamma p)'}\right)'\right)^{\frac{1}{s(\gamma p)'}}.
\end{split}
\end{equation*}
An application of Lemmas 6 and 7 in \cite{muller-rr} yields that 
\begin{equation*}
\begin{split}
    \left(\left(\frac{p'}{s(\gamma p)'}\right)'\right)^{\frac{1}{s(\gamma p)'}}
    \lesssim p t_2^{\frac{1}{p'}} [U]_{A^{sc}_{p,\infty}}^{\frac{1}{p'}},
\end{split}
\end{equation*} 
and as a result we deduce that
\begin{equation*}
\begin{split}
    \|M_{\mathcal{R},\vec{B},\sigma,V^{\frac{1}{p}},s,*}\|_{L^p}\lesssim \|\vec{B}\|_{BMO^{s \gamma p}_{V^{s \gamma},U^{s \gamma},\sigma,1}} pt_2^{\frac{1}{p'}}[U]^{\frac{1}{p'}}_{A^{sc}_{\infty,p}}.
\end{split}
\end{equation*}

Putting everything together, we obtain the desired result with 
the first element on the minimum. The other inequality can be 
obtained easily by repeating this argument but with 
the adjoint operator $T^*$ (which is also a rough singular integral), and exchanging the roles of
$\vec{h}$ and $\vec{g}$. Note that 
\begin{equation*}
\begin{split}
    \int_{\Omega}\left\lvert \esc{{(T\otimes I_n)}_{\vec{B}}(U^{\frac{-1}{p}}\vec{h})(x)}{V^{\frac{1}{p}}(x)\vec{g}(x)}\right\rvert
    &= 
    \int_{\Omega}\left\lvert \esc{\Psi(x){(T\otimes I_n)}(\tilde{\Psi}U^{\frac{-1}{p}}\vec{h})(x)}{V^{\frac{1}{p}}(x)\vec{g}(x)}\right\rvert
    \\ & = 
    \int_{\Omega}\left\lvert \esc{\tilde{\Psi}(x)U^{\frac{-1}{p}}(x)\vec{h}(x)}{{(T^*\otimes I_n)}(\Psi^*V^{\frac{1}{p}}\vec{g})(x)}\right\rvert,
\end{split}
\end{equation*}
so Theorem \ref{th:tocho3} cannot be applied directly. However, an identical argument 
shows that 
\begin{equation*}
\begin{split}
    \int_{\Omega}\left\lvert \esc{\tilde{\Psi}(x)U^{\frac{-1}{p}}(x)\vec{h}(x)}{{(T^*\otimes I_n)}(\Psi^*V^{\frac{1}{p}}\vec{g})(x)}\right\rvert
    \leq c_{n,d}\|\Omega\|_{L^\infty(\mathbb{S}^{d-1})}{s}'
    \sum_{Q\in \mathcal{G}}|Q|\iinteg{\Psi^*V^{\frac{1}{p}}\vec{g}}_{1,Q}\iinteg{\tilde{\Psi}U^{\frac{-1}{p}}\vec{h}}_{s,Q}
\end{split}
\end{equation*}
for some sparse family $\mathcal{G}$. The rest follows more or less analogously, 
performing some simple changes. The bound one obtains is
\begin{equation*}
\begin{split}
    { \int_{\Omega}}&{\left\lvert \esc{\tilde{\Psi}(x)U^{\frac{-1}{p}}(x)\vec{h}(x)}{{(T^*\otimes I_n)}(\Psi^*V^{\frac{1}{p}}\vec{g})(x)}\right\rvert}
    \\ & { \lesssim c_{n,d}\|\Omega\|_{L^\infty(\mathbb{S}^{d-1})}t_1 [V^{-\frac{p'}{p}}]_{A_{p',\infty}^{sc}}[U,V]_{A_p}
    t_1^{\frac{1}{p}}[V^{-\frac{p'}{p}}]^{\frac{1}{p}}_{A_{p',\infty}^{sc}}t_2^{\frac{1}{p'}}[U]_{A_{p,\infty}^{sc}}^{\frac{1}{p'}}
    }\\ & \;\;\;\;{\times \;\sum_{\sigma\in C(m)}
    \|\vec{B}\|_{BMO_{V^{s \gamma},U^{s \gamma},\sigma,1}^{s \gamma p}}
    \|\vec{B}\|_{BMO^{(rp')',*}_{V^{\frac{(rp')'}{p}},U^{\frac{(rp')'}{p}},\sigma^c,2}}
    \|h\|_{p}\|g\|_{p'}}.
\end{split}
\end{equation*}


\end{proof}

We continue providing the proof of Theorem \ref{th:strTocho}
\begin{proof}[Proof of Theorem \ref{th:strTocho}]
    We shall assume first that \[\beta_1=\max\left\{\max_{\sigma\in C(m),\sigma\not=\varnothing,\tilde{\sigma}_{2^m}} \|\vec{B}\|_{\widetilde{BMO}_{V,V,\sigma}^{p}}^{\frac{1}{|\sigma|}},\max_{\sigma\in C(m),\sigma\not=\varnothing}\|\vec{B}\|_{\widetilde{BMO}_{V,U,\sigma}^{p}}^{\frac{1}{|\sigma|}}\right\}\leq1.\] 
    
     We will use the block matrix
    \begin{equation*}
    \begin{split}
        \Phi = \begin{pmatrix}
        \Phi_{i,j}
        \end{pmatrix}_{i,j=1}^{2^m}
    \end{split}
    \end{equation*}
    where 
    \begin{equation}\label{eq:matrix_def}
    \begin{split}
        \Phi_{i,j} = \begin{cases}
            W_i^{\frac{1}{p}}
            \vec{B}_{\tilde{\sigma}_j-\tilde{\sigma}_i} & \text{if }\tilde{\sigma}_i \leq \tilde{\sigma}_j, 1\leq i\leq j\leq2^m\\
            0 & \text{otherwise}
        \end{cases}
    \end{split}
    \end{equation}
    and $W_1=W_2=...=W_{2^m-1}=V$, $W_{2^m}=U$.
    It is easy to see that 
    \begin{equation*}
    \begin{split}
        \Phi^{-1} = \begin{pmatrix}
        \widetilde{\Phi}_{i,j}
        \end{pmatrix}_{i,j=1}^{2^m},
    \end{split}
    \end{equation*}
    where 
    \begin{equation*}
    \begin{split}
        \widetilde{\Phi}_{i,j} = 
        \begin{cases}
           (-1)^{|\tilde{\sigma}_j-\tilde{\sigma}_i|}\vec{B}_{(\tilde{\sigma}_j-\tilde{\sigma}_i)^t}W_j^{\frac{-1}{p}}
             & \text{if }\tilde{\sigma}_i \leq \tilde{\sigma}_j , 1\leq i\leq j\leq 2^m\\
            0 & \text{otherwise}
        \end{cases}
    \end{split}
    \end{equation*}
    and that, in particular, 
    \begin{equation*}
    \begin{split}
        \Phi(x)\Phi^{-1}(y) = 
        \begin{pmatrix}
        \widetilde{\widetilde{\Phi}}_{i,j}(x,y)
        \end{pmatrix}_{i,j=1}^{2^m},
    \end{split}
    \end{equation*}
    where 
    \begin{equation*}
    \begin{split}
        \widetilde{\widetilde{\Phi}}_{i,j}(x,y)=
        \begin{cases}
            W_i^{\frac{1}{p}}(x)\left( \sum_{\sigma \in  C(\tilde{\sigma}_j-\tilde{\sigma}_i)}(-1)^{
                |\tilde{\sigma}_j-\tilde{\sigma}_i|-
                |\sigma|}\vec{B}_{\sigma}(x)\vec{B}_{((\tilde{\sigma}_j-\tilde{\sigma}_i)-\sigma)^t}(y)\right)W_j^{\frac{-1}{p}}(y)
             & \text{if } \tilde{\sigma}_i \leq \tilde{\sigma}_j \\
            0 & \text{otherwise}
        \end{cases}
    \end{split}
    \end{equation*}
 (the fact that the matrices are inverses is a consequence of Lemma \ref{lemma:tralalero} and the previous expression).
    Notice that, when $\tilde{\sigma}_i \leq \tilde{\sigma}_j$, we made a change of variable to express the 
    corresponding matrix entries appropriately. Namely, we did $\sigma = \tilde{\sigma}-\tilde{\sigma}_i$, where 
    $\tilde{\sigma}$ is the variable obtained by multiplying the matrices, which satisfies 
    $\tilde{\sigma}_i\leq \tilde{\sigma} \leq \tilde{\sigma}_j$.
    Observe as well that the top right block of 
    $\Phi (T\otimes I_{2^mn})\Phi^{-1}$
    coincides with $V^{\frac{1}{p}} (T\otimes I_n)_{\vec{B}} U^{\frac{-1}{p}}$, since it 
    is, in the notation of Lemma \ref{lemma:comm_exp}, $V^{\frac{1}{p}}\Psi\overline{T}(\tilde{\Psi})U^{\frac{-1}{p}}$.

    If we now consider $W = (\Phi^* \Phi)^{\frac{p}{2}}$, and apply polar decomposition, we deduce that 
    $\Phi = \mathcal{U} W^{\frac{1}{p}}$, where $\mathcal{U}$ is some unitary-valued matrix function almost everywhere.
    Observe that, using that for a block matrix $A=(A_{i,j})_{i,j=1,...,n}$, $\|A\|\leq \sum_{i,j=1}^n\|A_{i,j}\|$,
    and applying brute-force to count the blocks,
    \begin{align*}
        { \fint_Q }&{\left(\fint_Q \|W^{\frac{1}{p}}(x)W^{\frac{-1}{p}}(y)\|^{p'}dy\right)^{\frac{p}{p'}}dx
        =  \fint_Q \left(\fint_Q \|\Phi(x)\Phi^{-1}(y)\|^{p'}dy\right)^{\frac{p}{p'}}dx}
        \\ &{\leq 4^{{mp}{}}\left((2^{m}-1)[V]_{A_p}+[U]_{A_p}   
        +{4^m\sum_{\underrel{\sigma\neq \varnothing,\tilde{\sigma}_{2^m}}{\sigma\in C(m)}}\|\vec{B}\|_{\widetilde{BMO}^p_{V,V,\sigma}}}    + \sum_{\underrel{\sigma\neq \varnothing}{\sigma\in C(m)}}\|\vec{B}\|_{\widetilde{BMO}^p_{V,U,\sigma}}\right)}
        \\ &\leq\left((2^{m}-1)[U]_{A_{p}}+[V]_{A_{p}}+4^{m}\sum_{\underrel{\sigma\neq\varnothing,\tilde{\sigma}_{2^{m}}}{\sigma\in C(m)}}1+\sum_{\underrel{\sigma\neq \varnothing}{\sigma\in C(m)}}1\right) \leq c_{m,p}\left([U]_{A_{p}}+[V]_{A_{p}}\right),
        \end{align*}
    so that 

    \[[W]_{A_p}\leq c_{m,p}\left([U]_{A_{p}}+[V]_{A_{p}}\right).\]



    As a result, we conclude that 
    \begin{align}\label{eq:conj_estimate}
        \nonumber\|(T\otimes I_n)_{\vec{B}}\|_{L^p(U)\rightarrow L^p(V)}
        &= \|V^{\frac{1}{p}}(T\otimes I_n)_{\vec{B}}U^{\frac{-1}{p}}\|_{L^p\rightarrow L^p}
        \leq \|\Phi(T\otimes I_{2^mn})\Phi^{{-1}}\|_{L^p\rightarrow L^p}
    \\ & = \|W^{\frac{1}{p}}(T\otimes I_{2^mn})W^{\frac{-1}{p}}\|_{L^p\rightarrow L^p}
    \leq \phi([W]_{A_p})
    \leq\phi\left(c_{m,p}\left([U]_{A_{p}}+[V]_{A_{p}}\right)\right).
    \end{align}
    Now if we call $\vec{B}/\beta_1^\frac{1}{p}=\frac{1}{\beta_1^\frac{1}{p}}(B_{1},\dots,B_{m})$ then
we have that $\|\vec{B}/\beta_1^\frac{1}{p}\|_{\widetilde{BMO}_{V,U,\sigma}^{p}}=\frac{1}{\beta_1^{|\sigma|}}\|\vec{B}\|_{\widetilde{BMO}_{V,U,\sigma}^{p}}\leq1$
by the definition of $\beta_1${, and the same applies to the other possible choices of weights}. Consequently,
\begin{align*}
\|(T\otimes I_{n})_{\vec{B}}\vec{f}\|_{L^{p}(V)} & =\|(T\otimes I_{n})_{\vec{B}/\beta_1^\frac{1}{p}}(\beta_1^{\frac{m}{p}}\vec{f})\|_{L^{p}(V)}\leq\phi(c_{m,p}([U]_{A_{p}}+[V]_{A_{p}}))\|\beta_1^{\frac{m}{p}}\vec{f}\|_{L^{p}(U)}\\
 & =\beta_1^{\frac{m}{p}}\phi(c_{m,p}([U]_{A_{p}}+[V]_{A_{p}}))\|\vec{f}\|_{L^{p}(U)}.
\end{align*}
This proves the first part of the desired estimate.

The other part can be obtained easily by replacing the blocks in \eqref{eq:matrix_def} by the ones given by
\begin{align*}
    \Phi_{i,j} = \begin{cases}
            \widetilde{W}_i^{\frac{1}{p}}
            \vec{B}_{\tilde{\sigma}_j-\tilde{\sigma}_i} & \text{if }\tilde{\sigma}_i \leq \tilde{\sigma}_j, 1\leq i\leq j\leq2^m\\
            0 & \text{otherwise},
            \end{cases}
\end{align*}
where $\widetilde{W}_1=V$ and $\widetilde{W}_k=U$ for $k=2,...,2^m$,
and following an identical procedure.
\end{proof}

\begin{remark}\label{rem:allcomm}
    Observe that, in the notation of the proof, considering $\Phi$ as in 
    \eqref{eq:matrix_def}, the block that takes position $(i,j)$ in
    $\Phi (T\otimes I_{2^mn})\Phi^{-1}$ is $W_i^{\frac{1}{p}}(T\otimes I_{n})_{{\vec{B}^{\tilde{\sigma}_j-\tilde{\sigma}_i}}}W_j^{\frac{-1}{p}}$ (where $\vec{B}^\sigma=(\vec{B}_{\sigma(1)},...,\vec{B}_{\sigma(|\sigma|)})$), assuming $\tilde{\sigma}_i \leq \tilde{\sigma}_j$. As a result, the proof provides boundedness for every possible commutator with respect to $U$ and $V$, and for every commutator of order smaller that $m$ with respect to $V$, since our previous observation allows us to write $\|(T\otimes I_{n})_{{\vec{B}^{\tilde{\sigma}_j-\tilde{\sigma}_i}}}\|_{L^p(W_j)\rightarrow L^p(W_i)}$ in place of $ \|(T\otimes I_n)_{\vec{B}}\|_{L^p(U)\rightarrow L^p(V)}$ in \eqref{eq:conj_estimate}. This provides a worse constant compared to applying the result to these commutators directly, 
    but illustrates that, by asking for all the possible $BMO$ norms to be bounded with respect to $V$, and $U$ and $V$, we actually get boundedness for all the lower order commutators with respect to the corresponding weights.
    The same conclusion can be drawn by changing $\Phi$ as in the end of the proof, but in terms of $U$.
\end{remark}

Finally in this section we prove Theorem \ref{HWMatrixThm}

\begin{proof}[Proof of Theorem \ref{HWMatrixThm}]
The proof is similar to the proof of Theorem \ref{th:strTocho} in conjunction with Theorem \ref{StrongHigherJN}.  Similar to \cite{PottTalk},  define the $(m+1) \times (m+1)$ matrix $\Phi$ with $n \times n$ blocks by

 \begin{equation*}
    \begin{split}
        {\Phi}_{i,j}(x)=
        \begin{cases}
         \binom{m+1-i}{j-i}{b^{j-i} (x)} W_i ^\frac{1}{p} (x) 
            &\text{if } j \geq i \\            
            0 & \text{otherwise}
        \end{cases}
    \end{split}
    \end{equation*}
where $W_1 = W_2= \cdots = W_{m} = V$ and $W_{m+1} = U.$

Then 
 \begin{equation*}
    \begin{split}
        {\Phi}_{i,j} ^{-1} (y)=
        \begin{cases}
          \binom{m+1-i}{j-i}(-1)^{j-i} {b^{j-i} (y)} W_j ^{-\frac{1}{p}} (y) 
            &\text{if } j \geq i \\            
            0 & \text{otherwise}
        \end{cases}
    \end{split}
    \end{equation*}

\noindent and the binomial theorem tells us that $$[\Phi(x) \Phi^{-1}(y)]_{ij} = \binom{m+1-i}{j-i} W_i ^{\frac{1}{p}} (x) W_j ^{-\frac{1}{p}} (y) (b(x) - b(y))^{j-i}.$$ Moreover, again observe that the top right block of 
    $\Phi (T\otimes I_{(m+1)n})\Phi^{-1}$
    coincides with $V^{\frac{1}{p}} (T\otimes I_n)_{b} ^m U^{\frac{-1}{p}}$.  Also again  consider $W = (\Phi^* \Phi)^{\frac{p}{2}}$, and applying polar decomposition, we deduce that 
    $\Phi = \mathcal{U} W^{\frac{1}{p}}$, where $\mathcal{U}$ is some unitary-valued matrix function almost everywhere.   Applying brute-force to count the blocks,
    \begin{align*}
        { \fint_Q }&{\left(\fint_Q \|W^{\frac{1}{p}}(x)W^{\frac{-1}{p}}(y)\|^{p'}dy\right)^{\frac{p}{p'}}dx
        =  \fint_Q \left(\fint_Q \|\Phi(x)\Phi^{-1}(y)\|^{p'}dy\right)^{\frac{p}{p'}}dx}
        \\ &  \lesssim  [U]_{\text{A}_p} +  [V]_{\text{A}_p} + \sum_{i < j < m } \|b\|_{\widetilde{BMO}_{V, V, (j-i)}^p}  + \sum_{i = 0}^{m-1} \|b\|_{\widetilde{BMO}_{V, U, (m-i)}^p}  \end{align*}

        However, if $U = V$ then by Theorem \ref{StrongHigherJN} we have for any $j \in \mathbb{N}$ that $ \|b\|_{\widetilde{BMO}_{V, V, j}^p} \approx \|b\|_{{BMO}_{V, V, j}^p} ^{jp} = \|b\|_{BMO} ^{jp}$. Rescalling $b \mapsto b/\|b\|_{\text{BMO}}$ similar to the proof of Theorem \ref{th:strTocho} tells us that $$ \|(T\otimes I_n)_{b} ^m  \|_{L^p(U) \rightarrow L^p(U)} \leq C \|b\|_{\text{BMO}}^m. $$   
        

        Finally,  as in the proof of Lemma $1.3$ in \cite{Cit2SaMWS},  define the $2 \times 2$ block matrix function $\Phi$ by \begin{equation}  \widetilde{\Phi} = \left(\begin{array}{cc} V^\frac{1}{p} & V^\frac{1}{p}  B \\ 0 & U^\frac{1}{p} \end{array} \right) \label{Phi} \end{equation}  so that  \begin{equation*} \widetilde{\Phi}^{-1} =  \left(\begin{array}{cc} V^{-\frac{1}{p}} & - B U^{-\frac{1}{p}}  \\ 0 & U^{-\frac{1}{p}}. \end{array} \right) \end{equation*} and 
        
        \begin{align*} \widetilde{\Phi} \left((T\otimes I_n)_{b} ^m\right)  \widetilde{\Phi}^{-1} & = \widetilde{\Phi}  \left( \begin{array}{cc}\left((T\otimes I_n)_{b} ^m\right) & 0 \\ 0 & \left((T\otimes I_n)_{b} ^m\right) \end{array}\right) \widetilde{\Phi}^{-1} 
        \\ & = \left(\begin{array}{cc} V^\frac{1}{p} \left((T\otimes I_n)_{b} ^m\right) V^{-\frac{1}{p}} & V^\frac{1}{p}  \left((T\otimes I_n)_{b} ^{m+1}\right) U^{-\frac{1}{p}}   \\ 0 & U^{\frac{1}{p}} \left((T\otimes I_n)_{b} ^m\right) U^{-\frac{1}{p}} \end{array} \right). \end{align*}

Let   $\widetilde{W} = (\widetilde{\Phi} ^* \widetilde{\Phi})^\frac{p}{2} $.  Then using the polar decomposition, we can write \begin{equation*} \widetilde{\Phi} =  \widetilde{\mathcal{U}} \widetilde{W}^\frac{1}{p} \end{equation*} where $\mathcal{U}$ is unitary valued a.e.  Thus, we have

\begin{align*} \|(T\otimes I_{2n})_{b} ^{m+1}\|_{L^p(U) \rightarrow L^p(V)} & = \|V^\frac{1}{p}  (T\otimes I_n)_{b} ^m U^{-\frac{1}{p}} \|_{L^p \rightarrow L^p}
\\ & \leq \|\widetilde{\Phi} \left(T\otimes I_n)_{b} ^m\right)  \widetilde{\Phi}^{-1}\|_{L^p \rightarrow L^p}
\\ & = \|\widetilde{W}^\frac{1}{p} \left(T \otimes {I}_{2m}\right)  \widetilde{W}^{-\frac{1}{p}}\|_{L^p \rightarrow L^p}
\\ & = \| (T\otimes I_n)_{b} ^m  \|_{L^p(\widetilde{W}) \rightarrow L^p(\widetilde{W})}
\\ & \leq \|b\|_{\text{BMO}}^m \phi([\widetilde{W}]_{\text{A}_p}) 
\\ & \leq \|b\|_{\text{BMO}}^m \phi([U]_{\text{A}_p} + [V]_{\text{A}_p} + \|b\|_{\widetilde{\text{BMO}}_{U, V, 1}^p}).  
\end{align*}
But by  Theorem \ref{StrongHigherJN} (or Corollary $4.7$ in \cite{Cit2SaMWS}) we have $\|b\|_{\widetilde{\text{BMO}}_{U, V, 1}^p} \approx \|b\|_{{\text{BMO}}_{U, V, 1}^p} ^p$ so rescalling $b \mapsto b / \|b\|_{{\text{BMO}}_{U, V, 1}^p}$ completes the proof. 

\end{proof}
\nocite{MR805957}

\begin{remark} \label{LOR19Rem} Note that Theorem \ref{HWMatrixThm} could also be proved by an induction argument and the results in \cite{Cit2SaMWS}.  However, it is hoped that the proof above can be modified to  prove \eqref{LOR19Conj} via a clever choice of rescalling weights and symbols.  As was stated in the introduction, we are currently unable to do this.

\end{remark} 



\section{An estimate in terms of bumped weighted $BMO$ spaces}\label{sec:Bumped}

In this section we will prove a boundedness result in terms of Orlicz norms for operators that 
satisfy Theorem \ref{cor:Tochodiago}. We recall that $\Phi:[0,\infty)\rightarrow[0,\infty)$ is a Young function if it is a continuous, nonnegative, strictly increasing and convex function such that $\lim_{t\rightarrow\infty}\Phi(t)=\infty$ and $\Phi(0)=0$. The associated 
Orlicz norm in a cube $Q$ is given by
\begin{align*}
    \|f\|_{\Phi(L)(Q)} = \inf\left\{\lambda>0: \fint_{Q} \Phi\left(\frac{|f|}{\lambda}\right)\leq 1\right\}\cup \{\infty\}.
\end{align*}
It is well known that if $\Phi$ and $\Psi$ are Young functions such that there exists $\kappa>0$ satisfying that \[\Phi^{-1}(t)\Psi^{-1}(t)\leq \kappa t\]for every $t>0$, then the following version of Hölder's inequality holds
\begin{align*}
    \fint_Q |fg|\leq 2\kappa\|f\|_{\Phi(L)(Q)}\|g\|_{{\Psi}(L)(Q)}.
\end{align*}
Additionally, $B_p$ will denote the class of Young functions such that 
\begin{align*}
    \alpha_p(\Phi)=\left(\int_{1}^\infty \frac{\Phi(t)}{t^p} \frac{dt}{t}\right)<\infty.
\end{align*}
It is known (\cite{cposc,HPLlogeps}) that, for a Young function
$\Phi$, the maximal operator given by 
\begin{align*}
    M_{\Phi}f(x) = \sup_{Q\ni x}\|f\|_{\Phi(L)(Q)}
\end{align*}
is bounded in $L^p$ if and only if $\Phi\in B_p$. Furthermore
\[\|M_\Phi f\|_{L^p}\lesssim\alpha_p(\Phi)\|f\|_{L^p}.\]
\nocite{MR4734976}
\nocite{MR4050113}
\nocite{raoren}

Now we introduce and prove the aforementioned boundedness result.
\begin{theorem}\label{th:strTochoMejor}
    Consider ${T}$, $\vec{B}$ and $\vec{f}$ satisfying the hypotheses of Corollary 
    \ref{cor:Tochodiago} with $r=1$, and a measurable function $\vec{g}$. Then, if $C,\overline{C},D$ and $\overline{D}$ are 
    Young functions such that $\overline{C}\in B_p$ and $\overline{D}\in B_{p'}$, $p\in (1,\infty)$, and there exists $\kappa>0$ such that
    \[C^{-1}(t)\overline{C}^{-1}(t)\leq \kappa t\qquad D^{-1}(t)\overline{D}^{-1}(t)\leq \kappa t\]
    for every $t>0$,  then 
    \begin{equation*}
    \begin{split}
        \left\lvert \int_{\mathbb{R}^n} \left\langle (T\otimes I_n)_{\vec{B}}\vec{f}(x),\vec{g}(x) \right\rangle dx\right\rvert
        \lesssim \alpha\min\left\{\|\vec{B}\|_{\widetilde{BMO}_{V,U,m}^{C,D}},\|\vec{B}\|_{\widetilde{BMO}_{V,U,m}^{D,C}}\right\}
        \|\vec{f}\|_{L^p(U)} \|\vec{g}\|_{L^{p'}(V^{\frac{-p'}{p}})}  ,
    \end{split}
    \end{equation*}
    where
    \begin{align*}
    \|\vec{B}\|_{\widetilde{BMO}_{V,U,m}^{C,D}} &= \sup_Q\left\|\left\|\left\|V^{\frac{1}{p}}(x)\left(\sum_{\sigma\in C(m)}(-1)^{m-|\sigma|}  \vec{B}_{\sigma}(x)\vec{B}_{(\sigma^c)^t}(y)\right)U^{\frac{-1}{p}}(y)\right\|\right\|_{{C_y}(L)\left(Q\right)}\right\|_{D_x(L)(Q)},
\\  \|\vec{B}\|_{\widetilde{BMO}_{V,U,m}^{D,C}} &= \sup_Q\left\|\left\|\left\|V^{\frac{1}{p}}(x)\left(\sum_{\sigma\in C(m)}(-1)^{m-|\sigma|}  \vec{B}_{\sigma}(x)\vec{B}_{(\sigma^c)^t}(y)\right)U^{\frac{-1}{p}}(y)\right\|\right\|_{{D_x}(L)\left(Q\right)}\right\|_{C_y(L)(Q)}
\end{align*}
and $\alpha=\alpha_p(\overline{C})\alpha_{p'}(\overline{D})$.
\end{theorem}


\begin{proof}
    Applying Corollary \ref{cor:Tochodiago} in the first equality, we obtain that
    \begin{align*}
        \left\lvert \int_{\mathbb{R}^n} \right. & \left\langle (T\otimes I_n)_{\vec{B}}\vec{f}(x),\left.\vec{g}(x) \vphantom{\int_{\mathbb{R}^n}} \right\rangle dx\right\rvert
        \\ &   =  \left\lvert \int_{\mathbb{R}^n}\left\langle C\sum_{Q\in \mathcal{S}}\chi_Q(x) \fint_Q k_Q(x,y)\left(\sum_{\sigma\in C(m)}(-1)^{m-|\sigma|}  \vec{B}_{\sigma}(x)\vec{B}_{(\sigma^c)^t}(y)\right)\vec{f}(y)dy ,\vec{g}(x) \right\rangle dx\right\rvert
        \\ & \lesssim  \sum_{Q\in \mathcal{S}}\int_{Q}\fint_Q\left\lvert\left\langle  k_Q(x,y)\left(\sum_{\sigma\in C(m)}(-1)^{m-|\sigma|}  \vec{B}_{\sigma}(x)\vec{B}_{(\sigma^c)^t}(y)\right)\vec{f}(y) ,\vec{g}(x) \right\rangle\right\rvert dydx
        \\ & = \sum_{Q\in \mathcal{S}}\int_{Q}\fint_Q\left\lvert\left\langle  k_Q(x,y)V^{\frac{1}{p}}(x)\left(\sum_{\sigma\in C(m)}(-1)^{m-|\sigma|}  \vec{B}_{\sigma}(x)\vec{B}_{(\sigma^c)^t}(y)\right)U^{\frac{-1}{p}}(y)U^{\frac{1}{p}}(y)\vec{f}(y), \right.\right.
        \\ & \;\;\;\; \left. \left. \vphantom{\left(\sum_{\sigma\in C(m)}\right)} V^{\frac{-1}{p}}(x)\vec{g}(x) \right\rangle\right\rvert dydx
        \\ &  \leq \sum_{Q\in \mathcal{S}}\int_{Q}\fint_Q\left\lvert  k_Q(x,y)V^{\frac{1}{p}}(x)\left(\sum_{\sigma\in C(m)}(-1)^{m-|\sigma|}  \vec{B}_{\sigma}(x)\vec{B}_{(\sigma^c)^t}(y)\right)U^{\frac{-1}{p}}(y)U^{\frac{1}{p}}(y)\vec{f}(y)\right\rvert
        \\ & \;\;\;\; \times  \left\lvert V^{\frac{-1}{p}}(x)\vec{g}(x) \right\rvert dydx
        \\ & \leq \sum_{Q\in \mathcal{S}}\int_{Q}\fint_Q\|V^{\frac{1}{p}}(x)\left(\sum_{\sigma\in C(m)}(-1)^{m-|\sigma|}  \vec{B}_{\sigma}(x)\vec{B}_{(\sigma^c)^t}(y)\right)U^{\frac{-1}{p}}(y)\|\left\lvert  U^{\frac{1}{p}}(y)\vec{f}(y)\right\rvert dy
        \\ & \;\;\;\; \times \left\lvert V^{\frac{-1}{p}}(x)\vec{g}(x) \right\rvert dx.
    \end{align*}
    Applying now Hölder's Inequality twice, we obtain that
    \begin{align*}
        \left\lvert \int_{\mathbb{R}^n} \right. & \left\langle (T\otimes I_n)_{\vec{B}}\vec{f}(x),\left.\vec{g}(x) \vphantom{\int_{\mathbb{R}^n}} \right\rangle dx\right\rvert
        \\ &  \lesssim \sum_{Q\in \mathcal{S}}   \fint_{Q}\|\|V^{\frac{1}{p}}(x)\left(\sum_{\sigma\in C(m)}(-1)^{m-|\sigma|}  \vec{B}_{\sigma}(x)\vec{B}_{(\sigma^c)^t}(y)\right)U^{\frac{-1}{p}}(y)\|\|_{{C_y}(L)\left(Q\right)}\left\lvert V^{\frac{-1}{p}}(x)\vec{g}(x) \right\rvert dx
        \\ & \;\;\;\; \times \|U^{\frac{1}{p}}\vec{f}\|_{\overline{C}(L)\left(Q\right)}|Q|
                \\ & \lesssim \sum_{Q\in \mathcal{S}}   \|\|\|V^{\frac{1}{p}}(x)\left(\sum_{\sigma\in C(m)}(-1)^{m-|\sigma|}  \vec{B}_{\sigma}(x)\vec{B}_{(\sigma^c)^t}(y)\right)U^{\frac{-1}{p}}(y)\|\|_{{C_y}(L)\left(Q\right)}\|_{D_x(L)(Q)}\|V^{\frac{-1}{p}}\vec{g}\|_{\overline{D}(L)(Q)}
        \\ & \;\;\;\; \times \|U^{\frac{1}{p}}\vec{f}\|_{\overline{C}(L)\left(Q\right)}|Q|
                        \\ & \leq  \|\vec{B}\|_{\widetilde{BMO}_{V,U,m}^{C,D}}\sum_{Q\in \mathcal{S}}  \|V^{\frac{-1}{p}}\vec{g}\|_{\overline{D}(L)(Q)} \|U^{\frac{1}{p}}\vec{f}\|_{\overline{C}(L)\left(Q\right)}|Q|  .    
        \end{align*}
        A standard procedure now yields that 
                \begin{align*}
                    \left\lvert \int_{\mathbb{R}^n} \right.\left\langle (T\otimes I_n)_{\vec{B}}\vec{f}(x),\left.\vec{g}(x) \vphantom{\int_{\mathbb{R}^n}} \right\rangle dx\right\rvert
                         & \lesssim  \|\vec{B}\|_{\widetilde{BMO}_{V,U,m}^{C,D}}\sum_{Q\in \mathcal{S}}   \inf_Q M_{\overline{D}}(V^{\frac{-1}{p}}\vec{g}) \inf_Q M_{\overline{C}}(U^{\frac{1}{p}}\vec{f})|E_Q|
          \\ &  \leq \|\vec{B}\|_{\widetilde{BMO}_{V,U,m}^{C,D}}   \int_{\rn} M_{\overline{D}}(V^{\frac{-1}{p}}\vec{g})  M_{\overline{C}}(U^{\frac{1}{p}}\vec{f})
        \\ & \leq \|\vec{B}\|_{\widetilde{BMO}_{V,U,m}^{C,D}}   \|M_{\overline{D}}(V^{\frac{-1}{p}}\vec{g})\|_{p'}  \|M_{\overline{C}}(U^{\frac{1}{p}}\vec{f})\|_p,
    \end{align*}
    so, using that $\overline{C}\in B_p, \overline{D}\in B_{p'}$, we conclude that 
    \begin{align*}
       \left\lvert \int_{\mathbb{R}^n} \right.\left\langle (T\otimes I_n)_{\vec{B}}\vec{f}(x),\left.\vec{g}(x) \vphantom{\int_{\mathbb{R}^n}} \right\rangle dx\right\rvert
        \leq\|M_{\overline{D}}\|_{p'}\|M_{\overline{C}}\|_{p}\|\vec{B}\|_{\widetilde{BMO}_{V,U,m}^{C,D}}   \|\vec{g}\|_{L^{p'}(V^{\frac{-p'}{p}})}  \|\vec{f}\|_{L^p(U)}.
    \end{align*}
If right before applying Hölder's inequality we switch the order on integration, and apply Hölder's Inequality
with respect to $D$ and $x$, and then with respect to $C$ and $y$, an identical procedure yields the other part of the desired inequality.
\end{proof}

Notice that the norm in which the symbols are contained can be seen as a $\widetilde{BMO}$ norm, 
but with Orlicz norms instead of a $p-$th norm. This justifies the notation employed.

As it happens with the other estimates we have proven in this section, it is very easy to obtain a 
strong-type estimate from this result, employing duality and choosing the functions adequately.



\section{Applications of Theorem \ref{th:tocho}}\label{sect:apps}

In this last section, we would like to show that condition $(P_1)$ that 
appears in Theorem \ref{th:tocho} is reasonable and 
is satisfied by some classes of operators. In particular, we will provide an
example of operators for which $(P_1)$ is satisfied for $r=1$.
First, we introduce the following lemma. Recall that 

%

\begin{equation*}
\begin{split}
    {M}_Tf(x) = \sup_{Q\ni x} \sup_{y\in Q}|T(f \chi_{\rn\setminus 3Q})(y)|.
\end{split}
\end{equation*}


\begin{lemma}\label{lemma:auxiliar_apps}
    Let $T$ be a linear operator, defined for functions taking values in $\mathbb{R}^d$,
     such that, $T$ and $M_{T}$ are of weak type $(1,1)$.
     Then, for any cube $Q_0$, $\varepsilon\in(0,1)$, and $\vec{f}\in L^1(3Q_0,\mathbb{R}^{n})$, there exist 
     disjoint cubes $\left\{ Q_k'\right\}_k\subset \mathcal{D}(Q_0)$ such that 
     \begin{equation*}
     \begin{split}
        \sum_{k}|Q_k'|\leq \varepsilon |Q_0|
     \end{split}
     \end{equation*}
     and, if $\left\{ Q_k\right\}_k\subset \mathcal{D}(Q_0)$ is a family of disjoint cubes such that 
     $\bigcup_k Q_k'\subset \bigcup_k Q_k$, then
     \begin{equation}\label{aux10}
     \begin{split}
        \chi_{Q_0}(x){(T\otimes I_n)}\left(\chi_{3Q_0}\vec{f}\right)(x)-\sum_{k}\chi_{Q_k}(x){(T\otimes I_n)}\left(\chi_{3Q_k}\vec{f}\right)(x) \in \chi_{Q_0}(x) \frac{C_{d,n} C_{T}}{\varepsilon}\iinteg{\vec{f}}_{3Q_0},
     \end{split}
     \end{equation}
     is satisfied a.e. on $\mathbb{R}^d$,
    where $C_{T} = \|T\|_{L^1\rightarrow L^{1,\infty}} + \|M_{T}\|_{L^1\rightarrow L^{1,\infty}}$.
\end{lemma}


This lemma is a subtle extension of Lemmas 3.2 in \cite{cbdwczo} and  2.3.14 in \cite{hytonotas}, which in turn are adaptations 
of the proof of Theorem 3.1 in \cite{Opeiso}, for which 
\eqref{aux10} is valid for more families of cubes than just the one provided at the beginning of the statement.
The proof can be done very easily by extending the one provided for the aforementioned lemma in \cite{hytonotas},
which is based on the well-known John Ellipsoid Theorem.
We will use Lemma \ref{lemma:auxiliar_apps} to show the following result.

\begin{lemma}
    Let $m\in \mathbb{N}$, $\{\hat{e}_j\}$ be a basis in $\mathbb{R}^n$, and $\vec{T}=(T_1,...,T_n)$ be a linear operator, defined for functions taking values in $\mathbb{R}^d$,
     such that, for each $i,j$, $T_{i,j}(f):=T_i(\hat{e}_jf)$ and $M_{T_{i,j}}$ are of weak type $(1,1)$
     (which implies this is satisfied for any basis $\{e_j\}$ of $\mathbb{R}^n$).
     Then, for any cube $Q_0$, $\varepsilon\in(0,1)$, and $\vec{f}\in L^1(3Q_0,\mathbb{R}^{mn})$, there exist 
     disjoint cubes $\left\{ Q_k\right\}_k\subset \mathcal{D}(Q_0)$ such that 
     \begin{equation*}
     \begin{split}
        \sum_{k}|Q_k|\leq \varepsilon |Q_0|
     \end{split}
     \end{equation*}
     and there is a function $K_{Q_0}:Q_0\times 3Q_0\rightarrow \mathcal{M}_n(\mathbb{R})$, with 
     $\|K_{Q_0}\|_{L^\infty}\leq 1$, satisfying that, for the $mn\times mn$ matrix function
     \begin{equation*}
        \begin{split}
            \overline{K}_Q = 
\begin{pNiceArray}{cc|cc|cc|cc}
  \quad K_Q &&  {{0}} && \dots && {0}\\
  \hline
  {{0}} && K_Q&& \dots && {0}\\
  \hline
  \vdots && \vdots && \ddots     && \vdots\\
  \hline
  0 && 0 && \dots &&K_Q
\end{pNiceArray},
        \end{split}
        \end{equation*} 
        and
        \begin{equation*}
        \begin{split}
            \overline{T}(\vec{g}) = \begin{pmatrix}
            \vec{T}(\vec{g}^1)\\\vec{T}(\vec{g}^2)\\\vdots\\\vec{T}(\vec{g}^{m})
            \end{pmatrix},
        \end{split}
        \end{equation*}
        for an appropriate $\vec{g} = (g_1,g_2,...,g_{mn})$
        ($\vec{g}^k=(g_{(k-1)n+1},g_{(k-1)n+2},...,g_{(k-1)n+n})$),
     the identity
     \begin{equation*}
     \begin{split}
        \chi_{Q_0}(x)\overline{T}\left(\chi_{3Q_0}\vec{f}\right)(x)-\sum_{k}\chi_{Q_k}(x)\overline{T}\left(\chi_{3Q_k}\vec{f}\right)(x)= \chi_{Q_0}(x) \frac{C_{d,n} C_{T}}{\varepsilon}\integ{\overline{K}_{Q_0}(x,\cdot)\vec{f}}_{3Q_0}
     \end{split}
     \end{equation*}
     is verified a.e. on $\mathbb{R}^d$,
    where $C_{T} = \sum_{i,j=1}^n\|T_{i,j}\|_{L^1\rightarrow L^{1,\infty}} + \|M_{T_{i,j}}\|_{L^1\rightarrow L^{1,\infty}}$.
\end{lemma}

\begin{proof}
    First, let us consider,
if $\vec{f} = (f_1,...,f_{mn})$ (where the coordinates are in terms of $\{\hat{e}_j\}$), 
$\vec{f}^j = (f_j,f_{n+j},...,f_{n(m-1)+j})$ and
$\vec{T}^{i,j}=T_{i,j}\otimes I_m$, $i,j=1,2,...,n$.
The condition
    $\vec{f}\in L^1(3Q_0,\mathbb{R}^{mn})$ implies that 
    $\vec{f}^j\in L^1(3Q_0,\mathbb{R}^m)$ for all $j=1,2,...,n$, while,
    if we fix $i,j=1,...,n$, the conditions asked to
    $T_{i,j}$ imply that 
    the previous lemma
    is applicable to $\vec{T}^{i,j}$ and 
    $\vec{f^j}$.
     As a result,
    there exists a family of disjoint cubes 
    $\{Q_{l}^{i,j}\}_l\subset \mathcal{D}(Q_0)$ such that 
    \begin{equation*}
    \begin{split}
        \sum_{k}|Q_{l}^{i,j}|\leq \varepsilon |Q_0|
    \end{split}
    \end{equation*} 
    and satisfying that if $\{Q_l\}_l\subset \mathcal{D}(Q_0)$
    verifies that $\bigcup_l Q_l^{i,j}\subset \bigcup_l Q_l$, 
    then 
    \begin{equation}\label{aux9}
    \begin{split}
        \chi_{Q_0}\vec{T}^{i,j}\left(\chi_{3Q_0}{\vec{f}^j}\right)-\sum_{l}\chi_{Q_l}\vec{T}^{i,j}\left(\chi_{3Q_l}{\vec{f}^j}\right)
        \in \chi_{Q_0} \frac{C_{d,n} C_{T}}{\varepsilon}\iinteg{{\vec{f}^j}}_{3Q_0}
    \end{split}
    \end{equation}
    a.e. on $\mathbb{R}^d$ (we also applied that $C_{T_{i,j}}\leq C_T$). 

    If we take $\left\{ Q_l\right\}_l$ the family of all the 
    maximal cubes in $\left\{ Q_l^{i,j}\right\}_{i,j,l}$,
    then it is clear that 
    \begin{equation*}
    \begin{split}
        \sum_{l}|Q_{l}|\leq \sum_{i,j=1}^n\sum_{l}|Q_{l}^{i,j}|\leq n^2\varepsilon |Q_0|
    \end{split}
    \end{equation*}
    and that \eqref{aux9} is satisfied a.e. on $\mathbb{R}^d$ for all $i,j$.

    We may now apply Lemma \ref{lemma:cbdexprbonita} and ensure there exists a function 
    $K_Q^{i,j}:\mathbb{R}^d\times Q\rightarrow \mathbb{R}$ with $\|K_Q^{i,j}\|_{L^\infty}\leq1 $
    such that 
    \begin{equation*}
    \begin{split}
        \chi_{Q_0}(x)\vec{T}^{i,j}\left(\chi_{3Q_0}{\vec{f}^j}\right)(x)-\sum_{l}\chi_{Q_l}(x)\vec{T}^{i,j}\left(\chi_{3Q_l}{\vec{f}^j}\right)(x)=\chi_{Q_0}(x) \frac{C_{d,n} C_{T}}{\varepsilon}\integ{K_Q^{i,j}(x,\cdot){\vec{f}^j}}_{3Q_0},
    \end{split}
    \end{equation*}
    so we obtain that, for all $i,j=1,...,n$, and $k=0,...,m-1$,
    \begin{equation*}
        \begin{split}
            \chi_{Q_0}(x){T}_{i}\left(\chi_{3Q_0}\hat{e}_j{{f}_{nk+j}}\right)(x)
            -\sum_{l}\chi_{Q_k}(x){T}_{i}\left(\chi_{3Q_k}\hat{e}_j{{f}_{nk+j}}\right)(x)
            = \chi_{Q_0}(x) \frac{C_{d,n} C_{T}}{\varepsilon}\integ{K_Q^{i,j}(x,\cdot){f}_{nk+j}}_{3Q_0}
        \end{split}
        \end{equation*}
        a.e. on $\mathbb{R}^d$.
    If we define the $n\times n$ matrix function $K_Q$ that has $K_Q^{i,j}$ in position $(i,j)$, 
    and
    \begin{equation*}
    \begin{split}
        \overline{K}_Q = 
\begin{pNiceArray}{cc|cc|cc|cc}
  \quad K_Q &&  {{0}} && \dots && {0}\\
  \hline
  {{0}} && K_Q&& \dots && {0}\\
  \hline
  \vdots && \vdots && \ddots     && \vdots\\
  \hline
  0 && 0 && \dots &&K_Q
\end{pNiceArray}
    \end{split}
    \end{equation*}
    we obtain that
    \begin{equation*}
        \begin{split}
           \chi_{Q_0}(x)\overline{T}\left(\chi_{3Q_0}\vec{f}\right)(x)-\sum_{k}\chi_{Q_k}(x)\overline{T}\left(\chi_{3Q_k}\vec{f}\right)(x)= \chi_{Q_0}(x) \frac{C_{d,n} C_{T}}{\varepsilon}\integ{\overline{K}_{Q_0}(x,\cdot)\vec{f}}_{3Q_0},
        \end{split}
        \end{equation*}
    a.e. on $\mathbb{R}^d$,
    which concludes the proof if we take $\varepsilon=\frac{\varepsilon}{n^2}$ and $C_{d,n} = n^2 C_{d,n}$.
\end{proof}

With this lemma in hand, it is very simple to show the desired result.

\begin{theorem}
    Let $m\in \mathbb{N}$ and $\vec{T}=(T_1,...,T_n)$ be a linear operator, defined for functions taking values in $\mathbb{R}^d$,
     such that, for each $i,j$, $T_{i,j}(f):=T_i(\hat{e}_jf)$ and $M_{T_{i,j}}$ are of weak type $(1,1)$.
     Then, given $\vec{f}\in L_c^1(\mathbb{R}^d,\mathbb{R}^{mn})$ and $\varepsilon\in(0,1)$, there exist a 
     $(1-\varepsilon)-$sparse family $\mathcal{Q}$ of dyadic cubes
     and functions $K_Q:Q\times 3Q\rightarrow \mathcal{M}_n(\mathbb{R})$ with $\|K_Q\|_{L^\infty}\leq 1$
     satisfying that,
     for the $mn\times mn$ matrix function
     \begin{equation*}
        \begin{split}
            \overline{K}_Q = 
\begin{pNiceArray}{cc|cc|cc|cc}
  \quad K_Q &&  {{0}} && \dots && {0}\\
  \hline
  {{0}} && K_Q&& \dots && {0}\\
  \hline
  \vdots && \vdots && \ddots     && \vdots\\
  \hline
  0 && 0 && \dots &&K_Q
\end{pNiceArray},
        \end{split}
        \end{equation*} 
        and 
        \begin{equation*}
            \begin{split}
                \overline{T} = \left.\begin{pmatrix}
                \vec{T}\\\vec{T}\\\vdots\\\vec{T}
                \end{pmatrix}\right\}m\text{ times},
            \end{split}
            \end{equation*} 
            the identity
     \begin{equation*}
     \begin{split}
        \overline{T}\vec{f}(x) = \frac{C_{d,n}C_T}{\varepsilon}\sum_{Q\in \mathcal{Q}}\chi_Q(x)\integ{\overline{K}_Q(x,\cdot)\vec{f}}_{3Q}
     \end{split}
     \end{equation*}
     is satisfied a.e. on $\mathbb{R}^d$.

     In addition, it is possible to find $3^d$ dyadic lattices $\mathcal{D}_j$, $3^d$ $3^{-d}(1-\varepsilon)-$sparse
      families $\mathcal{Q}_j\subset D_j$ and functions $K_Q:Q\times Q\rightarrow \mathcal{M}_n(\mathbb{R})$
      with  $\|K_Q\|_{L^\infty}\leq 1$ such that, for 
      \begin{equation*}
        \begin{split}
            \overline{K}_Q = 
\begin{pNiceArray}{cc|cc|cc|cc}
  \quad K_Q &&  {{0}} && \dots && {0}\\
  \hline
  {{0}} && K_Q&& \dots && {0}\\
  \hline
  \vdots && \vdots && \ddots     && \vdots\\
  \hline
  0 && 0 && \dots &&K_Q
\end{pNiceArray},
        \end{split}
        \end{equation*} 
    we have the identity
      \begin{equation*}
        \begin{split}
           \overline{T}\vec{f}(x) = \frac{C_{d,n}C_T}{\varepsilon}\sum_{j=1}^{3^d}\sum_{Q\in \mathcal{Q}_j}\chi_Q(x)\integ{\overline{K}_Q(x,\cdot)\vec{f}}_{Q}
        \end{split}
        \end{equation*}
      a.e. on $\mathbb{R}^d$.
\end{theorem}


The proof for this result follows the typical structure of a sparse
 domination proof that makes use of Lerner's maximal operator $M_T$. 
The original proof can be found in \cite{Opeiso} (Theorem 3.1), but adapted versions for
operators of the form $T\otimes I_n$ are present in \cite{cbdwczo} (Theorem 3.4) and \cite{hytonotas}(Theorem 2.3.2 and Corollary 2.3.18).



\bibliographystyle{amsplain}
\bibliography{ref}

\end{document}